\documentclass[preprint,11pt]{imsart}
\usepackage{natbib}
\usepackage{apalike}

\RequirePackage[T1]{fontenc}
\usepackage{color}
\usepackage[latin1]{inputenc}
\usepackage[T1]{fontenc}
\RequirePackage{amsfonts,amsthm,amsmath}
\RequirePackage[colorlinks,citecolor=blue,urlcolor=blue]{hyperref}
\usepackage{a4wide}
\usepackage{here}
\usepackage{float}
\allowdisplaybreaks
\startlocaldefs
\theoremstyle{plain}

\endlocaldefs
\newtheorem{theo}{Theorem}[section]

\newtheorem{lemme}[theo]{Lemma}
\newtheorem{cor}[theo]{Corollary}

\newtheorem{prop}[theo]{Proposition}

\newtheorem{rem}[theo]{Remark}

\hypersetup{pdfstartview=FitH}

\def\p{{\mathbb P}}
\def\P{{\mathbb P}}
\def\e{{\mathbb E}}
\usepackage{bbm}

\newcommand{\bc}{\boldsymbol{c}}

\newcommand{\1}{\ensuremath{\mathbbm{1}}}
\def\var {\mathop{\rm Var}\nolimits}    
\def\cov{\mathop{\rm Cov}\nolimits}    

\newcommand{\nbR}{\mathbb{R}}
\newcommand{\nbN}{\mathbb{N}}

\newcommand{\nbC}{\mathbb{C}}

\newcommand{\nbP}{\mathbb{P}}
\newcommand{\nbZ}{\mathbb{Z}}

\newcommand{\nbE}{\mathbb{E}}

  \newcommand{\zak}{\nobreak \ifvmode \relax \else
     \ifdim\lastskip<1.5em \hskip-\lastskip
     \hskip1.5em plus0em minus0.5em \fi \nobreak
     $\Box$\fi\\}

\def\PP{{\mathbb P}}

\def\EE{{\mathbb E}}

\newcommand{\Ss}{{\cal S}}

\usepackage{graphicx}
\begin{document}
\begin{frontmatter}

\title{Estimation of weak ARMA models with regime changes\\
}
\runtitle{Estimating weak ARMARC models}
%
\begin{aug}
\author{\fnms{\Large{Yacouba}}
\snm{\Large{Boubacar Ma\"{\i}nassara}}
\ead[label=e2]{yacouba.boubacar$\_$mainassara@univ-fcomte.fr}} \and \
\author{\fnms{\Large{Landy}}
\snm{\Large{Rabehasaina}}
\ead[label=e1]{lrabehas@univ-fcomte.fr}}
\runauthor{Y. Boubacar Ma\"{\i}nassara and L. Rabehasaina}
\affiliation{Universit\'e de Bourgogne Franche-Comt\'e}

\address{\hspace*{0cm}\\
Universit\'e Bourgogne Franche-Comt\'e, \\
Laboratoire de math\'{e}matiques de Besan\c{c}on, \\ UMR CNRS 6623, \\
16 route de Gray, \\ 25030 Besan\c{c}on, France.\\[0.2cm]
\printead{e2,e1}}
\end{aug}
\vspace{0.5cm}
\begin{abstract}
In this paper we derive the asymptotic properties of the
least squares estimator (LSE) of autoregressive
moving-average (ARMA) models with regime changes under the assumption that the errors are uncorrelated but not necessarily independent.
Relaxing the independence assumption considerably extends the range
of application of the class of ARMA models with regime changes. Conditions are given
for the consistency and asymptotic normality of the LSE. A particular attention is given to the estimation of the asymptotic covariance matrix,
which may be very different from that obtained in the standard framework.
The theoretical results are illustrated by means of Monte Carlo experiments.
\end{abstract}
\end{frontmatter}
%
\section{Introduction}
Since the works of \cite{Hamilton1988, Hamilton1989} and \cite{NQ1982}, the time series models with time-varying coefficients have become increasingly popular.
In statistical applications, a large part of the literature is devoted to the non-stationary  autoregressive
moving-average (ARMA) models with time-varying parameters (see \cite{AM1998, AM2006, BF2003, Dahlhaus1997}), see also the class of ARMA models with periodic coefficients (for instance
\cite{AM1997, BL2001}). But the most popular  class deals with the treatment of regime shifts and non-linear modeling strategies.  For instance, a Markov-switching model is a non-linear specification in which different states of the world affect the evolution of a time series (see, for examples, \cite{FR1997, Hamilton1990, HS1994autoregressive}). The asymptotic properties of Markov-switching ARMA models are well known in the literature (see, for instance, \cite{BMR1999, FR1998, FZ2001, FZ2002, KK2015} or \cite{Hamilton1994}).

The fact that changes in regimes may be very important for the evolution of interest rates has been emphasized in a number of recent studies.
Our attention here is focused on the class of ARMA models with regime changes (ARMARC for short); for instance, ARMA models with recurrent but non necessarily periodic changes in regime.
We consider a time series $(X_t)_{t\in\mathbb{Z}}$ exhibiting changes in regime at known dates and we
suppose that we have finite regimes. Contrarily to the famous Markov-switching approach, we  assume that the realization of the regimes is observed. Such a situation may be realistic, and would correspond e.g. to time series with periods of harsh and mild weather which are observed in practice. This model could also be applied to economic time series whose behaviour depends on worked days and public holidays, which are known in advance. Another motivating example would be financial times series, where regimes corresponding to typical known major events leading to high and quiet (low) volatility subperiods are observed, see e.g. Figure 1.2 p.7 in \cite{FZ2010_book} where the high volatility clusters corresponds to largely famous events such as September 11th 2001 or the 2008 financial crisis. Another example can be found for instance in \cite{FG2004jtsa}.

For such models, \cite{FG2004, FG2004jtsa} gave general conditions ensuring consistency and asymptotic
normality of least squares (LS) and quasi-generalized least-squares (QGLS) estimators under the assumption that the innovation processes is independent.
This independence assumption is often considered too restrictive by
practitioners. Relaxing the independence assumption considerably
extends the range of applications of the ARMARC models, and allows to cover
 general nonlinear processes. Indeed such nonlinearities
may arise for instance when the error process follows an autoregressive conditional
heteroscedasticity (ARCH) introduced by Engle \cite{E1982} and extended to the
generalized ARCH (GARCH) by \cite{B1986}, all-pass (see \cite{ADB2006}) or other models displaying
a second order dependence (see \cite{AF2009}). Other situations where the errors are
dependent can be found in \cite{FZ2005}, see also \cite{rt1996}.
This paper is devoted to the problem of estimating ARMARC representations under the assumption that the errors are uncorrelated but not necessarily independent.
These are called weak ARMARC models in contrast to the strong ARMARC models above-cited, in which the error terms are supposed to be independent
and identically distributed (iid).
Thus, the main goal of our paper is to complete the above-mentioned
results concerning the statistical analysis of ARMARC models, by considering the estimation problem under general error terms.
We establish the asymptotic distribution of the LS estimator of weak ARMARC models, under strongly mixing assumptions.

The paper is organized as follows. Section \ref{modelassumption} presents the
ARMARC models that we consider here. In Section \ref{result}, we established the strict stationarity condition and it is shown  that the LS estimator (LSE) is asymptotically normally distributed
when linear innovation process $(\epsilon_t)$ satisfies mild mixing assumptions. The
asymptotic covariance of the LSE may be very different in the weak and strong cases.
Particular attention is given to the estimation of
this covariance matrix. Modified version of the Wald test is proposed for testing linear restrictions on
the parameters. In Section \ref{examples}, we present two examples of weak ARMARC$(1,0)$ models with iid
and correlated realization of the regimes.  Numerical experiments are
presented in Section \ref{numericalIllustrations}. The proofs of the main results are collected in the appendix.

\section{Model and assumptions}\label{modelassumption}
Let $(\Delta_t)_{t\in \mathbb{Z}}$ be a stationary ergodic observed process with values in a finite set ${\cal S}$ of size $\mbox{Card} ({\cal S})=K$. We consider the ARMARC$(p,q)$ process $(X_t)_{t\in \mathbb{Z}}$ defined by
\begin{equation}\label{def_arma_pq0}
X_t-\sum_{i=1}^p a_i^0(\Delta_t)X_{t-i}=\epsilon_t-\sum_{j=1}^q
b_j^0(\Delta_t)\epsilon_{t-j}
\end{equation}
where the linear innovation process $\epsilon:=(\epsilon_t)_{t\in \mathbb Z}$ is assumed to be a stationary sequence satisfies $\EE (\epsilon_t)=0,\quad \EE(\epsilon_t \epsilon_{t'})=\sigma^2 \1_{[t=t']}$.  Under the above assumptions, the process  $\epsilon$ is called a weak white noise.

An important example of a weak white noise is the GARCH model (see \cite{FZ2010_book}). In the modeling of financial time series the GARCH assumption on the errors is often used
to capture the  conditional  heteroscedasticity. However, the multiplicative noise structure of this GARCH model is often too restrictive in practical situations. This is one motivation of this paper, which considers an even more general weak noise, where the error is subject to unknown conditional heteroscedasticity.

This representation is said to be a weak ARMARC$(p,q)$ representation under the assumption that $\epsilon$ is a weak white noise.
For the statistical inference of ARMA models, the weak white noise assumption is often replaced by the strong white noise
assumption,  {\em i.e.} the assumption that $\epsilon$
is an iid sequence of random variables with mean 0 and common variance.
Obviously the strong white noise assumption is more restrictive than the weak white noise assumption, because independence entails uncorrelatedness. Consequently  weak ARMARC representation is more general than the strong one.

The unknown parameter of interest denoted $\theta_0:=(a_i^0(s), b_j^0(s),\ i=1,\dots,p,\ j=1,\dots,q,\ s\in {\cal S})$ lies in a compact set of the form
$$
\Theta \subset \left\{(a_i(s), b_j(s),\ i=1,\dots,p,\ j=1,\dots,q,\ s\in {\cal S}) \in \mathbb{R}^{(p+q)\times K}\right\},
$$
with non empty interior, within which we suppose that $\theta_0$ lies.
The parameter $\sigma^2$ is considered as a nuisance parameter.
In order to estimate $\theta_0$, we thus have at our disposal the observations $(X_t,\Delta_t)$, $t=1,\dots,n$, from which we aim to build a strongly consistent and asymptotically normal estimator $\hat{\theta}_n$.  We now introduce, the strong mixing coefficients $(\alpha_Z (h))_{h\in\mathbb{Z}}$ of a stationary process $(Z_t)_{t\in\mathbb{Z}}$  defined by
\begin{equation}
\alpha_Z (h):=\sup_{A\in{\cal F}_{-\infty}^{t},\ B\in {\cal F}_{t+h}^{\infty}}  \left| \mathbb{P}(A\cap B)-\mathbb{P}(A)\cdot\mathbb{P}(B)\right|,
\label{def_mixing}
\end{equation}
measuring the temporal dependence of the process and
where ${\cal F}_{-\infty}^{t}$, and ${\cal F}_{t+h}^{\infty}$ be the $\sigma$-fields generated by $\{Z_u,\ u\le t\}$ and $\{Z_u,\ u\ge t+h\}$, respectively.
We will make an integrability assumption on the moment of the noise and a summability condition on
the strong mixing coefficients $\left(\alpha_Z (h)\right)_{h\geq0}$. Let us suppose the following assumptions. 
\begin{eqnarray*}
&\hspace{-1cm}{\bf (A1)} \text{ The processes }(\epsilon_t)_{t\in \mathbb{Z}}\mbox{ and }(\Delta_t)_{t\in \mathbb{Z}}\mbox{ are ergodic sequences, strictly  stationary,}\\
&\mbox{independent from each other}. 
\\
&\hspace{-.2cm}{\bf (A2)} \text{ For some } \nu>0, \text{ the processes }
 (\epsilon_t)_{t\in \mathbb{Z}}\mbox{ and } (\Delta_t)_{t\in \mathbb{Z}}\mbox{ satisfy } \sum_{h=0}^\infty \alpha_\epsilon (h)^{\frac{\nu}{\nu+2}}<+\infty\\&\mbox{ and }\sum_{h=0}^\infty \alpha_\Delta (h)^{\frac{\nu}{\nu+2}}<+\infty .&\\
&\hspace{-4.8cm}{\bf (A3)}  \text{ The process } (\epsilon_t)_{t\in \mathbb{Z}} \mbox{ also satisfies } \EE[|\epsilon_t|^{2\nu+4}]<+\infty .
&\\
&\hspace{-4.5cm}{\bf (A4)}  \text{  We have } \theta_0\in\stackrel{\circ}{{\Theta}}, \text{ where } \stackrel{\circ}{{\Theta}} \text{ denotes the interior of }{\Theta} .
\end{eqnarray*}
Note that the strong white noise assumption entails the ergodicity condition for $(\epsilon_t)_{t\in \mathbb{Z}}$. This is not the case if we impose the weak white noise assumption only, hence the assumption  {\bf (A1)}. Likewise, the ergodicity condition on $(\Delta_t)_{t\in \mathbb{Z}}$ is imposed in that assumption. For example, if $(\Delta_t)_{t\in \mathbb{Z}}$ is a finite Markov chain, then a necessary and sufficient condition for ergodicity is that it is irreducible, which ensures its positive recurrence (see Theorem 1.10.2 p.53 in \cite{Norris_book_Markov}), see for instance the example in Section \ref{examples}.

We introduce the following notations so as to emphasize dependence of unknown parameter $\theta_0$ in (\ref{def_arma_pq0}). For all $\theta= (a_i(s), b_j(s),\ i=1,\dots,p,\ j=1,\dots,q,\ s\in {\cal S})\in \Theta$, we let $\underline{a_i}:=(a_i(s),s\in \Ss)$, $i=1,\dots,p$ and $\underline{b_j}:=(b_j(s),s\in \Ss)$, $j=1,\dots,q$. Let ${\bf e}(s)$ be the row vector of size $1\times K $ such that the $i$th component is $\1_{[s=i]}$. Then we notice that $\forall t\in \nbZ$
$$
a_i(\Delta_t)=<{\bf e}(\Delta_t), \underline{a_i}>:= g_i^a(\Delta_t,\theta),\quad b_j(\Delta_t)=<{\bf e}(\Delta_t), \underline{a_j}>:= g_j^b(\Delta_t,\theta),\quad i=1,\dots,p,\ j=1,\dots,q,
$$
where $<\cdot ,\cdot >$ denotes the scalar product between vectors of appropriate dimension. Thus (\ref{def_arma_pq0}) reads
\begin{equation}\label{def_arma_pq}
X_t-\sum_{i=1}^p g_i^a(\Delta_t,\theta_0)
X_{t-i}=\epsilon_t-\sum_{j=1}^q
g_j^b(\Delta_t,\theta_0)\epsilon_{t-j}.
\end{equation}
Let us furthermore note that for all $i$, $j$ and $s$, $g_i^a(s,\theta)$ and $g_j^b(s,\theta)$ are {\it linear} in $\theta$. We thus introduce the following companion matrices
$$
A(s):=\left(
\begin{array}{cccc}
    g_1^a(s,\theta_0)& \cdots & \cdots & g_p^a(s,\theta_0)\\
    &&& 0 \\
    & I_{p-1}& & \vdots\\
    &&& 0
\end{array}
\right),\quad
B(s,\theta):=\left(
\begin{array}{cccc}
    g_1^b(s,\theta)& \cdots & \cdots & g_q^b(s,\theta)\\
    &&& 0 \\
    & I_{q-1}& & \vdots\\
    &&& 0
\end{array}
\right)
$$
for all $s\in\Ss$, $\theta\in\Theta$. A remark that will prove useful later on is that $\theta \mapsto B(s,\theta)$ is, for all $s\in \Ss$, an {\it affine} function.

We next introduce the residuals corresponding to parameter $\theta\in\Theta$ as the stationary process $(\epsilon_t(\theta))_{t\in \nbZ}$ satisfying
\begin{equation}\label{def_inov_theta}
\epsilon_t(\theta)-\sum_{j=1}^q g_j^b(\Delta_t,\theta)\epsilon_{t-j}(\theta) = X_t-\sum_{i=1}^p g_i^a(\Delta_t,\theta) X_{t-i},\quad \forall t\in \nbZ .
\end{equation}
This process is unique in $L^2$, as explained in Proposition \ref{prop_stationary}. In particular, we have $(\epsilon_t(\theta_0))_{t\in \nbZ}= (\epsilon_t)_{t\in \nbZ}$, the initial white noise.
We next define the approximating residuals as the process $(e_t(\theta))_{t\in \nbZ}$ verifying
\begin{align}
\label{e_t}
e_t(\theta)-\sum_{j=1}^q g_j^b(\Delta_t,\theta)e_{t-j}(\theta) = \tilde{X}_t-\sum_{i=1}^p g_i^a(\Delta_t,\theta) \tilde{X}_{t-i},\quad \forall t\in \nbZ ,
\end{align}
where values corresponding to negative indices are set to zero, i.e. the processes $(e_t(\theta))_{t\in \nbZ}$ and $(\tilde{X})_{t\in \nbZ}$ verify
$$
\begin{array}{rcl}
e_t(\theta) & =& 0,\quad t\le 0,\\
\tilde{X}_t &=& X_t \1_{[t\ge 1]},\quad \forall t\in \nbZ .
\end{array}
$$
The basic idea behind definition of $(e_t(\theta))_{t\in \nbZ}$ is that, given a realization $X_1,X_2,\dots{},X_n$ of length $n$, $\epsilon_t(\theta)$ is
approximated, for $0<t\leq n$, by $e_t(\theta)$.
Next, we define the cost function
\begin{align}
Q_n(\theta)& =\frac{1}{2n}\sum_{t=1}^n e_t^2(\theta). \label{qn}
\end{align}
Finally, we let for all $n\in\nbN$ the random variable $\hat{\theta}_n$ the least squared estimator that satisfies, almost surely,
\begin{align}
Q_n(\hat{\theta}_n)& = \min_{\theta\in \Theta} Q_n(\theta),\label{LSE_Qn}
\end{align}
We finish this section by giving some notations. In the following, $||.||$ will denote the norm of matrices or vectors of appropriate size, depending on the context, whereas $||.||_p$ will denote the $L^p$ norm defined by $||X||_p=\left[ \EE(|X|^p)\right]^{1/p}$ for all random variable $X$ admitting a $p-$th order moment, $p\ge 1$. For all matrix $M$, $M'$ will denote its transpose. For all three times differentiable function $f:\Theta\longrightarrow \nbR$, we will let $\nabla f(\theta)=\left( \frac{\partial}{\partial \theta_k}f(\theta)\right)_{k=1,...,(p+q)K}$,  $\nabla^2 f(\theta)=\left( \frac{\partial^2}{\partial \theta_i \partial \theta_j}f(\theta)\right)_{i,j=1,...,(p+q)K}$ and $\nabla^3 f(\theta)=\left( \frac{\partial^3}{\partial \theta_\ell \partial \theta_i \partial \theta_j}f(\theta)\right)_{\ell, i,j=1,...,(p+q)K}$ respectively the first, second and third order derivatives with respect to the variable $\theta$.
\section{Case of general correlated process $(\Delta_t)_{t\in \nbZ}$}\label{result}
In this section, we display our main results.
\subsection{Weak stationarity}\label{subsec_stationarity}
A first step consists in giving sufficient conditions such that the processes $(X_t)_{t\in \nbZ}$ and $(\epsilon_t(\theta))_{t\in \nbZ}$ defined in (\ref{def_arma_pq0}) and (\ref{def_inov_theta}) are strictly stationary and admits moments of sufficiently high order so as to obtain consistency and asymptotic normality results. This approach is standard, see e.g. \cite[Theorem 1 and Section 3]{FZ2001} and \cite[Theorems 2.1 and 4.1]{Stelzer09}.
Let $||.||$ be any norm on the set of matrices, and let us introduce the following notations
\begin{eqnarray*}
w_1&:=& (1,0,\dots,0)\in \nbR^{p+q} ,\\
w_{p+1}&:=& (w_{p+1,i})_{i=1,\dots,p+q},\quad w_{p+1,1}=1,\quad w_{p+1,i}=\1_{[i=p+1]},\ i=2,\dots,p+q,\\
M&:=& (m_{ij})_{i,j=1,\dots,p+q},\quad m_{i,j}=\1_{[i=q+1, j=1 \mbox{ or }i=1, j=1]},\\
\Phi(s,\theta)&:=& \left(
\begin{array}{c| ccc}
&&&\\
 & g_1^a (s,\theta) & \cdots & g_p^a (s,\theta)\\
B(s,\theta) & & 0 & \\
&&&\\
\hline
&&&\\
 & 0 & \cdots & 0\\
 0 & & I_{p-1} & \vdots\\
 & & &  0
\end{array}
\right),\quad s\in \Ss,\ \theta\in\Theta ,
\\
\Psi(s)&:=&\left(
\begin{array}{c| ccc}
&&&\\
 & g_1^b (s,\theta_0) & \cdots & g_q^b (s,\theta_0)\\
A(s) & & 0 & \\
&&&\\
\hline
&&&\\
 & 0 & \cdots & 0\\
 0 & & I_{q-1} & \vdots\\
 & & &  0
\end{array}
\right),\quad s\in \Ss .
\end{eqnarray*}
Let us note that the matrices $\Phi(s,\theta)$ and $\Psi(s)$ are, like $B(s,\theta)$ and $A(s)$, reminiscent of companion matrices. As for $B(s,\theta)$, we also notice in particular that $\theta \mapsto \Phi(s,\theta)$ is an affine function for all $s\in \Ss$. We have the following result.
\begin{prop}\label{prop_stationary}
Let us suppose that
$$
{\bf (A5a)}\quad \limsup_{t\to\infty}\frac{1}{t}\ln  \nbE\left( \sup_{\theta\in \Theta} \left| \left| \prod_{i=1}^t \Phi(\Delta_i,\theta)\right|\right|^{8}\right)<0,\quad \limsup_{t\to\infty}\frac{1}{t}\ln  \nbE\left( \left|\left|\prod_{i=1}^t \Psi(\Delta_i)\right|\right|^{8}\right)<0 ,
$$
then for all $t\in\nbZ$ and $\theta\in\Theta$, the unique stationary solution to (\ref{def_inov_theta}) is given by
\begin{align}
\label{dse_eps_theta}\epsilon_t(\theta)&= \sum_{i=0}^\infty c_i(\theta,\Delta_{t},\dots,\Delta_{t-i+1})\epsilon_{t-i},\quad \mbox{where} \\
c_i(\theta,\Delta_{t},\dots,\Delta_{t-i+1})&= \sum_{k=0}^i w_1\prod_{j=0}^{k-1}\Phi(\Delta_{t-j},\theta)M \prod_{j'=k}^{i-1}\Psi(\Delta_{t-j'})w_{p+1}' ,\label{expression_eps_theta}
\end{align}
with the usual convention $\prod_i^j=1$ if $i>j$.
Furthermore, for each $t\in\nbZ$, $(c_i(\theta,\Delta_{t},\dots,\Delta_{t-i+1}))_{i\in\nbN}$ is the unique sequence in the set of sequences of random variables
$$
{\cal H}:=\left\{ (d_i)_{i\in\nbN}\mbox{ independent from } (\epsilon_t)_{t\in\nbZ}\mbox{ s.t. } \nbE\left( \sum_{i=0}^\infty d_i^2\right) <+\infty \right\}
$$
satisfying the decomposition \eqref{dse_eps_theta}.
\end{prop}
The uniqueness property in this proposition can be seen as an identifiability property. Such a property is guaranteed in a similar context by Assumption {\bf A6} page 56 in \cite{Gautier04} (see also \cite{FG2003}) in the case of strong ARMA processes modulated by a Markov chain.
Note also that the decomposition \eqref{dse_eps_theta} is a slight generalization of the Wold decomposition of stationary processes which are squared integrable, see Theorem 5.7.1 p.187 of \cite{broc-d}. Remark that the stability condition {\bf (A5a)} is reminiscent of the one in \cite[Theorem 1]{FZ2001} and \cite[Theorem 2.1]{Stelzer09} (see also \cite{Brandt86}); it is however stronger as we need integrability conditions for the process $(\epsilon_t(\theta))_{t\in \nbZ}$ (as well as on its derivatives), uniformly on $\theta \in \Theta$. More precisely, we note that the right inequality condition in {\bf (A5a)} is equivalent to \cite[Remark 4.1 (a)]{Stelzer09}.
\begin{cor}\label{coro_decompo_e_t}
The process $(e_t(\theta))_{t\in \nbZ}$ defined by \eqref{e_t} has the following decomposition
\begin{eqnarray}
e_t(\theta)&=& \sum_{i=0}^\infty c_i^e(t,\theta,\Delta_{t},\dots,\Delta_{t-i+1})\epsilon_{t-i},\quad t\ge p+1,\mbox{ where} \label{dse_e_t_theta}\\
c_i^e(t,\theta,\Delta_{t},\dots,\Delta_{t-i+1})&=& \sum_{k=0}^{\min(t-1,i)} w_1\prod_{j=0}^{k-1}\Phi(\Delta_{t-j},\theta)M \prod_{j'=k}^{i-1}\Psi(\Delta_{t-j'})w_{p+1}' ,\label{expression_e_t_theta}
\end{eqnarray}
where the matrix $M$ and vectors $w_1$, $w_{p+1}$, are defined at the beginning of the section.
\end{cor}
\begin{lemme}\label{c_i^2_expo_decrease}
The random coefficients $c_i(\theta,\Delta_{t},\dots,\Delta_{t-i+1})$, $i\in\nbZ$, $t\in\nbZ$,  verify the following properties:
\begin{itemize}
\item $\theta\mapsto c_i(\theta,\Delta_{t},\dots,\Delta_{t-i+1})$, $\theta\mapsto \nabla [c_i(\theta,\Delta_{t},\dots,\Delta_{t-i+1})]^2$ and $\theta\mapsto \nabla^2 [c_i(\theta,\Delta_{t},\dots,\Delta_{t-i+1})]^2$  are a.s. polynomial functions,
\item Let us assume, instead of {\bf (A5a)}, that the stronger assumption
$$
{\bf (A5b)}\limsup_{t\to\infty}\frac{1}{t}\ln  \nbE\left( \sup_{\theta\in \Theta} \left| \left| \prod_{i=1}^t \Phi(\Delta_i,\theta)\right|\right|^{4\nu+8}\right)<0,\quad \limsup_{t\to\infty}\frac{1}{t}\ln  \nbE\left( \left|\left|\prod_{i=1}^t \Psi(\Delta_i)\right|\right|^{4\nu+8}\right)<0
$$
holds. Then we have
\begin{equation}\label{coeff_c_i_expo_decrease}
\begin{array}{rcl}
\limsup_{i\to\infty}\frac{1}{i}\ln  \nbE\left(\sup_{\theta\in \Theta}[ c_i(\theta,\Delta_{i},\dots,\Delta_{1})]^{2\nu+4}\right)&<&0 ,\\
\limsup_{i\to\infty}\frac{1}{i}\ln  \nbE\left(\sup_{\theta\in \Theta}\left| \left| \nabla^j [c_i(\theta,\Delta_{i},\dots,\Delta_{1})]\right|\right|^{2\nu+4}\right)&<& 0 ,\quad j=2,3.
\end{array}
\end{equation}
\end{itemize}
Furthemore, the coefficients $c_i^e(t,\theta,\Delta_{t-1},\dots,\Delta_{t-i})$, $i\in\nbZ$, $t\ge 0$, satisfy
\begin{equation}\label{coeff_c_i_e_expo_decrease}
\begin{array}{rcl}
\limsup_{i\to\infty}\frac{1}{i}\ln  \sup_{t\ge 0}\nbE\left(\sup_{\theta\in \Theta}[ c_i^e(t,\theta,\Delta_{t},\dots,\Delta_{t-i+1})]^{2\nu+4}\right)&<&0 ,\\
\limsup_{i\to\infty}\frac{1}{i}\ln  \sup_{t\ge 0}\nbE\left(\sup_{\theta\in \Theta}\left| \left| \nabla^j [c_i^e(t,\theta,\Delta_{t},\dots,\Delta_{t-i+1})]\right|\right|^{2\nu+4}\right)&<& 0 ,\quad j=2,3.
\end{array}
\end{equation}
\end{lemme}
Note that one of the differences with \cite{Gautier04,FG2003} (apart for the obvious one where the noise is weak here) is that {\bf (A5b)} leads to the exponential decrease \eqref{coeff_c_i_expo_decrease} for the coefficient $[c_i(\theta,\Delta_{i},\dots,\Delta_{1})]^{2\nu+4}$ (uniformly in $\theta$) as well as its derivatives. This is to be compared with Condition {\bf A8} page 56 of \cite{Gautier04} (see also \cite{FG2003}), where the exponent is $4$ instead of $2\nu+4$. This $\nu>0$ is what makes the difference between weak and strong noise, as this is the parameter that measures the dependence among the random variables in the (non iid) sequence $(\epsilon_t)$. Also note that \eqref{coeff_c_i_expo_decrease} and \eqref{coeff_c_i_e_expo_decrease} are akin to Conditions $({\bf A2})$ and $({\bf A8})$ in \cite{FG2004}.

\subsection{Preliminary results}
We define the cost function
\begin{align}
O_n(\theta)&=\frac{1}{2n}\sum_{t=1}^n \epsilon_t^2(\theta).\label{on}
\end{align}
Similarly to $\hat{\theta}_n$, let us introduce  $\check{\theta}_n$ the least squared estimators corresponding to the cost function $O_n(\theta)$:
\begin{equation}
O_n(\check{\theta}_n) = \min_{\theta\in \Theta} O_n(\theta).\label{LSE}
\end{equation}
The following results are necessary in order to prove the asymptotic properties for the estimators $\hat{\theta}_n$ and $\check{\theta}_n$ defined in \eqref{LSE} and \eqref{LSE_Qn}. We first justify that $e_t(\theta)$ asymptotically behaves as $\epsilon_t(\theta)$ as $t\to\infty$ for all $\theta$ as follows:
\begin{lemme}\label{lemme_prelim}
Let us suppose that {\bf (A1)} and that stationarity condition {\bf (A5a)} hold. Sequences $(\epsilon_t(\theta))_{t\in\nbZ}$ and $(e_t(\theta))_{t\in\nbZ}$ satisfy
\begin{enumerate}
\item $\left|\left| \sup_{\theta\in \Theta} |\epsilon_0(\theta)|\right| \right|_4<+\infty$ and $\sup_{t\ge 0}\left|\left| \sup_{\theta\in \Theta} |e_t(\theta)|\right| \right|_4<+\infty$,
\item $\left|\left| \sup_{\theta\in \Theta} |\epsilon_t(\theta)-e_t(\theta)|\right| \right|_2$ tends to $0$ exponentially fast as $t\to \infty$,
\item For all $\alpha>0$, $t^\alpha \sup_{\theta\in \Theta} |\epsilon_t(\theta)-e_t(\theta)|\longrightarrow 0$ a.s. as $t\to\infty$,
\item For all $j=1, 2,3$, $\left|\left| \sup_{\theta\in \Theta} ||\nabla^j\epsilon_0(\theta)||\right| \right|_4<+\infty$, $\sup_{t\ge 0}\left|\left| \sup_{\theta\in \Theta} ||\nabla^j e_t(\theta)||\right| \right|_4<+\infty$ and we have $t^\alpha\left|\left| \sup_{\theta\in \Theta} ||\nabla (e_t- \epsilon_t)(\theta)||\right| \right|_{8/5}\longrightarrow 0$ , $t^\alpha\left|\left| \sup_{\theta\in \Theta} ||\nabla^2 (e_t- \epsilon_t)(\theta)||\right| \right|_{4/3}\longrightarrow 0$  and  $t^\alpha\left|\left| \sup_{\theta\in \Theta} ||\nabla^3 (e_t- \epsilon_t)(\theta)||\right| \right|_{1}\longrightarrow 0$ as $t\to\infty$ for all $\alpha>0$.
\end{enumerate}
\end{lemme}
We then show that the LSE is asymptotically equivalent to $Q_n(\theta)$:
\begin{prop}\label{prop_prelim}
Under the same assumptions in Lemma \ref{lemme_prelim}, we have that, for all $\alpha\in (0,1)$,
\begin{enumerate}
\item $\sup_{\theta\in \Theta} |Q_n(\theta)-O_n(\theta)|$ converges a.s. to $0$, and $n^\alpha \left|\left| \sup_{\theta\in \Theta} |Q_n(\theta)-O_n(\theta)|\right| \right|_1 $ tends to $0$ as $n\to\infty$,
\item  $\sup_{\theta\in \Theta} ||\nabla(Q_n(\theta)-O_n(\theta))||$ and $\sup_{\theta\in \Theta} ||\nabla^j(Q_n(\theta)-O_n(\theta))||, \text{ for }j=2,3$ converge a.s. to
$0$,
\item $n^\alpha \left|\left| \sup_{\theta\in \Theta} |\nabla(Q_n-O_n)(\theta)|\right| \right|_1 \longrightarrow 0$ as $n\to\infty$.
\end{enumerate}
\end{prop}

\subsection{Asymptotic properties}
We now turn to the main results of the paper, i.e. the strong consistency and normality of the estimator $\hat{\theta}_n$.
\begin{prop}\label{theo_consistency}
Let {(\bf A1)}, {(\bf A4)} as well as stationarity condition {\bf (A5a)} hold. The estimator $\check{\theta}_n$ defined by (\ref{LSE}) converges a.s. towards $\theta_0$.
\end{prop}
\begin{theo}[Consistency of the estimator]\label{theo_consistency_vrai}
Let {(\bf A1)}, {(\bf A4)} as well as stationarity condition {\bf (A5a)} hold. The estimator $\hat{\theta}_n$ defined by (\ref{LSE_Qn}) converges a.s. towards $\theta_0$.
\end{theo}

\begin{theo}[Asymptotic normality for the estimator]\label{CLT_theorem}
Let us suppose that assumptions {(\bf A1)}, {(\bf A2)}, {(\bf A3)}, {(\bf A4)} and {(\bf A5b)} hold, and let $\hat{\theta}_n$ defined in (\ref{LSE_Qn}). We have the following Central Limit Theorem
\begin{equation}\label{CLT}
\sqrt{n}\left( \hat{\theta}_n-\theta_0\right)\stackrel{\cal D}{\longrightarrow} {\cal N}\left(0,\Omega:=J^{-1}I J^{-1}\right),\quad n\to +\infty,
\end{equation}
matrices $I$ and $J$ being defined as
\begin{eqnarray}
J&:=&J(\theta_0)=\e\left( \nabla\epsilon_t(\theta_0)  [\nabla\epsilon_t(\theta_0)]' \right),\label{expression_J}\\
I&:=&I(\theta_0)=\sum_{k=-\infty}^\infty\e\left(\epsilon_t(\theta_0)\epsilon_{t-k}(\theta_0)\nabla\epsilon_t(\theta_0)  [\nabla\epsilon_{t-k}(\theta_0)]'\right)\nonumber\\
&=& \sum_{k=-\infty}^{+\infty}\cov(\Upsilon_t,\Upsilon_{t-k}),\quad \mbox{where}\label{spectraldensityatzero}\\
\Upsilon_t &:=& \Upsilon_t(\theta_0)=\epsilon_t(\theta_0) \nabla\epsilon_t(\theta_0) .\label{upsilon}
\end{eqnarray}
\end{theo}

\begin{rem}
\label{I=2J} {\em In the strong ARMARC case, {\em i.e.} when {(\bf
A1)} is replaced by the assumption that $(\epsilon_t)$ is iid, we
have $I=\sigma^2 J$, so that the covariance matrix in the strong case is $\Omega_S:=\sigma^2 J^{-1}$. In the general case we have
$I\neq\sigma^2 J$. As a consequence the ready-made software used to fit
ARMARC do not provide a correct estimation for weak ARMARC
processes.}
\end{rem}

\subsection{Estimating the asymptotic covariance matrix}
\label{estimationOmega}

Theorem \ref{CLT_theorem} can be used to obtain confidence intervals and significance tests for the parameters.
The asymptotic covariance $\Omega$ must however be estimated.
The matrix $J$ can easily be estimated by its empirical counterpart
$$\hat{J}_{n}=\frac{1}{n}\sum_{t=1}^n \nabla e_t(\hat{\theta}_n) [\nabla e_t(\hat{\theta}_n)]'.$$
In the standard strong ARMARC case $\hat{\Omega}_S=\hat\sigma^2\hat{J}_n^{-1}$ is
a strongly consistent estimator of $\Omega$. In the general weak
ARMARC case this estimator is not consistent when $I\neq \sigma^2 J$ (see
Remark \ref{I=2J}). So we need a consistent estimator of $I$, defined by \eqref{spectraldensityatzero}.

The estimation of this long-run covariance $I$ is more complicated. In the literature, two types of estimators are generally employed:  the
nonparametric kernel estimator, also called  Heteroskedasticity and Autocorrelation Consistent (HAC) estimators (see  \cite{andrews} and \cite{newey} for general references, and \cite{FZ2007} for an application to testing strong linearity in weak ARMA models) and spectral density estimators (see e.g. \cite{berk} and \cite{haan} for a general references and \cite{BMCF2012} for estimating $I$ when $\theta$ is not necessarily equal to $\theta_0$).

In the present paper, we focus on an estimator based on a spectral density form for $I$.

Interpreting $(2\pi)^{-1}I$  as the spectral density of the
stationary process $(\Upsilon_t)$ evaluated at frequency 0 (see
\cite{broc-d}, p. 459) of the process (\ref{upsilon}).
This approach, which has been studied by \cite{berk}
(see also \cite{haan}), rests on the expression
\begin{equation}\label{Idensitespectrale}
I=\mathbf{\Phi}^{-1}(1)\Sigma_u \mathbf{\Phi}^{-1}(1)
\end{equation}
when $(\Upsilon_t)$  satisfies an AR$(\infty)$ representation of the form
\begin{equation}
\label{arinfty}
\mathbf{\Phi}(L)\Upsilon_t:=
\Upsilon_t+\sum_{i=1}^{\infty}\Phi_i\Upsilon_{t-i}=u_t,
\end{equation} where $u_t$ is a $(p+q)K$-variate weak white noise with covariance matrix $\Sigma_u$.
Note incidentally that, since $(\Upsilon_t)$ depends on the regime $(\Delta_t)$, then so does the weak white noise $(u_t)$.
Let $\hat{\Upsilon}_t$ be the vector
obtained by replacing $\theta_0$ by $\hat{\theta}_n$ in $\Upsilon_t$ and
$\hat{\mathbf{\Phi}}_r(z)=\mathrm{I}_{(p+q)K}+\sum_{i=1}^r\hat{\Phi}_{r,i}z^i$,
where $\hat{\Phi}_{r,1},\dots,\hat{\Phi}_{r,r}$ denote the
coefficients  of the least squares regression of $\hat{\Upsilon}_t$ on
$\hat{\Upsilon}_{t-1},\dots,\hat{\Upsilon}_{t-r}$. Let
$\hat{u}_{r,t}$ be  the residuals of this regression, and let
$\hat{\Sigma}_{\hat{u}_r}$ be the empirical covariance of
$\hat{u}_{r,1},\dots, \hat{u}_{r,n}$.

In the framework of linear processes with independent innovations, \cite{berk} showed that the spectral density
can be consistently estimated by fitting autoregressive models of order $r=r(n)$, whenever $r\to\infty$ and $r^3/n\to0$ as $n\to\infty$. It can be shown that this result remains valid for the linear process $(\Upsilon_t)$, though its innovation $(u_t)$ is not an independent process. Another difference with \cite{berk}, is that $(\Upsilon_t)$ is not directly observed and is replaced by $(\hat\Upsilon_t)$.

We are now able to state the following theorem.
\begin{theo}\label{estimationI}
In addition to the assumptions of Theorem \ref{CLT_theorem}, assume
that the process $(\Upsilon_t)$
 defined in (\ref{upsilon}) admits an AR$(\infty)$ representation (\ref{arinfty})
in which the roots of $\det\mathbf{\Phi}(z)=0$ are outside the unit
disk, $\|\Phi_i\|=o(i^{-2})$, and $\Sigma_u=\mbox{Var}(u_t)$ is
non-singular. Moreover we assume that
$\mathbb{E}\left|\epsilon_t\right|^{8+4\nu}<\infty$ and $\sum_{k=0}^{\infty}
\{\alpha_{\epsilon} (k)\}^{\nu/(2+\nu)}<\infty$  and $\sum_{k=0}^{\infty}
\{\alpha_{\Delta} (k)\}^{\nu/(2+\nu)}<\infty$ for some $\nu>0$.
Then the spectral estimator of $I$
$$\hat{I}^{\mathrm{SP}}:=
\hat{\mathbf{\Phi}}_r^{-1}(1)\hat{\Sigma}_{\hat{u}_r}
\hat{\mathbf{\Phi}}_r'^{-1}(1)\to I$$ in probability when
$r=r(n)\to\infty$ and $r^3/n\to 0$ as $n\to\infty$.
\end{theo}
The matrix $\Omega$ is then estimated by a "sandwich" estimator of
the form
$$\hat{\Omega}^{\mathrm{SP}}=\hat{J}_n^{-1}\hat{I}^{\mathrm{SP}}\hat{J}_n^{-1},\quad
\hat{I}^{\mathrm{SP}}=
\hat{\mathbf{\Phi}}_r^{-1}(1)\hat{\Sigma}_{\hat{u}_r}
\hat{\mathbf{\Phi}}_r'^{-1}(1).$$
\subsection{Testing linear restrictions on the parameter}
\label{testWald}
It may be of interest to test $s_0$ linear constraints on the elements of $\theta_0$.
Let $R$ be a given matrix of size $s_0\times(p+q)K$ and rank $s_0$, and let $r_0$ and $r_1$ be given vectors of size $s_0$ such that
$r_1\neq r_0$.
Consider the testing problem
\begin{equation}\label{hypnulle}
H_0:R\theta_0=r_0\qquad\mbox{against}\qquad H_1:R\theta_0=r_1.
\end{equation}
The Wald principle is employed frequently for testing \eqref{hypnulle}. We now examine if this principle remains valid in the non standard framework of weak ARMARC models.

Let $\hat{\Omega}=\hat{J}^{-1}\hat{I}\hat{J}^{-1}$, where $\hat{J}$
and $\hat{I}$  are consistent estimators of $J$ and $I$, as defined
in Section \ref{estimationOmega}. Under the assumptions of Theorems~\ref{CLT_theorem} and \ref{estimationI}, and the assumption that $I$ is invertible, the modified Wald statistic
$$\mathrm{\mathbf{W}}_M:=n(R_0\hat{\theta}_n-r_0)'(R_0\hat{\Omega}R_0')^{-1}(R_0\hat{\theta}_n-r_0)$$
asymptotically follows a $\chi^2_{s_0}$ distribution under $H_0$. Therefore, the standard formulation of the Wald test remains valid.
More precisely, at the asymptotic level $\alpha$, the modified Wald test consists in rejecting $H_0$  when
$\mathrm{\mathbf{W}}_M>\chi^2_{s_0}(1-\alpha)$. It is however important to note that a consistent estimator of the form $\hat{\Omega}=\hat{J}^{-1}\hat{I}\hat{J}^{-1}$ is required. The estimator $\hat{\Omega}_S=\hat\sigma^2\hat{J}^{-1}$, which is routinely used in the time series softwares, is only valid in the strong ARMARC case. Thus standard  Wald statistic takes the following form
$$\mathrm{\mathbf{W}}_S:=n(R_0\hat{\theta}_n-r_0)'(R_0\hat{\Omega}_SR_0')^{-1}(R_0\hat{\theta}_n-r_0),$$ which asymptotically follows a $\chi^2_{s_0}$ distribution under $H_0$.

\section{Examples}\label{examples}
In this section, we give examples of weak ARMARC$(1,0)$ model with iid and  correlated  process $(\Delta_t)_{t\in \nbZ}$.
\subsection{Independent and identically distributed process $(\Delta_t)_{t\in \nbZ}$: the ARMARC$(1,0)$ model}
We provide here some results that show that we obtain very neat results in the particular case where the state space verifies ${\cal S}\subset \nbR$,  $(\Delta_t)_{t\in \mathbb{Z}}$ is i.i.d. and  satisfies $\nbE(\Delta_t)=0$. We consider a particular $AR(1)$ model where \eqref{def_arma_pq0} reads
\begin{equation}\label{def_ar1_iid}
X_t - a^0 \Delta_t X_{t-1}=\epsilon_t ,
\end{equation}
i.e. $a^0(s)=a^0 s$ for all $s\in {\cal S}$, where $a^0=\theta_0$ is  here the unknown (scalar) parameter and belongs to some compact set $\Theta\subset \nbR$, and the state space ${\cal S}$ is a finite subset of $\nbR$. It is easy to check that, using the notations defined in Section \ref{subsec_stationarity}, we have that $B(s, \theta)$ is not defined (as here $q=0$), $A(s)=g_1^a(s,\theta_0)=a^0 s$ and $\Psi(s)=A(s)=a^0 s$. Stationarity condition {\bf (A5a)} in Proposition \ref{prop_stationary} is translated as
\begin{equation}\label{cond_stationary_iid}
|a^0|\left[\nbE(|\Delta_0|^8)\right]^{1/8}<1 \quad \iff\quad a^0 \in \left(-\frac{1}{\left[\nbE(|\Delta_0|^8)\right]^{1/8}},\frac{1}{\left[\nbE(|\Delta_0|^8)\right]^{1/8}}\right).
\end{equation}
Let us note that \eqref{cond_stationary_iid} allows some interesting cases where we have $|a^0 \Delta_t|\ge 1$, which is a non stable state case and is somewhat a paradox to the usual stability condition in the classical $AR(1)$ model where it is standard that the process $(X_t)_{t\in\nbZ}$ defined by $X_t=a X_{t-1}+\epsilon_t$ is stable iff $|a|<1$. One simple example is when $(\Delta_t)_{t\in\nbZ}$ is i.i.d. with distribution $\Delta_t\sim \frac{1}{4}\delta_{-1}+\frac{1}{2}+\frac{1}{4}\delta_{+1}$, in which case \eqref{cond_stationary_iid} reads $|a^0|<2$, so that  $|a^0 \Delta_t|=\frac{3}{2}>1$ if we pick for example $a^0=\frac{3}{2}$, when $\Delta_t=1$.\\
Furthermore, we compute easily that, for all $a=\theta\in \Theta$, $\epsilon_t(a)=X_t-a \Delta_t X_{t-1}$, where $X_t$ has the classical decomposition obtained from \eqref{def_ar1_iid}:
\begin{equation}\label{dec_X_t_iid}
X_t=\sum_{i=0}^\infty \prod_{j=0}^{i-1}(a^0 \Delta_{t-j})\epsilon_{t-i}.
\end{equation}
Since Assumption {(\bf A2)} is trivially satisfied here, we only need suppose that {(\bf A1)}, {(\bf A3)} and {(\bf A4)} hold for some $\nu>0$. In that case, Theorems \ref{theo_consistency_vrai} and \ref{CLT_theorem} translate as
\begin{theo}\label{theo_example_iid}
$\hat{\theta}_n$ defined as \eqref{LSE_Qn} converges a.s. towards $\theta_0=a^0$. Besides, we have the asymptotic normality
\begin{equation}\label{CLT_iid}
\sqrt{n}\left( \hat{\theta}_n-a^0\right)\stackrel{\cal D}{\longrightarrow} {\cal N}\left(0,\Omega\right),\quad n\to +\infty,
\end{equation}
where
\begin{equation}\label{Omega_iid}
\Omega= \frac{\left[ 1-(a^0)^2\nbE(\Delta_0^2)\right]^2}{\nbE(\Delta_0^2)} \sum_{i=0}^\infty \left[(a^0)^{2}\nbE(\Delta_0^2)\right]^i \nbE(\epsilon_t^2 \epsilon_{t-i}^2).
\end{equation}
\end{theo}
\begin{proof}
Strong consistency and asymptotic normality are straightforward consequences of Theorems \ref{theo_consistency_vrai} and \ref{CLT_theorem}. In order to compute $\Omega$, we need to compute $J=J(a^0)$ and $I=I(a^0)$ in \eqref{CLT}. Since $\frac{\partial}{\partial a}\epsilon_t(a)=-\Delta_t X_{t-1}$, and since $\nbE(X_t^2)$ is equal to $\frac{1}{1-(a^0)^2\nbE(\Delta_0^2)}$ thanks to \eqref{dec_X_t_iid} and the fact that $(\epsilon_t)_{t\in \mathbb Z}$ is a weak noise, independent from $(\Delta_t)_{t\in \mathbb{Z}}$. Hence we have, by independence of $\Delta_t$ from $X_{t-1}$,
$$
J(a^0)=\nbE\left( \left[ \frac{\partial}{\partial a}\epsilon_t(a^0)\right]^2\right)=\nbE\left(  \Delta_t^2 X_{t-1}^2\right)=\frac{\nbE(\Delta_0^2)}{1-(a^0)^2\nbE(\Delta_0^2)}.
$$
There then remains to get $I=I(a^0)$. From Theorem \ref{CLT_theorem} we need to compute the expectation of
\begin{eqnarray*}
&&\epsilon_t(a^0)\epsilon_{t-k}(a^0)\frac{\partial \epsilon_t(a^0)}{\partial a}\frac{\partial \epsilon_{t-k}(a^0)}{\partial a}= \epsilon_t \epsilon_{t-k} \Delta_t X_{t-1} \Delta_{t-k} X_{t-k-1}\\
&=& \epsilon_t \epsilon_{t-k} \Delta_t \left[\sum_{i=0}^\infty \prod_{j=0}^{i-1}(a^0 \Delta_{t-1-j})\epsilon_{t-1-i}\right] \Delta_{t-k} \left[\sum_{i'=0}^\infty \prod_{j'=0}^{i'-1}(a^0 \Delta_{t-k-1-j'})\epsilon_{t-k-1-i'}\right]
\end{eqnarray*}
for all $k\in \nbN$. Using independence of the processes $(\epsilon_t)_{t\in \mathbb Z}$ and $(\Delta_t)_{t\in \mathbb{Z}}$, we have
\begin{equation}\label{expectation_compute_iid}
\nbE \left( \epsilon_t(a^0)\epsilon_{t-k}(a^0)\frac{\partial \epsilon_t(a^0)}{\partial a}\frac{\partial \epsilon_{t-k}(a^0)}{\partial a}\right)=\sum_{i,i'=0}^\infty V^{i,i',k} d(k,1+i,k+1+i')
\end{equation}
where $d(n,m,r):=\nbE(\epsilon_0 \epsilon_{-n}\epsilon_{-m}\epsilon_{-r})$ for  all $n$, $m$, $r$ in $\nbN$, and
$$
V^{i,i',k}:=(a^0)^{i+i'+2}\nbE \left( \prod_{j=-1}^{i-1} \Delta_{t-1-j}.\prod_{j'=-1}^{i'-1} \Delta_{t-k-1-j'} \right).
$$
Since $\Delta_t$ is centered, we check immediately that $V^{i,i',k}$ is non zero if and and only if $k=0$ and $i=i'$, in which case we have $V^{i,i,0}=\left[(a^0)^{2}\nbE(\Delta_0^2)\right]^{i+1}$. Hence \eqref{expectation_compute_iid} is in that case equal to $(a^0)^{2}\nbE(\Delta_0^2)\sum_{i=0}^\infty \left[(a^0)^{2}\nbE(\Delta_0^2)\right]^i \nbE(\epsilon_t^2 \epsilon_{t-1-i}^2)$, which is also the expression for $I(a^0)$, yielding \eqref{Omega_iid}.
\end{proof}

\subsection{Modulating Markov chain}\label{example_Markov}
We now give an example of process $(\Delta_t)_{t\in\nbZ}$ with correlated trajectories by considering a discrete time stationary irreducible finite Markov chain (hence, ergodic) with state space $\mathcal{S}=\{1,2\}$ and transition probabilities matrix
$$
P=(p(i,j))_{i,j=1,2}=\left(\begin{array}{cc} 0 &1\\p& 1-p\end{array}\right),
$$
where $p$ lies in $(0,1)$, and with stationary distribution
\begin{equation}\label{stationary_distrib_Markov}
(\nbP(\Delta_t=1),\ \nbP(\Delta_t=2))=(\pi_1,\pi_2)=\left(\frac{p}{p+1},\frac{1}{p+1}\right).
\end{equation}
We also consider, as in the previous section, an ARMARC$(1,0)$ model of the form
\begin{equation}\label{def_ar1_Markov}
X_t - a^0 (\Delta_t) X_{t-1}=\epsilon_t ,
\end{equation}
where parameter $\theta_0=(a^0(1),a^0(2))$ verifies $a^0(1)=0$, in order to have nice expressions later for asymptotic normality. In order to establish the stationarity condition {\bf (A5a)} we need to compute $\nbE\left[ || \prod_{k=1}^t a^0(\Delta_k)||^8\right]$ which, because of $a^0(1)=0$, simplifies to
$$
\nbE\left[ || \prod_{k=1}^t a^0(\Delta_k)||^8\right]= |a^0(2)|^{8t} \nbP(\Delta_1=...=\Delta_t=2)=|a^0(2)|^{8t}\pi_2 (1-p)^{t-1},
$$
so that stationarity condition {\bf (A5a)} here reads
\begin{equation}\label{cond_stability_Markov}
a^0(2) \in \left( -\frac{1}{(1-p)^{1/8}}, \frac{1}{(1-p)^{1/8}}\right).
\end{equation}
Here again, as in the i.i.d. case for $(\Delta_t)_{t\in\nbZ}$, and since $\frac{1}{(1-p)^{1/8}} >1$, we can allow $ |a^0(2)|$ to be larger than $1$ so that state $2\in {\cal S}$ is non stable, although the process is stationary. Let us furthermore note that the Markov chain $(\Delta_t)_{t\in\nbZ}$ verifies the Doeblin condition so is geometrically ergodic, hence has exponentially fast strong mixing property (see \cite{Jones2004}), so that {\bf (A2)} is satisfied. We furthermore suppose that {\bf (A1)}, {\bf (A3)} and {\bf (A4)} hold for some $\nu>0$. As in \eqref{dec_X_t_iid}, we have
\begin{equation}\label{dec_X_t_Markov}
X_t=\sum_{i=0}^\infty \prod_{j=0}^{i-1}a^0 (\Delta_{t-j})\epsilon_{t-i},
\end{equation}
and $\epsilon_t(a)=X_t-a(\Delta_t) X_{t-1}$ for all $\theta=(a(1),a(2))\in \Theta$. We introduce matrices $Q(l)$, $l\in{\cal S}=\{1,2\}$ as well as vector $\pi_V$ defined by
\begin{equation}\label{def_matrix_Q(l)}
Q(1)
=\left(\begin{array}{cc}
0 & p\\
0 & 0
\end{array}
\right),\quad
Q(2)=
\left(\begin{array}{cc}
0& 0\\
1 & 1-p
\end{array}
\right),\quad \pi_V=(0,\pi_2)'.
\end{equation}
Theorems \ref{theo_consistency_vrai} and \ref{CLT_theorem} read
\begin{theo}
$\hat{\theta}_n$ defined as \eqref{LSE_Qn} converges a.s. towards $\theta_0=(a^0(1), a^0(2))$. Besides, we have the asymptotic normality
\begin{equation}\label{CLT_Markov}
\sqrt{n}\left( \hat{\theta}_n-\theta_0\right)\stackrel{\cal D}{\longrightarrow} {\cal N}\left(0,\Omega\right),\quad n\to +\infty,
\end{equation}
where $
\Omega=J^{-1} IJ^{-1}$, matrices $J=(J(l,l'))_{l,l'\in {\cal S}^2}$ and $I=(I(l,l'))_{l,l'\in {\cal S}^2}$ being defined by
\begin{equation}\label{Matrix_J_Markov}
\begin{array}{rcl}
J(1,1)&=& \sigma^2 \frac{p}{p+1}\frac{1+a^0(2)^2 p}{1-a^0(2)^2(1-p)},\\
J(1,2) &=& J(2,1)=0,\\
J(2,2)&=& \sigma^2 \frac{1}{p+1}\frac{1}{1-a^0(2)^2(1-p)}
\end{array}
\end{equation}
and $I(l,l')=I(l,l',0)+2\sum_{k=1}^\infty I(l,l',k)$, where
\begin{equation}\label{Matrix_I_Markov}
I(l,l',k)=\left\{
\begin{array}{cl}
\sum_{i'=0}^\infty \1' \left[ \sum_{i<k} a^0(2)^{i+i'} Q(l) Q(2)^i {P'} ^{k-i-1} Q(l') Q(2)^{i'} d(k,i+1,k+i'+1) \right. & \\
 \sum_{k\le i\le k+i'}  a^0(2)^{i+i'} Q(l) Q(2)^{k+i'}d(k,i+1,k+i'+1) & \\
 \left. + \sum_{i>k+i'}a^0(2)^{i+i'}Q(l) Q(2)^id(k,i+1,k+i'+1)\right]\pi_V, & l'=2,\\
\sum_{i'=0}^\infty \1'  \left[\sum_{i<k} a^0(2)^{i+i'} Q(l) Q(2)^i {P'} ^{k-i-1} Q(l') Q(2)^{i'} d(k,i+1,k+i'+1)\right]\pi_V, & l'=1,
\end{array}
\right.
\end{equation}
where $d(i,i',i''):=\nbE(\epsilon_{t}\epsilon_{t-i}\epsilon_{t-i'}\epsilon_{t-i''})$, $i$, $i'$, $i''$ in $\nbN$.
\end{theo}
\begin{proof}
It is not hard to check that, for all $i\in {\cal S}=\{1,2\}$ and $a=(a(1),a(2))$, $\frac{\partial}{\partial a(i)} \epsilon_t (a)=- \1_{[\Delta_t=i]}X_{t-1}$. We compute easily
$$
 \nabla\epsilon_t(\theta_0)   [\nabla\epsilon_t(\theta_0)]'
=\left(
\begin{array}{cc}
\1_{[\Delta_t=1]}X_{t-1}^2 &0 \\
0 & \1_{[\Delta_t=2]}X_{t-1}^2
\end{array}
\right),
$$
so that it suffices to compute  $\nbE(\1_{[\Delta_t=l]}X_{t-1}^2)$ for all $l=1,2$, in order to compute $J$. By the usual argument of independence of the Markov chain from the weak white noise, and since $a^0(1)=0$, we get, for $l=1,2$,
\begin{eqnarray*}
\nbE(\1_{[\Delta_t=l]}X_{t-1}^2)&=&\sigma^2 \sum_{i=0}^\infty \nbE\left( \1_{[\Delta_t=l]}\prod_{j=0}^{i-1}(a^0 (\Delta_{t-1-j}))^2\right)\\
&=&\sigma^2 \pi_l + \sigma^2 \sum_{i=1}^\infty a^0(2)^{2i}\pi_2 (1-p)^{i-1} p(2,l)=  \sigma^2  \pi_l +\sigma^2 \frac{a^0(2)^2 \pi_2}{1-a^0(2)^2(1-p)}p(2,l),
\end{eqnarray*}
so that those quantities along with \eqref{stationary_distrib_Markov} yield the expression for the for matrix $J$ in \eqref{Matrix_J_Markov}.\\
In order to compute $I$, we need to take the expectation of $\epsilon_t(\theta_0)\epsilon_{t-k}(\theta_0) \frac{\partial}{\partial a(l)} \epsilon_t (\theta_0)\frac{\partial}{\partial a(l')} \epsilon_{t-k} (\theta_0)=\epsilon_t \epsilon_{t-k} \1_{[\Delta_t=l]} X_{t-1}\1_{[\Delta_{t-k}=l']} X_{t-1-k}$ for all $l$, $l'$ in ${\cal S}$ and $k\in\nbN$. As in \eqref{expectation_compute_iid} in the proof of Theorem \ref{theo_example_iid}, this expectation is equal to $\sum_{i,i'=0}^\infty V^{i,i',k}(l,l') d(k,1+i,k+1+i')$ where
$$
V^{i,i',k}(l,l'):=\nbE \left( \1_{[\Delta_t=l]} \prod_{j=0}^{i-1} a^0(\Delta_{t-1-j}). \1_{[\Delta_{t-k}=l']}\prod_{j'=0}^{i'-1} a^0(\Delta_{t-k-1-j'}) \right).
$$
This quantity can be obtained straightforwardly using e.g. Lemma 1 of \cite{FG2004}. Remembering that $Q(1)$, $Q(2)$ and $\pi_V$ are defined by \eqref{def_matrix_Q(l)}, we then have the following expression for $V^{i,i',k}(l,l')$, according to whether $t-i>t-k\iff i<k$, $t-k\ge t-i \ge t-k-i'\iff k\le i \le k+i'$ or $t-k-i' \ge t-i\iff k+i'<i$:
$$
V^{i,i',k}(l,l')=\left\{
\begin{array}{cl}
a^0(2)^{i+i'}\1' Q(l) Q(2)^i {P'} ^{k-i-1} Q(l') Q(2)^{i'} \pi_V, & i<k, \\
a^0(2)^{i+i'}\1'  Q(l)Q(2)^{k+i'}\pi_V, & k\le i \le k+i', \ l'=2,\\
0, & k\le i \le k+i', \ l'=1,\\
a^0(2)^{i+i'}\1' Q(l)Q(2)^{i}\pi_V,& k+i'<i, \ l'=2,\\
0, & k+i'<i, \ l'=1 ,
\end{array}
\right.
$$
yielding \eqref{Matrix_I_Markov}.
\end{proof}

\section{Numerical illustrations}
\label{numericalIllustrations}
We study numerically the behaviour of our estimator for
strong and weak ARMARC models. We consider the following ARMARC$(1,1)$ model
\begin{eqnarray}
\label{ARMA11MonteCarlo}
X_{t}=a_1^0(\Delta_t)X_{t-1}+\epsilon_{t}+b_1^0(\Delta_t)\epsilon_{t-1},
\end{eqnarray}
where the innovation process  $(\epsilon_t)$ follows a strong or a weak white noise. This model is to be compared with the example in Section 3.4 of \cite{Gautier04}
or Section 4 of \cite{FG2003}.
The process $(\Delta_t)$ is simulated (independently of $(\epsilon_t)$) according to the law of a stationary Markov chain with state-space $\mathcal{S}=\{1,2\}$ and transition probabilities matrix
$$\left(\begin{array}{cc}p(1,1)&1-p(1,1)\\1-p(2,2)&p(2,2)\end{array}\right)=
\left(\begin{array}{cc}0.95&0.05\\0.05&0.95\end{array}\right).$$
By an argument similar to the one explained in the example in Section \ref{example_Markov}, one has that this Markov chain is geometrically ergodic, so that Condition {\bf (A2)} is satisfied. We first consider the strong ARMARC case. To generate this model, we assume  the innovation process $(\epsilon_t)$ in
(\ref{ARMA11MonteCarlo}) is
defined by an iid sequence such that
\begin{equation} \label{bruitfort}
\epsilon_{t} \  \overset{\cal D}{=} \ {\cal
N}(0,1).
\end{equation}
Following \cite{rt1996}, we propose a set of two experiments for weak ARMARC with innovation processes $\epsilon_t$ in (\ref{ARMA11MonteCarlo}) defined by
\begin{align}
\label{RT}
\epsilon_{t}& =\eta_{t}(|\eta_{t-1}|+1)^{-1}, \\
\label{PT}
\epsilon_{t}& =\eta_{t}^2\eta_{t-1},
\end{align}
where  $(\eta_t)_{t\ge 1}$ is a sequence of iid standard Gaussian random variable.
The noises defined by  \eqref{RT} and \eqref{PT} are a direct extension of the
weak noises  in Examples 2.1 and 2.2 defined by \cite{rt1996}. Thus we easily check that those weak noises meet the requirements of assumptions {\bf (A1)} to {\bf (A4)} for all $\nu>0$. We also note that the innovation process \eqref{RT} is a martingale difference, as opposed to \eqref{PT}.

The numerical illustrations of this section are made with the free
statistical software {\tt R} (see {\tt http://cran.r-project.org/}). We
simulated $N=1,000$ independent trajectories of size $n=2,000$ of
Model (\ref{ARMA11MonteCarlo}), first with the strong Gaussian noise
(\ref{bruitfort}), second with the weak noise (\ref{RT}) and
third with the weak noise (\ref{PT}).

Recall that the regimes $(\Delta_t)$ are supposed to be known. For each of these $N$ replications,
we estimate the coefficient $\theta_0=(a_1^0(1),a_1^0(2),b_1^0(1),b_1^0(2))'=(0.90,-0.45,0.10,0.85)'$.

Figures~\ref{strongRCARMA} and \ref{weakRCARMA} display the realization  of length 400 of Model (\ref{ARMA11MonteCarlo})  in the strong (\ref{bruitfort}) and weak (\ref{PT}) noises  cases. Note that here stationarity condition {\bf (A5a)} in Proposition \ref{prop_stationary} is trivially satisfied as all coefficients $a_1^0(1)$, $a_1^0(2)$, $b_1^0(1)$, $b_1^0(2)$ are all less than $1$ in modulus.

Figure~\ref{fig1} compares the distribution of the least squares estimators (LSE) in the strong and the two weak noises
 cases. The distributions of $\hat{a}_1^0(1)$, $\hat{a}_1^0(2)$ and $\hat{b}_1^0(2)$ are similar in all cases, whereas the LSE of $\hat{b}_1^0(1)$ is  more accurate in the weak case with noise
 \eqref{RT} than in the strong one.
 Similar simulation experiments reveal that the
situation is opposite, that is the LSE is more accurate in the
strong case than in the weak case, when the weak noise is defined by
\eqref{PT}.
This is in accordance with the results of \cite{rt1996} who showed
that, with similar noises, the asymptotic covariance of the sample
autocorrelations can be greater (for noise \eqref{PT}) or less (for noise \eqref{RT}) than 1 as well (1 is the
asymptotic covariance for strong white noises).

Figure~\ref{fig2} compares the standard estimator
$\hat{\Omega}_S=\hat\sigma^2 \hat{J}^{-1}$ and the sandwich estimator
$\hat{\Omega}=\hat{J}^{-1}\hat{I}^{\mathrm{SP}}\hat{J}^{-1}$ of the LSE asymptotic covariance $\Omega$. We used the spectral estimator
$\hat{I}:=\hat{I}^{\mathrm{SP}}$ defined in Theorem~\ref{estimationI},
and the AR order $r$ is automatically selected by AIC, using the
function {\tt VARselect()} of the {\bf vars} R package. In the
strong ARMARC case we know that the two estimators are consistent. In
view of the two top panels of Figure~\ref{fig2}, it seems that the
sandwich estimator is less accurate in the strong case. This is not
surprising because the sandwich estimator is more robust, in the
sense that this estimator continues to be consistent in the weak
ARMARC case, contrary to the standard estimator. It is clear that in
the weak cases $n\mbox{Var}\left\{\hat{b}_1^0(1)-{b}_1^0(1)\right\}^2$
is better estimated by $\hat{\Omega}^{\mathrm{SP}}(3,3)$ (see the
box-plot (c) of the right-middle and right-bottom panel of Figure~\ref{fig2}) than by
$\hat{\Omega}_S(3,3)$ (box-plot (c) of the left-middle and left-bottom panel). The
failure of the standard estimator of $\Omega$ in the weak ARMARC
framework may have important consequences in terms of identification
or hypothesis testing and validation.

Table \ref{wald} displays the relative percentages of rejection of the standard and modified Wald tests ($\mathrm{\mathbf{W}}_S$ and $\mathrm{\mathbf{W}}_M$) proposed in Section \ref{testWald}
 for testing the null hypothesis $H_0: b_1^0(1)=0$.
We simulated $N=1,000$ independent trajectories of size $n=500$, $n=2,000$ and $n=10,000$ of the  strong ARMARC$(1,1)$ model (\ref{ARMA11MonteCarlo})--\eqref{bruitfort} and  of two  weak ARMARC$(1,1)$ model (\ref{ARMA11MonteCarlo}) with first noise \eqref{RT} and second \eqref{PT}.
The nominal asymptotic level of the tests is $\alpha=5\%$  and the empirical  size over the $N$ independent replications should vary between the significant limits 3.6\% and 6.4\% with probability 95\%.
The line in bold corresponds to the null hypothesis $H_0$.
For the strong ARMARC model \eqref{ARMA11MonteCarlo}--\eqref{bruitfort}, the relative rejection
frequencies of the $\mathrm{\mathbf{W}}_S$ and $\mathrm{\mathbf{W}}_M$ tests are close to the nominal 5\% level when $b_1^0(1)=0$, and are close to 100\% under the alternative when $n$ is large. In this strong ARMARC example, the $\mathrm{\mathbf{W}}_S$ and $\mathrm{\mathbf{W}}_M$ tests have very similar powers  under the alternative for all sizes.
As expected, for the two weak ARMARC models \eqref{ARMA11MonteCarlo}--\eqref{RT} and \eqref{ARMA11MonteCarlo}--\eqref{PT}, the relative rejection frequencies of the standard  $\mathrm{\mathbf{W}}_S$  Wald test is definitely outside the significant limits. Thus the error of first kind is well controlled by all the tests in the strong case, but only by the $\mathrm{\mathbf{W}}_M$ modified version test in the weak cases (Model \eqref{ARMA11MonteCarlo}--\eqref{RT}) and (Model \eqref{ARMA11MonteCarlo}--\eqref{PT}, for $n$ large) when $b_1^0(1)=0$. Note also that for Models \eqref{ARMA11MonteCarlo}--\eqref{RT}
and \eqref{ARMA11MonteCarlo}--\eqref{PT}, the relative rejection frequencies of the  $\mathrm{\mathbf{W}}_M$ test tend rapidly to 100\% as $n$ increases  under the alternative. By contrast the empirical powers of the standard $\mathrm{\mathbf{W}}_S$  test is hardly interpretable for Models \eqref{ARMA11MonteCarlo}--\eqref{RT} and \eqref{ARMA11MonteCarlo}--\eqref{PT}. This is not surprising because we have already seen in Table \ref{wald} that the standard  version of the $\mathrm{\mathbf{W}}_S$ test does not correctly control the error of first kind in the weak ARMARC frameworks.

From these simulation experiments and from the asymptotic theory, we
draw the conclusion that the standard methodology, based on the
LSE, allows to fit ARMARC representations of a wide class of
nonlinear time series. This standard methodology, including in particular the significance tests on the parameters,
needs however to be adapted to take into account the possible lack
of independence of the errors terms. In future works, we intend to
study how the existing identification  and diagnostic checking procedures should be adapted in the weak ARMARC framework considered in the present paper.

\begin{figure}[hbtp]
\vspace*{12cm} \protect \includegraphics{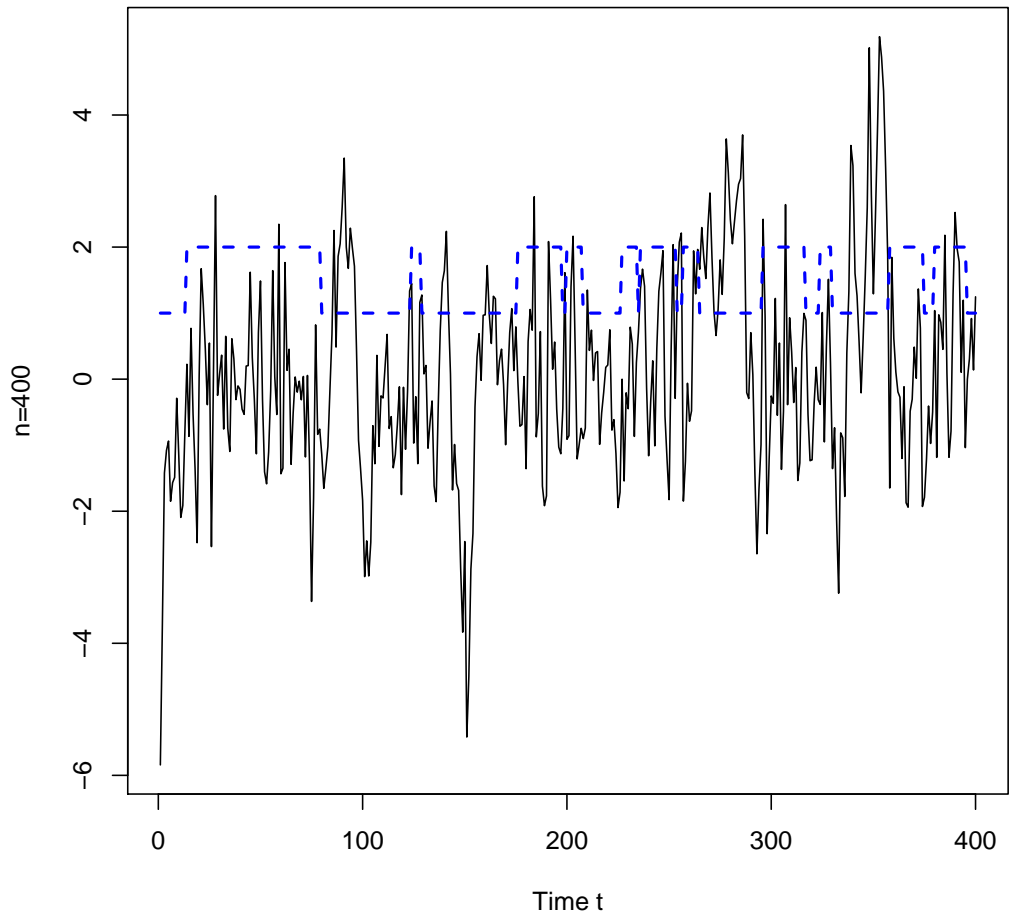} \vspace*{2.5cm}
\caption{\label{strongRCARMA} {\footnotesize Simulation of length 400 of Model (\ref{ARMA11MonteCarlo})--(\ref{bruitfort}) with $\theta_0=(a_1^0(1),a_1^0(2),b_1^0(1),b_1^0(2))'=(0.90,-0.45,0.10,0.85)'$, . The process $(X_t)$ is drawn in full line, the Markov chain $(\Delta_t)$ is plotted in dotted line. }}
\end{figure}

\begin{figure}[hbtp]
\vspace*{12cm} \protect \includegraphics{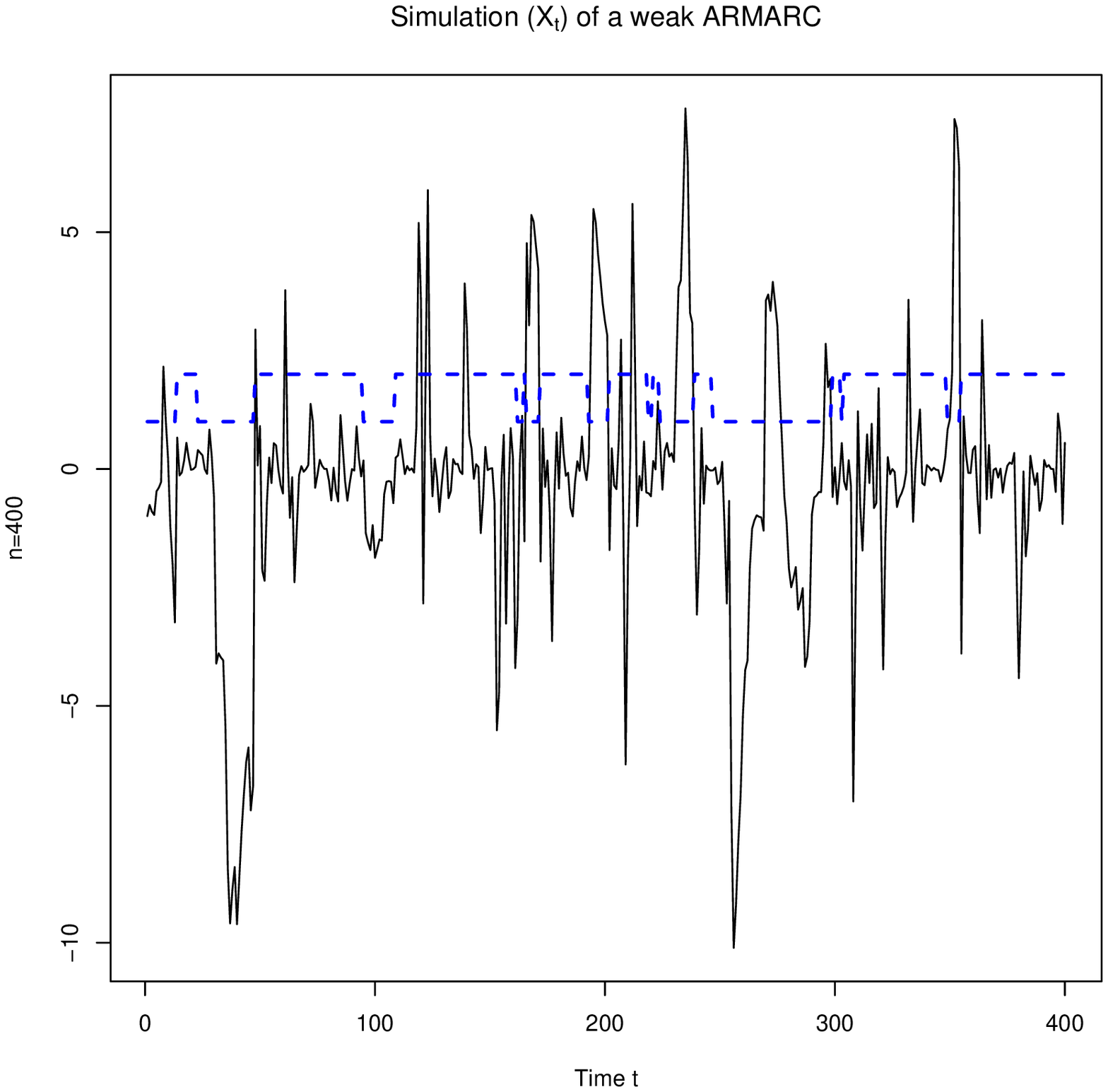} \vspace*{2.5cm}
\caption{\label{weakRCARMA} {\footnotesize Simulation of length 400 of Model (\ref{ARMA11MonteCarlo})--(\ref{PT}) with $\theta_0=(a_1^0(1),a_1^0(2),b_1^0(1),b_1^0(2))'=(0.90,-0.45,0.10,0.85)'$. The process $(X_t)$ is drawn in full line, the Markov chain $(\Delta_t)$ is plotted in dotted line. }}
\end{figure}

\begin{figure}[hbtp]
\vspace*{12cm} \protect \includegraphics{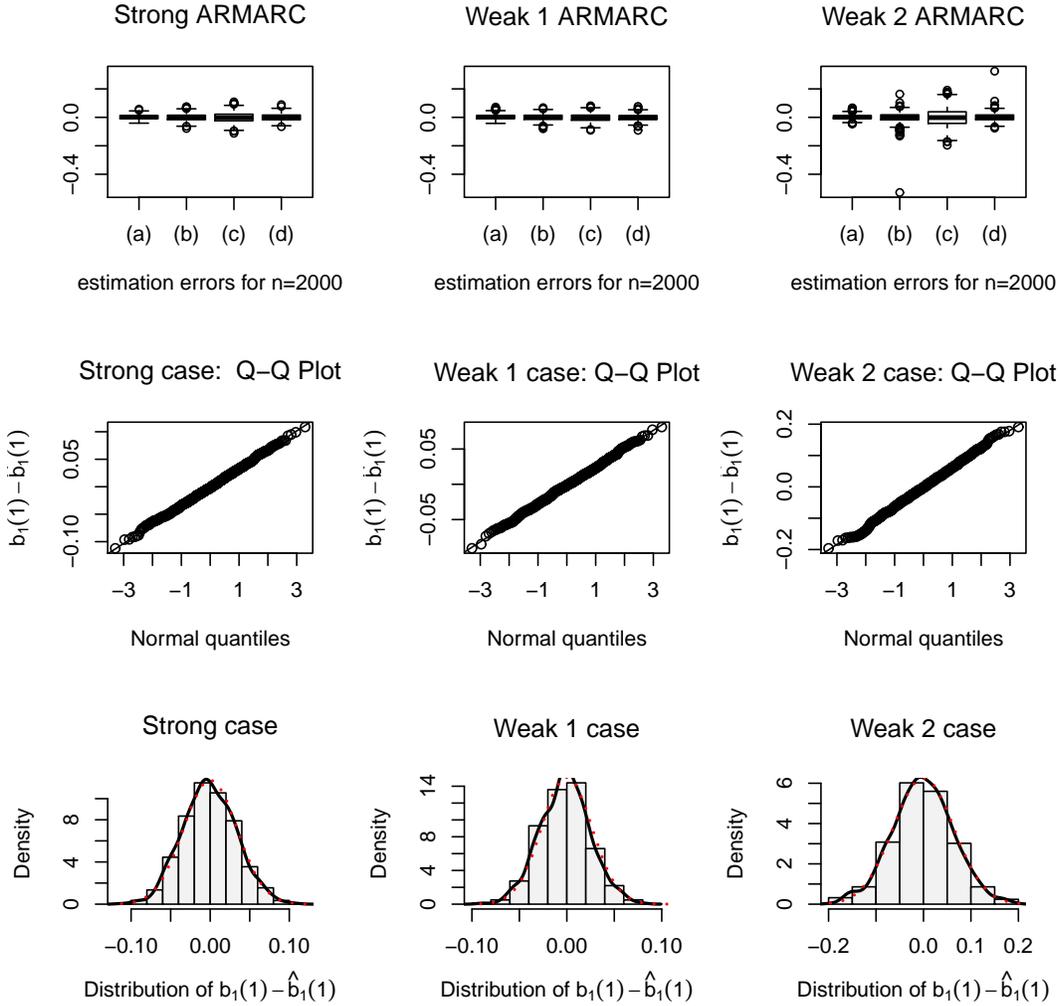} \vspace*{2.5cm}
\caption{\label{fig1} {\footnotesize LSE of $N=1,000$ independent
simulations of the model (\ref{ARMA11MonteCarlo}) with size $n=2,000$ and unknown parameter
$\theta_0=(a_1^0(1),a_1^0(2),b_1^0(1),b_1^0(2))'=(0.90,-0.45,0.10,0.85)'$, when the noise is respectively the strong one defined by (\ref{bruitfort}) (left panel), the weak one defined by \eqref{RT} (middle panel) and the weak one defined by (\ref{PT}) (right panels). Points (a)-(d), in the box-plots of the top panels,  display the distribution of the
 estimation errors $\hat{\theta}(i)-\theta_0(i)$ for $i=1,\dots,4$.
 The panels of the middle present the Q-Q plot of the estimates
 $\hat{\theta}(3)=\hat{b}_1^0(1)$ of the last parameter. The bottom panels  display  the distribution of the same estimates. The kernel density estimate is displayed in full line, and the centered Gaussian density   with the same variance is plotted in dotted line. }}
\end{figure}
\begin{figure}[hbtp]
\vspace*{12cm} \protect \includegraphics{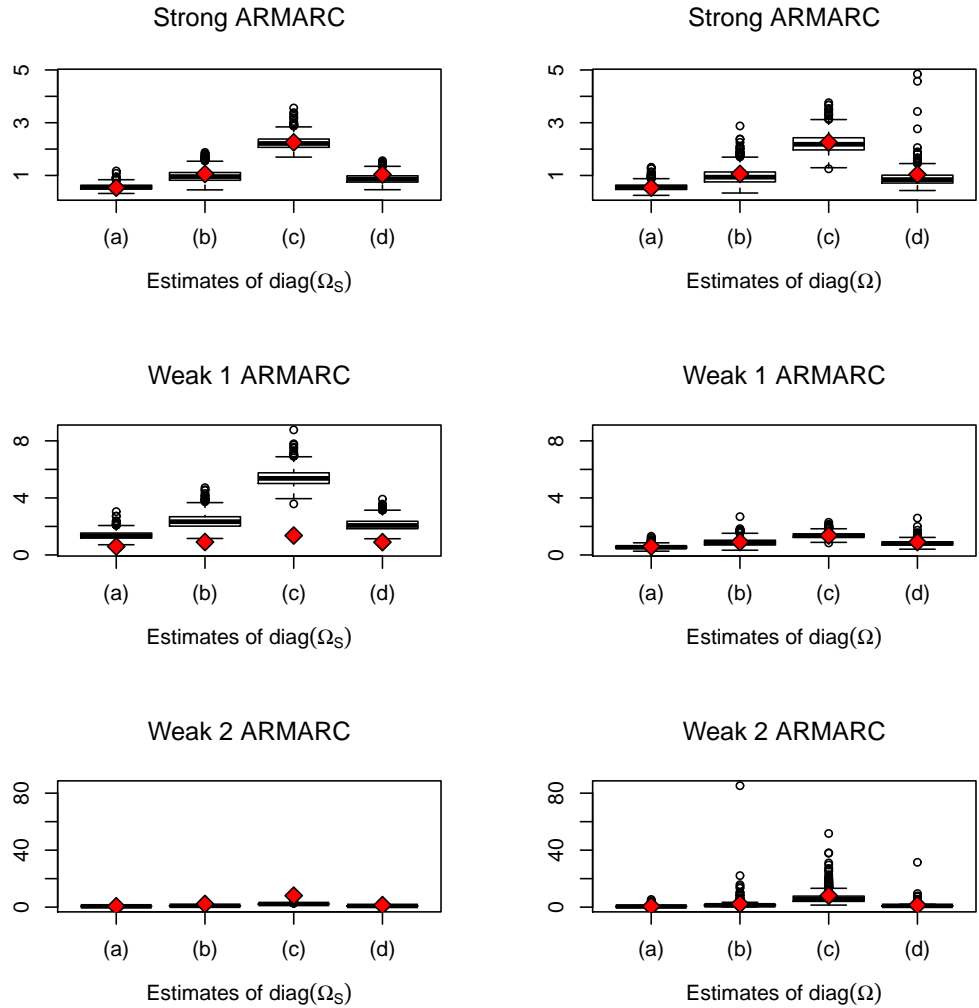} \vspace*{2.5cm}
\caption{\label{fig2} {\footnotesize Comparison of standard and
modified estimates of the asymptotic covariance matrix $\Omega$ of the LSE,
on the simulated models presented in Figure \ref{fig1}.
Weak 1 ARMARC corresponds to Model \eqref{ARMA11MonteCarlo}--\eqref{RT}
and Weak 2 to Model \eqref{ARMA11MonteCarlo}--\eqref{PT}.
The diamond
symbols represent the mean, over the $N=1,000$ replications, of the
standardized squared errors $n\left\{\hat{a}_1^0(1)-0.90\right\}^2$
for (a) (0.54 in the strong case and 0.60 (resp. 0.59) in the weak 1 case (resp. weak 2 case)),
$n\left\{\hat{a}_1^0(2)+0.45\right\}^2$ for (b) (1.06 in the strong
case and 0.91 (resp. 2.24) in the weak 1 case (resp. weak 2 case)), $n\left\{\hat{b}_1^0(1)-0.10\right\}^2$ for (c) (2.25 in the strong case and 1.36 (resp. 8.05) in the weak 1 case (resp. weak 2 case)) and
$n\left\{\hat{b}_1^0(2)-0.85\right\}^2$ for (d) (1.04 in the strong case and 0.90 (resp. 1.41) in the weak 1 case (resp. weak 2 case)).
}}
\end{figure}

\begin{table}[!h]
 \caption{\small{Percentages of rejection of standard $\mathbf{W}_S$ and modified  $\mathbf{W}_M$ Wald tests for testing the null hypothesis $H_0: b_1^0(1)=0$, in the ARMARC$(1,1)$ model (\ref{ARMA11MonteCarlo}). The nominal asymptotic level of the tests is $\alpha=5\%$. The number of replications is $N=1,000$. The line in bold corresponds to the null hypothesis $H_0$.}}
\begin{center}
\begin{tabular}{ ccccccc}
\hline\hline
&\multicolumn{2}{c}{$n=500$}\quad& \multicolumn{2}{c}{$n=2,000$}\quad& \multicolumn{2}{c}{$n=10,000$} \\
$b_1^0(1)$&$\mathrm{\mathbf{W}}_S$&$\mathrm{\mathbf{W}}_M$\quad&$\mathrm{\mathbf{W}}_S$&$\mathrm{\mathbf{W}}_M$\quad&$\mathrm{\mathbf{W}}_S$&$\mathrm{\mathbf{W}}_M$\\
\hline
\\
&\multicolumn{6}{c}{Strong ARMARC-Model \eqref{ARMA11MonteCarlo}--\eqref{bruitfort}}\\
\\
 0.9&100.0&100.0\quad& 100.0&100.0\quad& 100.0&100.0\\
 0.4&100.0&100.0\quad& 100.0&100.0\quad& 100.0&100.0\\
 0.2&84.7 &84.5\quad& 100.0&100.0\quad& 100.0&100.0\\
0.1&34.6& 36.4\quad& 85.5 &85.2\quad& 100.0&100.0\\
\textbf{0.0}&\textbf{5.9} &\textbf{8.6}\quad& \textbf{4.7 }&\textbf{5.2}\quad&\textbf{5.8}& \textbf{6.0}\\
  -0.1&27.4 &29.4\quad& 78.8 &79.2\quad& 100.0&100.0\\
 -0.2&73.6& 74.0\quad& 100.0&100.0\quad& 100.0&100.0\\
 -0.4&99.1 &98.9\quad& 100.0&100.0\quad& 100.0&100.0\\
-0.9&86.7& 86.6\quad& 99.6& 99.6\quad& 100.0&100.0\\
\hline
\\
&\multicolumn{6}{c}{Weak ARMARC-Model \eqref{ARMA11MonteCarlo}--\eqref{RT}}\\
\\
 0.9&100.0&100.0\quad& 100.0&100.0\quad& 100.0&100.0\\
 0.4&99.7 &100.0\quad& 100.0&100.0\quad& 100.0&100.0\\
 0.2&57.4& 96.2\quad& 100.0&100.0\quad& 100.0&100.0\\
 0.1&3.5 &52.4\quad& 50.3 &98.0\quad& 100.0&100.0\\
\textbf{0.0}&\textbf{0.2}& \textbf{5.8}\quad& \textbf{0.0} &\textbf{4.7}\quad&\textbf{0.0}& \textbf{5.6}\\
 -0.1&2.8 &39.5\quad& 37.6 &93.8\quad& 100.0&100.0\\
 -0.2&34.1& 89.6\quad& 99.9 &100.0\quad& 100.0&100.0\\
 -0.4&96.0 &99.6\quad& 100.0&100.0\quad& 100.0&100.0\\
-0.9&86.1 &89.7\quad& 99.7& 99.7\quad& 100.0&100.0\\
\hline
\\
&\multicolumn{6}{c}{Weak ARMARC-Model \eqref{ARMA11MonteCarlo}--\eqref{PT}}\\
\\
 0.9&100.0&100.0\quad& 100.0&100.0\quad& 100.0&100.0\\
 0.4&99.7& 96.9\quad& 100.0&100.0\quad& 100.0&100.0\\
 0.2&86.4& 63.7\quad& 99.6& 92.4\quad& 100.0&100.0\\
 0.1&62.4 &31.5\quad& 85.0 &48.3\quad& 99.8& 92.5\\
\textbf{0.0}&\textbf{46.8}& \textbf{14.1}\quad& \textbf{53.6} &\textbf{9.5}\quad&\textbf{54.2}& \textbf{5.3}\\
 -0.1&60.2& 26.2\quad&84.1 &44.1\quad& 99.9& 92.0\\
 -0.2&80.9 &52.9\quad& 97.8 &87.6\quad& 100.0&99.9\\
 -0.4&98.9& 89.2\quad& 100.0&99.4\quad& 100.0&100.0\\
-0.9&74.0& 67.3\quad& 95.7& 93.1\quad& 100.0&100.0\\
\hline\hline \\
\end{tabular}
\end{center}
\label{wald}
\end{table}

\section{Conclusion}\label{sec:conclusion}
We considered in this paper an ARMA model modulated by an exterior (observed) regime $\{\Delta_t,\ t\ge 0\}$ with possibly dependent errors. Under some technical assumptions, we proved the consistency and the asymptotic normality of the LSE. An efficient weak estimator for the asymptotic covariance matrix has been given. Numerical illustrations corroborate our theoretical results. Some future works include how to extend those results to the case of vector ARMA (VARMA) models, as well as how the existing identification (see \cite{yac_jtsa, BMK2016} ) and diagnostic checking (\cite{BMS2018, yac}) procedures could be adapted to the present model.

\appendix
\section{Proofs}\label{proof}
\subsection{Proofs of Proposition \ref{prop_stationary} and Lemma \ref{c_i^2_expo_decrease}}
{\bf Proof of Proposition \ref{prop_stationary}.}
Let us first note that Condition {\bf (A5a)} is equivalent to
\begin{equation}\label{ineg_expo}
\nbE\left(\sup_{\theta\in \Theta}\left| \left| \prod_{i=1}^t \Phi(\Delta_i,\theta)\right|\right|^{8}\right)\le C \rho^t,\quad \nbE\left( \left|\left|\prod_{i=1}^t \Psi(\Delta_i)\right|\right|^{8}\right)\le C \rho^t,
\end{equation}
for some constant $C>0$ and $0<\rho<1$ (independent from $\theta$). Let us first introduce the processes $(\tilde{Z}_t)_{t\in\nbZ}$ and $(\tilde{\omega}_t)_{t\in\nbZ}$ by
$$
\tilde{Z}_t=(X_t,\dots,X_{t-p+1},\epsilon_t,\dots,\epsilon_{t-q+1})'\in \nbR^{(p+q)\times 1},\quad \tilde{\omega}_t= (\epsilon_t,0,\dots,\epsilon_t,\dots,0)'\in \nbR^{(p+q)\times 1}
$$
where $\epsilon_t$ in the latter is in $(p+1)$th position in $\tilde{\omega}_t$. Then it is clear that we have the following equation for $\tilde{Z}_t$:
$$
\tilde{Z}_t=\Psi(\Delta_{t})\tilde{Z}_{t-1} + \tilde{\omega}_t ,\quad \forall t\in\nbZ ,
$$
of which a candidate for the solution of the above equation is, with the usual convention $\prod_{j=0}^{-1}=1$,
\begin{equation}\label{candidate_eq1}
\tilde{Z}_t= \sum_{k=0}^\infty \prod_{j=0}^{k-1}\Psi(\Delta_{t-j})\tilde{\omega}_{t-k},\quad t\in\nbZ ,
\end{equation}
a stationary process, provided that the series converges, which we prove now. Let us pick for $||\cdot||$ a subordinate norm on the set of matrices. By independence of the processes $(\Delta_t)_{t\in \nbZ}$ and $(\epsilon_t)_{t\in \nbZ}$, and using the fact that the latter is square integrable, we easily get, for $k\ge 1$,
\begin{multline*}
\nbE\left( \left|\left|\Psi(\Delta_{t})\dots\Psi(\Delta_{t-k+1})\tilde{\omega}_{t-k}\right|\right|^2\right)\le \nbE\left( \left|\left|\Psi(\Delta_{t})\dots\Psi(\Delta_{t-k+1}) \right|\right|^2 .\left|\left|\tilde{\omega}_{t-k}\right|\right|^2\right)\\
=\nbE\left( \left|\left|\Psi(\Delta_{t})\dots\Psi(\Delta_{t-k+1}) \right|\right|^2\right) \nbE\left(\left|\left|\tilde{\omega}_{t-k}\right|\right|^2\right)\le C \nbE\left(\left|\left|\tilde{\omega}_{0}\right|\right|^2\right) \rho^k,
\end{multline*}
the last inequality stemming from (\ref{ineg_expo}), so that series (\ref{candidate_eq1}) converges in $L^2$. Note that we prove that $\tilde{Z}_t$ (hence $X_t$) is in $L^4$ by replacing $||\cdot||^2$ by $||\cdot||^4$ in the above inequalities,
using again (\ref{ineg_expo}) and the fact that $(\epsilon)_{t\in\nbZ}$ is in $ L^4$, see assumption ({\bf A3}). Similarly, defining
\begin{equation}\label{def_Z_t_theta}
Z_t(\theta):=(\epsilon_t(\theta),\dots,\epsilon_{t-q+1}(\theta),X_t,\dots,X_{t-p+1})' , \quad \omega_t= (X_t,0,\dots,X_t,\dots,0)'
\end{equation}
where $X_t$ in the latter is in $(q+1)$th position, we also get that $Z_t(\theta)$ satisfies
$$
Z_t(\theta)=\Phi(\Delta_{t},\theta)Z_{t-1}(\theta)+\omega_t .
$$
A solution candidate to the above solution is
\begin{equation}\label{candidate_eq2}
Z_t(\theta)= \sum_{k=0}^\infty \prod_{j=0}^{k-1}\Phi(\Delta_{t-j},\theta)\omega_{t-k},\quad t\in\nbZ .
\end{equation}
Similarly to the proof leading to (\ref{candidate_eq1}), convergence of (\ref{candidate_eq2}) is obtained thanks to (\ref{ineg_expo}) as well as stationarity of $(X_t)_{t\in \nbZ}$ and the fact that $X_t\in L^4$.

We check that $\omega_t=M\tilde{Z}_t$ and $\epsilon_t(\theta)=w_1Z_t(\theta)$, which, plugged into (\ref{candidate_eq1}) and (\ref{candidate_eq2}) yields (\ref{dse_eps_theta}) with coefficients $c_i(\theta,\Delta_{t},\dots,\Delta_{t-i+1})$ given by (\ref{expression_eps_theta}). Finally, let us verify that $(c_i(\theta,\Delta_{t},\dots,\Delta_{t-i+1}))_{i\in\nbN}$ is the unique sequence verifying (\ref{dse_eps_theta}). Let us then pick a sequence of r.v. $(d_i)_{i\in\nbN}$ in ${\cal H}$ such that $\epsilon_t(\theta)= \sum_{i=0}^\infty c_i(\theta,\Delta_{t},\dots,\Delta_{t-i+1})\epsilon_{t-i}= \sum_{i=0}^\infty d_i \epsilon_{t-i}$. We then get, by independence from $ (\epsilon_t)_{t\in\nbZ}$ as well as by the fact that the latter is a weak white noise:
$$
0=\nbE\left( \left[\sum_{i=0}^\infty (c_i(\theta,\Delta_{t},\dots,\Delta_{t-i+1})-d_i)\epsilon_{t-i}\right]^2\right)=\sigma^2\nbE\left( \sum_{i=0}^\infty (c_i(\theta,\Delta_{t},\dots,\Delta_{t-i+1})-d_i)^2\right)
$$
hence $(c_i(\theta,\Delta_{t},\dots,\Delta_{t-i+1}))_{i\in\nbN}=(d_i)_{i\in\nbN}$ a.s.\zak

{\bf Proof of Lemma \ref{c_i^2_expo_decrease}.}
The fact that the $\theta\mapsto c_i(\theta,\Delta_{t},\dots,\Delta_{t-i+1})$, $\theta\mapsto \nabla [c_i(\theta,\Delta_{t},\dots,\Delta_{t-i+1})]^2$ and $\theta\mapsto \nabla^2 [c_i(\theta,\Delta_{t},\dots,\Delta_{t-i+1})]^2$ are polynomial functions (of several variables) can be verified easily using the fact that, for all $s\in {\cal S}$, $\theta \mapsto  \Phi(s,\theta)$ and $\theta \mapsto  \Psi(\theta)$ are affine functions. We turn to (\ref{coeff_c_i_expo_decrease}). Using Minkovski's inequality, the fact that the matrix norm $||\cdot||$ is submultiplicative entails
\begin{multline}
\left|\left|\sup_{\theta\in \Theta} |c_i(\theta,\Delta_{i},\dots,\Delta_{1})|\right|\right|_{2\nu+4} \le   \sum_{k=0}^i \left|\left| \sup_{\theta\in \Theta} |w_1\Phi(\Delta_{i},\theta)\dots\Phi(\Delta_{i-k+1},\theta)M \Psi(\Delta_{i-k})\dots\Psi(\Delta_{1})w_{p+1}'|\right|\right|_{2\nu+4}\\
\le  C \sum_{k=0}^i \left[\nbE\left(\sup_{\theta\in \Theta}\left| \left| \Phi(\Delta_{i},\theta)\dots\Phi(\Delta_{i-k+1},\theta) \right|\right|^{2\nu+4} \left| \left| \Psi(\Delta_{i-k})\dots\Psi(\Delta_{1})\right|\right|^{2\nu+4}\right)\right]^{1/(2\nu+4)}
\label{lemma_coeff_c_i_expo_decrease0}
\end{multline}
for some constant $C>0$. The Cauchy-Schwarz inequality as well as {\bf (A5a)} yields
\begin{multline*}
\left[\nbE\left(\sup_{\theta\in \Theta}\left| \left| \Phi(\Delta_{i},\theta)\dots\Phi(\Delta_{i-k+1},\theta) \right|\right|^{2\nu+4} \left| \left| \Psi(\Delta_{i-k})\dots\Psi(\Delta_{1})\right|\right|^{2\nu+4}\right)\right]^{1/(2\nu+4)}\\
\le \left[\nbE\left(\sup_{\theta\in \Theta}\left| \left| \Phi(\Delta_{i},\theta)\dots\Phi(\Delta_{i-k+1},\theta) \right|\right|^{4\nu+8} \right)\right]^{\frac{1}{(4\nu+8 )}} \left[\nbE\left(  \left| \left| \Psi(\Delta_{i-k})\dots\Psi(\Delta_{1})\right|\right|^{4\nu+8} \right)\right]^{\frac{1}{(4\nu+8 )}}\le \kappa \rho^{\frac{i}{(2\nu+4 )}}
\end{multline*}
which, plugged in (\ref{lemma_coeff_c_i_expo_decrease0}), yields inequality (\ref{coeff_c_i_expo_decrease}) for $c_i(\theta,\Delta_{i},\dots,\Delta_{1})$. The inequalities for $\nabla^j [c_i(\theta,\Delta_{i},\dots,\Delta_{1})]$, $j=2,3$, are proved similarly. As to $c_i^e(t,\theta,\Delta_{t},\dots,\Delta_{t-i+1})$, \eqref{expression_e_t_theta} yields the upper bound
\begin{multline*}
\left|\left|\sup_{\theta\in \Theta} |c_i^e(t,\theta,\Delta_{t},\dots,\Delta_{t-i+1})|\right|\right|_{2\nu+4} \\
\le   \sum_{k=0}^i \left|\left| \sup_{\theta\in \Theta} |w_1\Phi(\Delta_{t},\theta)\dots\Phi(\Delta_{t-k+1},\theta)M \Psi(\Delta_{t-k})\dots\Psi(\Delta_{t-i+1})w_{p+1}'|\right|\right|_{2\nu+4},
\end{multline*}
so that upper bound \eqref{coeff_c_i_e_expo_decrease} for $c_i^e(t,\theta,\Delta_{t-1},\dots,\Delta_{t-i})$ follows again by a Cauchy-Schwarz argument. The upper bound \eqref{coeff_c_i_e_expo_decrease} for $\nabla c_i^e(t,\theta,\Delta_{t-1},\dots,\Delta_{t-i})$ is obtained similarly.\zak 
\subsection{Proofs of Lemma \ref{lemme_prelim} and Proposition \ref{prop_prelim}}\label{appendix_proof_lemma_prop_prelim}
{\bf Proof of Lemma \ref{lemme_prelim}.}
We first prove Point 1. Using decomposition \eqref{dse_eps_theta} of $\epsilon_t(\theta)$, independence of the white noise from the modulating process, as well as stationarity of the former, we obtain
$$
 \left|\left| \sup_{\theta\in \Theta} |\epsilon_0(\theta)|\right| \right|_4\le \sum_{i=0}^\infty \left|\left|\sup_{\theta\in \Theta} |c_i(\theta,\Delta_{t},\dots,\Delta_{t-i+1})|\right|\right|_4. ||\epsilon_{0}||_4
$$
which is a converging series because of \eqref{coeff_c_i_expo_decrease}. As to  $e_t(\theta)$, we use this time decomposition \eqref{dse_e_t_theta} as well as \eqref{coeff_c_i_e_expo_decrease} in order to get
$$
 \sup_{t\ge 0}\left|\left| \sup_{\theta\in \Theta} |e_t(\theta)|\right| \right|_4\le \sum_{i=0}^\infty \sup_{t\ge 0}\left|\left|\sup_{\theta\in \Theta}|c_i^e(\theta,\Delta_{t},\dots,\Delta_{t-i+1})|\right|\right|_4. ||\epsilon_{0}||_4<+\infty .
$$
In order to prove Point 2, we remind the following notations. From (\ref{def_inov_theta}) and (\ref{e_t}),
we have
$$Z_t(\theta)=\omega_t+ \Phi(\Delta_{t},\theta)Z_{t-1}(\theta)\qquad \forall t\in \mathbb{Z},$$
and
$$Z^e_t(\theta)=\omega^e_t+ \Phi(\Delta_{t},\theta)Z^e_{t-1}(\theta)\qquad  t=1,\dots,n,$$
where $Z^e_t(\theta):=(e_t(\theta),\dots,e_{t-q+1}(\theta),\tilde{X}_t,\dots,\tilde{X}_{t-p+1})' , \quad \omega^e_t= (\tilde{X}_t,0,\dots,\tilde{X}_t,\dots,0)',
$ so that
$\omega^e_t=\omega_t$ for $t\geq r+1$ (where $r=\max(p,q)$), $\omega^e_t(\theta)=0_{p+q}$ for $t\leq 0$. We recall that the processes $(\tilde{X}_t)_{t\in \nbZ}$ and $(e_t(\theta))_{t\in \nbZ}$ verify \eqref{e_t}.
Note that $\left|\left| \sup_{\theta\in \Theta} |\epsilon_t(\theta)-e_t(\theta)|\right| \right|_2\longrightarrow 0$ is equivalent to $\left|\left| \sup_{\theta\in \Theta} ||Z^e_t(\theta)-Z_t(\theta)||\right| \right|_2\longrightarrow 0$ as $t\to \infty$. Now, since $\tilde{X}_t=X_t$ for $t\ge 1$, we easily see that
\begin{align}\label{eq_Z^e_t}
Z^e_t(\theta)-Z_t(\theta)&=\Phi(\Delta_{t},\theta)[Z^e_{t-1}(\theta)-Z_{t-1}(\theta)],\quad \forall t\ge r+1,\\
\label{eq_Z^e_r}
Z^e_t(\theta)-Z_t(\theta)&=\omega^e_t-\omega_t+\Phi(\Delta_{t},\theta)[Z^e_{t-1}(\theta)-Z_{t-1}(\theta)],\mbox{ for }t=1,\dots,r.
\end{align}
Now, using \eqref{eq_Z^e_t} and \eqref{eq_Z^e_r} we obtain
\begin{eqnarray}
  &&Z^e_t(\theta)-Z_t(\theta) =\prod_{j=0}^{t-r-1}\Phi(\Delta_{t-j},\theta)[Z^e_{r}(\theta)-Z_{r}(\theta)],\quad \forall t\ge r+1,\nonumber\\
   &=& \prod_{j=0}^{t-r-1}\Phi(\Delta_{t-j},\theta)\left(\sum_{i=0}^{r-1}\prod_{j=0}^{i-1}\Phi(\Delta_{r-j},\theta)[\omega^e_{r-i}-\omega_{r-i}]
   \prod_{j=0}^{r-1}\Phi(\Delta_{r-j},\theta)\omega_{0} \right). \label{differenceZ}
\end{eqnarray}
Let us furthermore note that $$
\left|\left|\sup_{\theta\in \Theta} |\tilde{X}_t-X_t|\right|\right|_4= \left|\left|\sup_{\theta\in \Theta} |\sum_{i=t}^{r} g_i^a(\Delta_t,\theta){X}_{t-i}+\sum_{j=t}^{r} g_j^b(\Delta_t,\theta)\epsilon_{t-i}(\theta)|\right|\right|_4<+\infty
\mbox{ for }t=1,\dots,r
$$
as indeed $X_t\in L^4$ (as proved in the proof of Proposition \ref{prop_stationary}) and $|| \sup_{\theta\in\Theta} \epsilon_t(\theta)||_4 <+\infty$  as proved in Point 1.
In view of \eqref{differenceZ}, using Minkowski's and H{\"o}lder's inequalities and {\bf (A5a)}, we thus have
$$
\left|\left|\sup_{\theta\in \Theta} ||Z^e_t(\theta)-Z_t(\theta)||\right|\right|_2\le C\rho^t,
$$
for some constant $C>0$ and $0<\rho<1$ (independent from $\theta$).
\\
Let us turn to Point 3. This is due to
$$
\PP\left(t^\alpha \sup_{\theta\in \Theta} |\epsilon_t(\theta)-e_t(\theta)|>\eta \right)\le \frac{t^{2+2\alpha}\left|\left| \sup_{\theta\in \Theta} |\epsilon_t(\theta)-e_t(\theta)|\right| \right|_2^2 }{t^{2} \eta^2}= o\left( \frac{1}{t^{2}}\right),\quad \forall \eta >0 ,
$$
the last equality thanks to Point 2, and using Borel-Cantelli's lemma.\\
We now turn to Point 4. The fact that $\left|\left| \sup_{\theta\in \Theta} ||\nabla^j\epsilon_0(\theta)||\right| \right|_4$ and $\sup_{t\ge 0}\left|\left| \sup_{\theta\in \Theta} ||\nabla^j e_t(\theta)||\right| \right|_4$ are finite is proved similarly to Point 1 and using estimates \eqref{coeff_c_i_expo_decrease} and \eqref{coeff_c_i_e_expo_decrease}.
We then pass on to the limit of $t^\alpha\left|\left| \sup_{\theta\in \Theta} ||\nabla (e_t- \epsilon_t)(\theta)||\right| \right|_{4/3}$ as $t\to\infty$. Let $i\in {\cal S}$. Deriving \eqref{eq_Z^e_t} with respect to $\theta_i$ yields
\begin{equation}\label{proof_lemma_prelim_derive}
\frac{\partial}{\partial \theta_i}[Z^e_t(\theta)-Z_t(\theta)]=\Phi(\Delta_{t},\theta)\frac{\partial}{\partial \theta_i}[Z^e_{t-1}(\theta)-Z_{t-1}(\theta)] + \frac{\partial}{\partial \theta_i}\Phi(\Delta_{t},\theta) [Z^e_{t-1}(\theta)-Z_{t-1}(\theta)],\quad \forall t\ge p+1,
\end{equation}
hence we may write
$$
\frac{\partial}{\partial \theta_i}[Z^e_t(\theta)-Z_t(\theta)]=\sum_{k=0}^{t-p}\prod_{j=0}^{k-1}\Phi(\Delta_{t-j},\theta) \frac{\partial}{\partial \theta_i}\Phi(\Delta_{t-k},\theta) [Z^e_{t-k}(\theta)-Z_{t-k}(\theta)],
$$
hence, using Minkovski's and H{\"o}lder's inequalities, and letting $M_\Phi:=\max_{s\in {\cal S},\theta\in \Theta}\left| \frac{\partial}{\partial \theta_i}\Phi(s,\theta)\right| $, we get
\begin{multline}\label{lemm_prelim_ineq_derive_proof}
t^\alpha \left|\left| \sup_{\theta\in \Theta} ||\frac{\partial}{\partial \theta_i}[Z^e_t(\theta)-Z_t(\theta)]||\right|\right|_{{8/5}} \le M_\Phi \sum_{k=0}^{t-p} \left|\left|\sup_{\theta\in \Theta} |\prod_{j=0}^{k-1}\Phi(\Delta_{t-j},\theta)| \right|\right|_{{8}}\\
. t^\alpha \left|\left|\ \sup_{\theta\in \Theta}|| Z^e_{t-k}(\theta)-Z_{t-k}(\theta)|| \right|\right|_2 .
\end{multline}
Now, since $\left|\left|\sup_{\theta\in \Theta} ||\prod_{j=0}^{k-1}\Phi(\Delta_{t-j},\theta)|| \right|\right|_{{8}}\le \kappa \rho^k$ for some $\kappa>0$ and $\rho<1$ thanks to {\bf (A5a)}, and since $t^\alpha \left|\left|\ \sup_{\theta\in \Theta}|| Z^e_{t-k}(\theta)-Z_{t-k}(\theta)|| \right|\right|_2 $ is uniformly bounded in $t$ and $k\le t$, and tends to $0$ as $t\to \infty$, the dominated convergence theorem yields that $t^\alpha \left|\left| \sup_{\theta\in \Theta} ||\frac{\partial}{\partial \theta_i}[Z^e_t(\theta)-Z_t(\theta)]||\right|\right|_{{8/5}} \longrightarrow 0$ as $t\to\infty$, proving $t^\alpha \left|\left| \sup_{\theta\in \Theta} ||\nabla (e_t- \epsilon_t)(\theta)||\right| \right|_{{8/5}}\longrightarrow 0$ as $t\to\infty$ in Point 4. Let us now prove that $t^\alpha\left|\left| \sup_{\theta\in \Theta} ||\nabla^2 (e_t- \epsilon_t)(\theta)||\right| \right|_{{4/3}}\longrightarrow 0$. Deriving again \eqref{proof_lemma_prelim_derive} with respect to $\theta_\ell$, $\ell\in {\cal S}$, we obtain
\begin{multline}\label{lemm_prelim_proof_derive2}
\frac{\partial^2}{\partial \theta_\ell\partial \theta_i}[Z^e_t(\theta)-Z_t(\theta)]=\Phi(\Delta_{t},\theta)\frac{\partial^2}{\partial \theta_\ell\partial \theta_i}[Z^e_{t-1}(\theta)-Z_{t-1}(\theta)]+ \frac{\partial}{\partial \theta_\ell}\Phi(\Delta_{t},\theta) \frac{\partial}{\partial \theta_i}[Z^e_{t-1}(\theta)-Z_{t-1}(\theta)]\\
+\frac{\partial}{\partial \theta_i}\Phi(\Delta_{t},\theta) \frac{\partial}{\partial \theta_\ell}[Z^e_{t-1}(\theta)-Z_{t-1}(\theta)] + \frac{\partial^2}{\partial \theta_\ell\partial \theta_i}\Phi(\Delta_{t},\theta) [Z^e_{t-1}(\theta)-Z_{t-1}(\theta)],\quad \forall t\ge p+1,
\end{multline}
so that, in the same spirit as \eqref{proof_lemma_prelim_derive}, we obtain
\begin{multline}\label{ineg_derive}
t^\alpha \left|\left| \sup_{\theta\in \Theta} ||\frac{\partial^2}{\partial \theta_\ell\partial \theta_i}[Z^e_t(\theta)-Z_t(\theta)]||\right|\right|_{{4/3}} \le M_\Phi' \sum_{k=0}^{t-p} \left|\left|\sup_{\theta\in \Theta} |\prod_{j=0}^{k-1}\Phi(\Delta_{t-j},\theta)| \right|\right|_{{8}}\\
.t^\alpha \left[ \left|\left|\ \sup_{\theta\in \Theta}|| Z^e_{t-k}(\theta)-Z_{t-k}(\theta)|| \right|\right|_{{8/5}}+\left|\left|\ \sup_{\theta\in \Theta}|| \frac{\partial}{\partial \theta_\ell}[Z^e_{t-k}(\theta)-Z_{t-k}(\theta)]|| \right|\right|_{{8/5}}\right.\\
\left. +\left|\left|\ \sup_{\theta\in \Theta}|| \frac{\partial}{\partial \theta_i}[Z^e_{t-k}(\theta)-Z_{t-k}(\theta)]|| \right|\right|_{{8/5}} \right],
\end{multline}
for some positive constant $M_\Phi'$. Using Point 2 (so that $t^\alpha \left|\left|\ \sup_{\theta\in \Theta}|| Z^e_{t-k}(\theta)-Z_{t-k}(\theta)|| \right|\right|_{{8/5}}$ tends to $0$ as $t\to\infty$, since $8/5<2$) and the previous estimate $$t^\alpha \left|\left| \sup_{\theta\in \Theta} ||\frac{\partial}{\partial \theta_i}[Z^e_t(\theta)-Z_t(\theta)]||\right|\right|_{{8/5}} \longrightarrow 0$$ for all $i\in {\cal S}$, we conclude by a dominated convergence theorem that $$t^\alpha \left|\left| \sup_{\theta\in \Theta} ||\frac{\partial^2}{\partial \theta_\ell\partial \theta_i}[Z^e_t(\theta)-Z_t(\theta)]||\right|\right|_{{4/3}},\text{ hence }t^\alpha\left|\left| \sup_{\theta\in \Theta} ||\frac{\partial^2}{\partial \theta_\ell\partial \theta_i} (e_t- \epsilon_t)(\theta)||\right| \right|_{{4/3}},$$ tends to $0$.\\
We finish by sketching the proof leading to $t^\alpha\left|\left| \sup_{\theta\in \Theta} ||\nabla^3 (e_t- \epsilon_t)(\theta)||\right| \right|_{1}\longrightarrow 0$. The starting point is again deriving \eqref{lemm_prelim_proof_derive2} with respect to $\theta_{\ell '}$, $\ell '\in {\cal S}$, which yields, as in \eqref{ineg_derive}, the following estimate:
\begin{multline*}
t^\alpha \left|\left| \sup_{\theta\in \Theta} ||\frac{\partial^3}{\partial \theta_\ell' \partial \theta_\ell\partial \theta_i}[Z^e_t(\theta)-Z_t(\theta)]||\right|\right|_{1} \le M_\Phi'' \sum_{k=0}^{t-p} \left|\left|\sup_{\theta\in \Theta} |\prod_{j=0}^{k-1}\Phi(\Delta_{t-j},\theta)| \right|\right|_{8}\\
.t^\alpha \left[ \left|\left|\ \sup_{\theta\in \Theta}|| Z^e_{t-k}(\theta)-Z_{t-k}(\theta)|| \right|\right|_{4/3}+\left|\left|\ \sup_{\theta\in \Theta}|| \frac{\partial}{\partial \theta_\ell}[Z^e_{t-k}(\theta)-Z_{t-k}(\theta)]|| \right|\right|_{4/3}\right.\\
+\left|\left|\ \sup_{\theta\in \Theta}|| \frac{\partial}{\partial \theta_i}[Z^e_{t-k}(\theta)-Z_{t-k}(\theta)]|| \right|\right|_{4/3}  +\left|\left|\ \sup_{\theta\in \Theta}|| \frac{\partial^2}{\partial \theta_\ell\partial \theta_i}[Z^e_{t-k}(\theta)-Z_{t-k}(\theta)]|| \right|\right|_{4/3} \\
\left. + \left|\left|\ \sup_{\theta\in \Theta}|| \frac{\partial^2}{\partial \theta_\ell'\partial \theta_i}[Z^e_{t-k}(\theta)-Z_{t-k}(\theta)]|| \right|\right|_{4/3}+\left|\left|\ \sup_{\theta\in \Theta}|| \frac{\partial^2}{\partial \theta_\ell'\partial \ell}[Z^e_{t-k}(\theta)-Z_{t-k}(\theta)]|| \right|\right|_{4/3}\right],
\end{multline*}
for some constant $M_\Phi''$, so that we conclude similarly.\zak
{\bf Proof of Proposition \ref{prop_prelim}.} In this proof, $C$ will denote a generic positive constant that will change from line to line. Let us start with Point 1. The fact that $Q_n(\theta)$ converges a.s. to $O_\infty(\theta)=\EE(\epsilon_0(\theta))$ as $n\to \infty$ is a consequence of the fact that $ \sup_{\theta\in \Theta} |\epsilon_t(\theta)-e_t(\theta)|^2\longrightarrow 0$ (itself a consequence of Point 3 of Lemma \ref{lemme_prelim}) and is justified by the same exact proof of Lemma 7 in \cite{fz98}. We now prove that $n^\alpha \left|\left| \sup_{\theta\in \Theta} |Q_n(\theta)-O_n(\theta)|\right| \right|_1 $. Let $\alpha\in (0,1)$. Using the upper bound $\sup_{\theta\in \Theta}|e_t(\theta)^2-\epsilon_t(\theta)^2|\le   \left[ \sup_{\theta\in \Theta} |e_t(\theta)|  + \sup_{\theta\in \Theta} |\epsilon_t(\theta)|\right] .\sup_{\theta\in \Theta} |e_t(\theta)- \epsilon_t(\theta)|$, as well as Cauchy-Schwarz and Minkovski's inequalities, we get the following
$$
n^\alpha \left|\left| \sup_{\theta\in \Theta} |Q_n(\theta)-O_n(\theta)|\right| \right|_1 \le \frac{1}{n^{1-\alpha}}\sum_{t=1}^n \left[ \left|\left|\sup_{\theta\in \Theta} |e_t(\theta)| \right| \right|_2 + \left|\left|\sup_{\theta\in \Theta} |\epsilon_t(\theta)|\right| \right|_2\right] .\left|\left|\sup_{\theta\in \Theta} |e_t(\theta)- \epsilon_t(\theta)| \right| \right|_2 .
$$
Since $ \left|\left|\sup_{\theta\in \Theta} |e_t(\theta)| \right| \right|_2$ is upper bounded by Point 1 of Lemma \ref{lemme_prelim}, and $\left|\left|\sup_{\theta\in \Theta} |\epsilon_t(\theta)|\right| \right|_2$ is constant in $t$ and finite, there thus exists some constant $C>0$ such that
\begin{equation}\label{demo_Prop_prelim1}
n^\alpha \left|\left| \sup_{\theta\in \Theta} |Q_n(\theta)-O_n(\theta)|\right| \right|_1 \le C \frac{1}{n^{1-\alpha}}\sum_{t=1}^n  \left|\left|\sup_{\theta\in \Theta} |e_t(\theta)- \epsilon_t(\theta)| \right| \right|_2 .
\end{equation}
Let us write the right hand side of the above inequality in the form $\frac{1}{n^{1-\alpha}}\sum_{t=1}^n   [t^{1-\alpha}-(t-1)^{1-\alpha}] \frac{1}{t^{1-\alpha}-(t-1)^{1-\alpha}} \left|\left|\sup_{\theta\in \Theta}|e_t(\theta)- \epsilon_t(\theta)| \right| \right|_2 $. Since
$$
\frac{1}{t^{1-\alpha}-(t-1)^{1-\alpha}} \left|\left|\sup_{\theta\in \Theta}|e_t(\theta)- \epsilon_t(\theta)| \right| \right|_2\sim_{t\to\infty } \frac{1}{(1-\alpha)t^{-\alpha}}\left|\left|\sup_{\theta\in \Theta}|e_t(\theta)- \epsilon_t(\theta)| \right| \right|_2,
$$
which tends to $0$ as $t\to \infty$ (a consequence of Point 2 of Lemma \ref{lemme_prelim}), Toeplitz's lemma implies that the right hand side of \eqref{demo_Prop_prelim1} tends to $0$ as $n\to \infty$, and this proves Point 1.\\
We now prove Point 2. We have for all $\theta\in\Theta$
\begin{multline}\label{derive_inequality}
||\nabla [e_t(\theta)^2-\epsilon_t(\theta)^2]||=|| 2 e_t(\theta) \nabla[e_t(\theta)-\epsilon_t(\theta)]+ 2[e_t(\theta)-\epsilon_t(\theta)] \nabla \epsilon_t(\theta)||\\
\le 2 || e_t(\theta) \nabla[e_t(\theta)-\epsilon_t(\theta)]|| + 2|e_t(\theta)- \epsilon_t(\theta)|.|| \nabla \epsilon_t(\theta) ||.
\end{multline}
so that
\begin{multline}\label{Point2_ineq1}
\sup_{\theta\in \Theta} ||\nabla(Q_n(\theta)-O_n(\theta))||\le \frac{1}{n}\sum_{t=1}^n \sup_{\theta\in \Theta} |e_t(\theta)|.\sup_{\theta\in \Theta}|| \nabla[e_t(\theta)-\epsilon_t(\theta)]|| \\
+ \frac{1}{n}\sum_{t=1}^n \sup_{\theta\in \Theta} |e_t(\theta)-\epsilon_t(\theta)|. \sup_{\theta\in \Theta}|| \nabla\epsilon_t(\theta)||.
\end{multline}
Lemma \ref{lemme_prelim}, Points 2 and 4, along with Borel-Cantelli's lemma, yields that $\sup_{\theta\in \Theta} |\epsilon_t(\theta)-e_t(\theta)|$ and $\sup_{\theta\in \Theta} ||\nabla(\epsilon_t-e_t)(\theta)||$ a.s. tend to $0$ as $t\to\infty$. The second term on the right hand side of \eqref{Point2_ineq1} if then a.s. upper bounded thanks to Cauchy-Scwharz inequality by
$$
 \left[ \frac{1}{n}\sum_{t=1}^n \sup_{\theta\in \Theta} |e_t(\theta)-\epsilon_t(\theta)|^2\right]^{1/2}.\left[ \frac{1}{n}\sum_{t=1}^n  \sup_{\theta\in \Theta}|| \nabla\epsilon_t(\theta)||^2\right]^{1/2},
$$
which tends to zero thanks to Cesaro's Lemma and the ergodic theorem. And since, by Minkowski's inequality,
$$
\left[ \frac{1}{n}\sum_{t=1}^n  \sup_{\theta\in \Theta}|e_t(\theta)|^2\right]^{1/2}\le \left[ \frac{1}{n}\sum_{t=1}^n  \sup_{\theta\in \Theta}|e_t(\theta)-\epsilon_t(\theta)|^2\right]^{1/2} + \left[ \frac{1}{n}\sum_{t=1}^n  \sup_{\theta\in \Theta}|\epsilon_(\theta)|^2\right]^{1/2},
$$
we have that $\left[ \frac{1}{n}\sum_{t=1}^n  \sup_{\theta\in \Theta}|e_t(\theta)|^2\right]^{1/2}$ is a.s. upper bounded in $n\ge 1$, again by a Cesaro and ergodic theorem argument. The first term on the right hand side of \eqref{Point2_ineq1} if then again a.s. upper bounded thanks to Cauchy-Scwharz inequality by
$$
\left[ \frac{1}{n}\sum_{t=1}^n \sup_{\theta\in \Theta} ||\nabla(e_t-\epsilon_t)(\theta)||^2\right]^{1/2}.\left[ \frac{1}{n}\sum_{t=1}^n  \sup_{\theta\in \Theta}| e_t(\theta)|^2\right]^{1/2},
$$
which tends to zero as $t\to\infty$. Hence \eqref{Point2_ineq1} implies that $\sup_{\theta\in \Theta} ||\nabla(Q_n(\theta)-O_n(\theta))||$ a.s. tends to $0$ as $n\to\infty$. Proof of a.s. convergence of $\sup_{\theta\in \Theta} ||\nabla^j(Q_n(\theta)-O_n(\theta))||$ to $0$ for $j=2,3$ is obtained similarly, using arguments related to Points 3 and 4 from Lemma \ref{lemme_prelim}.
\\
Let us now prove Point 3. Let $\alpha\in (0,1)$. We deduce from \eqref{derive_inequality}, using Minkowski and H{\"o}lder inequalities, that
\begin{multline}\label{demo_Prop_prelim2}
n^\alpha \left|\left| \sup_{\theta\in \Theta} ||\nabla(Q_n(\theta)-O_n(\theta))||\right| \right|_1 \le \frac{C}{n^{1-\alpha}}\sum_{t=1}^n \left|\left|\sup_{\theta\in \Theta} |e_t(\theta)| \right| \right|_4 \left|\left|\sup_{\theta\in \Theta}|| \nabla[e_t(\theta)-\epsilon_t(\theta)]|| \right| \right|_{4/3}  \\
+ \frac{C}{n^{1-\alpha}}\sum_{t=1}^n \left|\left|\sup_{\theta\in \Theta} |e_t(\theta)-\epsilon_t(\theta)| \right| \right|_2 \left|\left|\sup_{\theta\in \Theta}|| \nabla\epsilon_t(\theta)|| \right| \right|_{2} .
\end{multline}
Using Point 1 of Lemma \ref{lemme_prelim}, we have that $\left|\left|\sup_{\theta\in \Theta} |e_t(\theta)| \right| \right|_4$ is upper bounded by some constant $C$. The first term in the righthandside of \eqref{demo_Prop_prelim2} may thus be upper bounded by
$$
C\frac{1}{n^{1-\alpha}}\sum_{t=1}^n  \left|\left|\sup_{\theta\in \Theta}|| \nabla[e_t(\theta)-\epsilon_t(\theta)]|| \right| \right|_{4/3}.
$$
Noting that $\left|\left|\sup_{\theta\in \Theta}|| \nabla[e_t(\theta)-\epsilon_t(\theta)]|| \right| \right|_{4/3}\le C' \left|\left|\sup_{\theta\in \Theta}|| \nabla[e_t(\theta)-\epsilon_t(\theta)]|| \right| \right|_{8/5}$ for some constant $C'$, the above expression is, similarly to the argument in \eqref{demo_Prop_prelim1}, a quantity that tends to $0$ as $n\to\infty$ thanks to Point 4 in Lemma \ref{lemme_prelim} coupled with Toeplitz's lemma. Hence the first term in the right hand side of \eqref{demo_Prop_prelim2}  tends to $0$ as $n\to \infty$. Again using Point 1 and Point 2 of the same lemma, and with the same argument, we also have that the second term in the right hand side of \eqref{demo_Prop_prelim2}  tends to $0$ as $n\to \infty$, which proves Point 2.\zak
\subsection{Proofs of Proposition \ref{theo_consistency} and Theorem \ref{theo_consistency_vrai}}
{\bf Proof of Proposition \ref{theo_consistency}.}
Independence of the processes $(\Delta_t)_{t\in \nbZ}$ and $(\epsilon_t)_{t\in \nbZ}$ as well their ergodicity yields that, for fixed $j\in \nbN$, the process $\left( (\Delta_{t-1},...,\Delta_{t-j},\epsilon_{t-j}) \right)$ is ergodic. We thus deduce from Expression (\ref{dse_eps_theta}), and using the fact that $(\epsilon_t)_{t\in \nbZ}$ is a weak white noise, that $O_n(\theta)$ defined by (\ref{on}) verifies
\begin{equation}\label{O_infinity}
2 O_n(\theta) \longrightarrow 2 O_\infty(\theta):=\sigma^2 \sum_{j=0}^\infty \nbE \left( [c_j(\theta,\Delta_0,...,\Delta_{-j})]^2\right)= \sigma^2 + \sigma^2 \sum_{j=1}^\infty \nbE \left( [c_j(\theta,\Delta_0,...,\Delta_{-j})]^2\right)\quad \mbox{a.s.}
\end{equation}
as $n\to\infty$ (remember that $c_0(\theta,\Delta_0)=1$). By uniqueness of decomposition (\ref{dse_eps_theta}) in Proposition \ref{prop_stationary}, and since $\epsilon_t(\theta_0)=\epsilon_t$, we have that $(c_i(\theta,\Delta_{t-1},\dots,\Delta_{t-i}))_{i\in\nbN}=(1,0,...)$ if and only if $\theta=\theta_0$, and that $O_\infty(\theta)$ given in (\ref{O_infinity}) is minimum at $\theta=\theta_0$, with minimum given by $O_\infty(\theta_0)=\sigma^2$.
Let us then deduce that the estimator $\check{\theta}_n$ defined in (\ref{LSE}) converges a.s. towards $\theta_0$. For this we let a subsequence $(\check{\theta}_{n_k})_{k\in\nbN}$ converging to some $\theta^*$ in the compact set $\Theta$ and we prove that $\theta^*=\theta_0$. Indeed, by definition of the estimator $\check{\theta}_{n_k}$ we have
\begin{equation}\label{proof_consistency0}
O_{n_k}(\theta_0)\ge O_{n_k}(\check{\theta}_{n_k})
\end{equation}
for all $k\in\nbN$. A Taylor expansion yields the inequality
\begin{equation}\label{proof_consistency1}
| O_{n_k}(\check{\theta}_{n_k})- O_{n_k}(\theta^*)|\le || \check{\theta}_{n_k} - \theta^* ||.  \frac{1}{n_k}\sum_{t=1}^{n_k}\sup_{\theta\in\Theta}[|\epsilon_t(\theta)|.|| \nabla \epsilon_t(\theta)||].
\end{equation}
But, using the ergodic theorem, we have
\begin{multline*}
 \frac{1}{n_k}\sum_{t=1}^{n_k}\sup_{\theta\in\Theta}[|\epsilon_t(\theta)|.|| \nabla \epsilon_t(\theta)||]\le \frac{1}{2 n_k}\sum_{t=1}^{n_k}\left[\sup_{\theta\in\Theta}|\epsilon_t(\theta)|^2 + \sup_{\theta\in\Theta}||\nabla \epsilon_t(\theta)||^2 \right]\\
\longrightarrow \frac{1}{2} \left|\left|\sup_{\theta\in\Theta}|\epsilon_0(\theta)| \right|\right|_2^2 + \frac{1}{2} \left|\left|\sup_{\theta\in\Theta}||\nabla\epsilon_0(\theta)|| \right|\right|_2^2<+\infty ,
\end{multline*}
so that we get from \eqref{proof_consistency1} that $O_{n_k}(\check{\theta}_{n_k})- O_{n_k}(\theta^*)\longrightarrow 0$ as $k\to\infty$. Since $O_{n_k}(\theta^*)\longrightarrow O_{\infty}(\theta^*)$, we obtain, passing to the limit in \eqref{proof_consistency0}, that
$$
O_{\infty}(\theta_0)\ge O_{\infty}(\theta^*),
$$
hence $ \theta^*=\theta_0$ thank to uniqueness of the minimum of $O_{\infty}(\theta)$.\zak
{\bf Proof of Theorem \ref{theo_consistency_vrai}.}
Similarly to the proof of the previous theorem, we let a subsequence $(\hat{\theta}_{n_k})_{k\in\nbN}$ converging to some $\theta_*$ in the compact set $\Theta$ and we prove that $\theta_*=\theta_0$ by proving that $O_{\infty}(\theta_0)= O_{\infty}(\theta_*)$. By definition of $\hat{\theta}_{n_k}$ we have
\begin{equation}\label{proof_consistency_vrai1}
Q_{n_k}(\theta_0)\ge Q_{n_k}(\hat{\theta}_{n_k}),\quad \forall k \ge 0.
\end{equation}
Now, a Taylor expansion yields, for all $\theta'$ and $\theta''$ in $\Theta$, similarly to the argument in the proof of Proposition \ref{theo_consistency},
\begin{equation}\label{proof_consistency_vrai2}
| Q_{n_k}(\theta')- Q_{n_k}(\theta'')|\le || \theta' - \theta'' ||.
\frac{1}{2 n_k}\sum_{t=1}^{n_k}\left[\sup_{\theta\in\Theta}|e_t(\theta)|^2 + \sup_{\theta\in\Theta}||\nabla e_t(\theta)||^2 \right].
\end{equation}
Using inequality $(a+b)^2\le 2 (a^2+b^2)$ for all $a$ and $b$, we deduce that $\sup_{\theta\in\Theta}|e_t(\theta)|^2\le 2 (\sup_{\theta\in\Theta}|e_t(\theta)-\epsilon_t(\theta)|^2) + \sup_{\theta\in\Theta}|\epsilon_t(\theta)|^2$. Since a consequence of Point 3 of Lemma \ref{lemme_prelim} is that $\sup_{\theta\in\Theta}|e_t(\theta)-\epsilon_t(\theta)|^2$ tends to $0$ as $t\to\infty$, the ergodic theorem yields that
$$
\frac{1}{n_k}\sum_{t=1}^{n_k}\left[\sup_{\theta\in\Theta}|e_t(\theta)|^2 + \sup_{\theta\in\Theta}||\nabla e_t(\theta)||^2 \right]\longrightarrow \left|\left|\sup_{\theta\in\Theta}|\epsilon_0(\theta)| \right|\right|_2^2 + \left|\left|\sup_{\theta\in\Theta}||\nabla\epsilon_0(\theta)|| \right|\right|_2^2<+\infty
$$
as $k\to\infty$. Thanks to \eqref{proof_consistency_vrai2} and Point 1 of Proposition \ref{prop_prelim}, we thus deduce that $Q_{n_k}(\theta_0)\longrightarrow O_\infty(\theta_0)$ and $Q_{n_k}(\hat{\theta}_{n_k})\longrightarrow O_\infty(\theta_*)$ as $k\to\infty$, and we conclude in the same way as in proof of Theorem \ref{theo_consistency}.\zak
\subsection{Proofs of Theorem \ref{CLT_theorem}}
Let us introduce the following matrices and vectors
\begin{eqnarray}
I_n(\theta)&:=& \var\left( \sqrt{n}\nabla O_n(\theta)\right)= \left(I_n(l,r)(\theta)\right)_{l,r=1\dots(p+q)K}\in \nbR^{(p+q)K\times (p+q)K},\quad n\in\nbN, ,\label{def_I_n_theta}\\
Y_k(\theta)&:=& \epsilon_k(\theta) \nabla \epsilon_k(\theta)=(Y_k(l)(\theta))_{l=1\dots(p+q)K} \in \nbR^{(p+q)K\times 1},\quad k\in \nbZ, \label{def_Y_k_theta}
\end{eqnarray}
Theorem \ref{CLT_theorem} can be established using the following lemmas.
\begin{lemme}[\cite{davy1968}]
\label{Davydov}
Let $p$, $q$ and $r$ three positive numbers such that
$p^{-1}+q^{-1}+r^{-1}=1$. Then 
\begin{equation}\label{DavydovInequality}
  \left|\mbox{Cov}(X,Y)\right|\leq K_0\|X\|_p\|Y\|_q\left[\alpha\left\{\sigma(X),\sigma(Y)\right\}\right]^{1/r},
\end{equation}
where $\|X\|_p^p=\e(X^p)$, $K_0$ is an universal constant, and $\alpha\left\{\sigma(X),\sigma(Y)\right\}$
denotes the strong mixing coefficient between the $\sigma$-fields $\sigma(X)$ and $\sigma(Y)$
generated by the random variables $X$ and $Y$, respectively.
\end{lemme}

\begin{lemme}\label{existenceI}
 Let the assumptions of Theorem \ref{CLT_theorem} be
satisfied. For all $l$, $r$ in $1$,\dots,$(p+q)K$ and $\theta\in \Theta$ we have
$$
I_n(l,r)(\theta)\longrightarrow I(l,r)(\theta):=\sum_{k=-\infty}^\infty  c_k (l,r)(\theta),\quad n\to +\infty ,
$$
where $c_k(l,r)(\theta)=\cov\left(Y_t(l)(\theta),Y_{t-k}(r)(\theta)\right)$, $k\in\mathbb{Z}$, the former being a convergent series.
\end{lemme}
{\bf Proof of Lemma \ref{existenceI}:}
Let us write
$$\nabla\epsilon_t(\theta) =
\left(\frac{\partial\epsilon_t(\theta)}{\partial\theta_1},\dots,
\frac{\partial\epsilon_t(\theta)}{\partial\theta_{(p+q)K}}\right)',
$$
where $\epsilon_t(\theta)$ is given by \eqref{dse_eps_theta}. The
process $\left(Y_k(\theta)\right)_k$ is strictly stationary and
ergodic. Moreover, we have
\begin{eqnarray*}
I_n(\theta)=\var\left(\sqrt{n}\frac{\partial}{\partial\theta}O_n(\theta)\right)
&=&\var\left(\frac{1}{\sqrt{n}}\sum_{t=1}^{n}Y_t(\theta)\right)
=\frac{1}{n}\sum_{t,s=1}^{n}\mbox{Cov}\left(Y_t(\theta),Y_s(\theta)\right)\\
&=&\frac{1}{n}\sum_{k=-n+1}^{n-1}(n-|k|)\mbox{Cov}\left(Y_t(\theta),Y_{t-k}(\theta)\right).
\end{eqnarray*}
From Proposition \ref{prop_stationary} and Lemma \ref{coeff_c_i_expo_decrease},  we have
$$\epsilon_t(\theta)= \sum_{i=0}^\infty c_i(\theta,\Delta_{t},\dots,\Delta_{t-i+1})\epsilon_{t-i}\text{ and }\frac{\partial\epsilon_t(\theta)}{\partial\theta_l}=\sum_{i=0}^\infty c_{i,l}(\theta,\Delta_{t},\dots,\Delta_{t-i+1}) \epsilon_{t-i},\text{ for }l=1,\dots,(p+q)K,$$
where we recall that $c_i(\theta,\Delta_{t},\dots,\Delta_{t-i+1})$ is defined by \eqref{expression_eps_theta}, and
\begin{eqnarray*}
c_{i,l}(\theta,\Delta_{t},\dots,\Delta_{t-i+1})&=&\frac{\partial}{\partial \theta_l} c_i(\theta,\Delta_{t},\dots,\Delta_{t-i+1})\\
&=&\frac{\partial}{\partial\theta_l} \left( \sum_{k=0}^i w_1\Phi(\Delta_{t},\theta)\dots\Phi(\Delta_{t-k+1},\theta)M \Psi(\Delta_{t-k})\dots\Psi(\Delta_{t-i+1})w_{p+1}'\right),
\end{eqnarray*}
with the following upper bound holding thanks to \eqref{coeff_c_i_e_expo_decrease}:
\begin{eqnarray*}
\mathbb{E}\sup_{\theta\in\Theta}(c_i(\theta,\Delta_{t},\dots,\Delta_{t-i+1}))^2\leq C\rho^i \text{ and }\mathbb{E}\sup_{\theta\in\Theta}(
c_{i,l}(\theta,\Delta_{t},\dots,\Delta_{t-i+1}))^2\leq C\rho^i,\quad \forall i.
\end{eqnarray*}
Let
\begin{eqnarray}
\beta_{i,j,i',j',k}(l,r)(\theta)&=& \e\left[c_i(\theta,\Delta_{t},\dots,\Delta_{t-i+1})c_{j,l}(\theta,\Delta_{t},\dots,\Delta_{t-j+1})
c_{i'}(\theta,\Delta_{t-k},\dots,\Delta_{t-k-i'+1})\right.\nonumber \\
&&\left. c_{j',r}(\theta,\Delta_{t-k},\dots,\Delta_{t-k-j'+1})\right]\e\left[\epsilon_{t-i}
\epsilon_{t-j}\epsilon_{t-k-i'}\epsilon_{t-k-j'}\right]\nonumber \\
&&-
\e\left[c_i(\theta,\Delta_{t},\dots,\Delta_{t-i+1})c_{j,l}(\theta,\Delta_{t},\dots,\Delta_{t-j+1})\right]\nonumber \\
&&\times
\e\left[c_{i'}(\theta,\Delta_{t-k},\dots,\Delta_{t-k-i'+1})c_{j',r}(\theta,\Delta_{t-k},\dots,\Delta_{t-k-j'+1})\right]\e\left[\epsilon_{t-i}
\epsilon_{t-j}\right]
\nonumber \\
&&\times\e\left[\epsilon_{t-k-i'}\epsilon_{t-k-j'}\right]
\nonumber \\
&=& \e\left[c_i(\theta,\Delta_{t},\dots,\Delta_{t-i+1})c_{j,l}(\theta,\Delta_{t},\dots,\Delta_{t-j+1})
c_{i'}(\theta,\Delta_{t-k},\dots,\Delta_{t-k-i'+1})\right. \nonumber\\
&&\left. c_{j',r}(\theta,\Delta_{t-k},\dots,\Delta_{t-k-j'+1})\right]\mbox{Cov}\left(\epsilon_{t-i}
\epsilon_{t-j},\epsilon_{t-k-i'}\epsilon_{t-k-j'}\right)\nonumber \\
&&+
\mbox{Cov}\left(c_i(\theta,\Delta_{t},\dots,\Delta_{t-i+1})c_{j,l}(\theta,\Delta_{t},\dots,\Delta_{t-j+1}),c_{i'}(\theta,\Delta_{t-k},\dots,\Delta_{t-k-i'+1})\right.\nonumber \\
&&\left. c_{j',r}(\theta,\Delta_{t-k},\dots,\Delta_{t-k-j'+1})\right)
\e\left[\epsilon_{t-i}
\epsilon_{t-j}\right]\e\left[\epsilon_{t-k-i'}\epsilon_{t-k-j'}\right].\label{difference_strong_weak_noise}
\end{eqnarray}
We then obtain
$$c_k(l,r)(\theta)=\sum_{i=0}^\infty\sum_{j=0}^\infty\sum_{i'=0}^\infty\sum_{j'=0}^\infty\beta_{i,j,i',j',k}(l,r)(\theta),\quad k\in\mathbb{Z}.$$
The Cauchy-Schwarz inequality implies that
\begin{eqnarray}
\nonumber&&\left|\e[c_i(\theta,\Delta_{t},\dots,\Delta_{t-i+1})c_{j,l}(\theta,\Delta_{t},\dots,\Delta_{t-j+1})
c_{i'}(\theta,\Delta_{t-k},\dots,\Delta_{t-k-i'+1})\right.\\\nonumber&\times&\left. c_{j',r}(\theta,\Delta_{t-k},\dots,\Delta_{t-k-j'+1})]\right|\leq
\left(\e[c_i(\theta,\Delta_{t},\dots,\Delta_{t-i+1})c_{j,l}(\theta,\Delta_{t},\dots,\Delta_{t-j+1})]^2\right)^{1/2}\\\nonumber&\times&\left(\mathbb{E}
[c_{i'}(\theta,\Delta_{t-k},\dots,\Delta_{t-k-i'+1}) c_{j',r}(\theta,\Delta_{t-k},\dots,\Delta_{t-k-j'+1})]^2\right)^{1/2}
\leq
\left(\e[c_i(\theta,\Delta_{t},\dots,\Delta_{t-i+1})]^4\right.\\&\times&\left. \e[c_{j,l}(\theta,\Delta_{t},\dots,\Delta_{t-j+1})]^4\right)^{1/4}\left(\mathbb{E}
[c_{i'}(\theta,\Delta_{t-k},\dots,\Delta_{t-k-i'+1})]^4\e[c_{j',r}(\theta,\Delta_{t-k},\dots,\Delta_{t-k-j'+1})]^4\right)^{1/4}\nonumber\\&\leq&C\rho^{i+j+i'+j'}.\label{c_i-Cauchy-Schw}
\end{eqnarray}

First, suppose that $k\geq0$, for all $l$, $r$ in $1$,\dots,$(p+q)K$ and $\theta\in \Theta$, in view of \eqref{c_i-Cauchy-Schw} it follows that
\begin{eqnarray*}
\left| c_k(l,r)(\theta)\right|&=&\left|\mbox{cov}\left(Y_t(l)(\theta),Y_{t-k}(r)(\theta)\right)\right|=\left|\sum_{i=0}^\infty\sum_{j=0}^\infty\sum_{i'=0}^\infty\sum_{j'=0}^\infty\beta_{i,j,i',j',k}(l,r)(\theta)\right|\\&\leq& g_1+g_2+g_3+g_4+g_5+h_1+h_2+h_3 ,
\end{eqnarray*}
where
\begin{eqnarray*}
g_1&=&\sum_{i>[k/2]}\sum_{j=0}^\infty\sum_{i'=0}^\infty\sum_{j'=0}^\infty\kappa\rho^{i+j+i'+j'}\left|\mbox{Cov}\left(\epsilon_{t-i}
\epsilon_{t-j},\epsilon_{t-k-i'}\epsilon_{t-k-j'}\right)\right|,\\
g_2&=&\sum_{i=0}^\infty\sum_{j>[k/2]}\sum_{i'=0}^\infty\sum_{j'=0}^\infty\kappa\rho^{i+j+i'+j'}\left|\mbox{Cov}\left(\epsilon_{t-i}
\epsilon_{t-j},\epsilon_{t-k-i'}\epsilon_{t-k-j'}\right)\right|\\
g_3&=&\sum_{i=0}^\infty\sum_{j=0}^\infty\sum_{i'>[k/2]}\sum_{j'=0}^\infty\kappa\rho^{i+j+i'+j'}\left|\mbox{Cov}\left(\epsilon_{t-i}
\epsilon_{t-j},\epsilon_{t-k-i'}\epsilon_{t-k-j'}\right)\right|,\\
g_4&=&\sum_{i=0}^\infty\sum_{j=0}^\infty\sum_{i'=0}^\infty\sum_{j'>[k/2]}\kappa\rho^{i+j+i'+j'}\left|\mbox{Cov}\left(\epsilon_{t-i}
\epsilon_{t-j},\epsilon_{t-k-i'}\epsilon_{t-k-j'}\right)\right|
\\g_5&=&\sum_{i=0}^{[k/2]}\sum_{j=0}^{[k/2]}\sum_{i'=0}^{[k/2]}\sum_{j'=0}^{[k/2]}\kappa\rho^{i+j+i'+j'}\left|\mbox{Cov}\left(\epsilon_{t-i}
\epsilon_{t-j},\epsilon_{t-k-i'}\epsilon_{t-k-j'}\right)\right|,
\\h_1&=&\sigma^4\sum_{i>[k/2]}\sum_{i'=0}^{\infty}\left|\mbox{Cov}(c_i(\theta,\Delta_{t},\dots,\Delta_{t-i+1})c_{i,l}(\theta,\Delta_{t},\dots,\Delta_{t-i+1}),\right.\\&&\left. c_{i'}(\theta,\Delta_{t-k},\dots,\Delta_{t-k-i'+1})c_{i',r}(\theta,\Delta_{t-k},\dots,\Delta_{t-k-i'+1}))\right|,
\\h_2&=&\sigma^4\sum_{i=0}^{\infty}\sum_{i'>[k/2]}\left|\mbox{Cov}(c_i(\theta,\Delta_{t},\dots,\Delta_{t-i+1})c_{i,l}(\theta,\Delta_{t},\dots,\Delta_{t-i+1}),\right.\\&&\left. c_{i'}(\theta,\Delta_{t-k},\dots,\Delta_{t-k-i'+1})c_{i',r}(\theta,\Delta_{t-k},\dots,\Delta_{t-k-i'+1}))\right|,
\\h_3&=&\sigma^4\sum_{i=0}^{[k/2]}\sum_{i'=0}^{[k/2]}\left|\mbox{Cov}(c_i(\theta,\Delta_{t},\dots,\Delta_{t-i+1})c_{i,l}(\theta,\Delta_{t},\dots,\Delta_{t-i+1}),\right.\\&&\left. c_{i'}(\theta,\Delta_{t-k},\dots,\Delta_{t-k-i'+1})c_{i',r}(\theta,\Delta_{t-k},\dots,\Delta_{t-k-i'+1}))\right|.
\end{eqnarray*}
Note that, in the strong noise case, we easily check  that the $\mbox{Cov}\left(\epsilon_{t-i}
\epsilon_{t-j},\epsilon_{t-k-i'}\epsilon_{t-k-j'}\right)$ term in \eqref{difference_strong_weak_noise} is non zero only for indices $i$, $j$, $i'$, $j'$ such that $i=j=k+i'=k+j'$. This fact entails that, instead of considering five sums $g_1$,..., $g_5$, we only need to consider one sum in the form $ \kappa \sum_{j=k}^\infty \rho^{2(2j-k)}$, which is a $\mathrm{O}(\rho^k)$.

Because
\begin{eqnarray*}
\left|\mbox{Cov}\left(\epsilon_{t-i}
\epsilon_{t-j},\epsilon_{t-k-i'}\epsilon_{t-k-j'}\right)\right|&\leq&
\sqrt{\e\left[\epsilon_{t-i}
\epsilon_{t-j}\right]^2
\e\left[\epsilon_{t-k-i'}\epsilon_{t-k-j'}\right]^2}\leq\e\left|\epsilon_t\right|^4<\infty
\end{eqnarray*}
by Assumption ${(\bf A3)}$, we have
$$g_1=\sum_{i>[k/2]}\sum_{j=0}^\infty\sum_{i'=0}^\infty\sum_{j'=0}^\infty\kappa\rho^{i+j+i'+j'}\left|\mbox{Cov}\left(\epsilon_{t-i}
\epsilon_{t-j},\epsilon_{t-k-i'}\epsilon_{t-k-j'}\right)\right|\leq \kappa_1\rho^{k/2},$$ for some positive constant $\kappa_1$. Using the same arguments we obtain that $g_i\quad(i=2,3,4)$ is bounded by $\kappa_i\rho^{k/2}$. Furthermore, {\bf (A3)} and the Cauchy-Schwarz inequality yields that $\left\|\epsilon_{i}
\epsilon_{i'}\right\|_{2+\nu}<+\infty$ for any $i$ and $i'$ in $\nbZ$. Lemma \ref{Davydov} thus entails that
\begin{eqnarray*}
g_5&=&\sum_{i=0}^{[k/2]}\sum_{j=0}^{[k/2]}\sum_{i'=0}^{[k/2]}\sum_{j'=0}^{[k/2]}\kappa\rho^{i+j+i'+j'}\left|\mbox{Cov}\left(\epsilon_{t-i}
\epsilon_{t-j},\epsilon_{t-k-i'}\epsilon_{t-k-j'}\right)\right|\\
&\leq& \sum_{i=0}^{[k/2]}\sum_{j=0}^{[k/2]}\sum_{i'=0}^{[k/2]}\sum_{j'=0}^{[k/2]}\kappa_5\rho^{i+j+i'+j'}\left\|\epsilon_{t-i}
\epsilon_{t-j}\right\|_{2+\nu}\left\|\epsilon_{t-k-i'}\epsilon_{t-k-j'}\right\|_{2+\nu}\\&&\times\left\{\alpha_{\epsilon}\left(\min\left[k+j'-i,k+i'-i,k+j'-j,k+i'-j\right]\right)\right\}^{\nu/(2+\nu)}\leq\kappa'
\alpha_{\epsilon}^{\nu/(2+\nu)}\left(\left[k/2\right]\right).
\end{eqnarray*}
Since
\begin{eqnarray*}
&&\left|\mbox{Cov}(c_i(\theta,\Delta_{t},\dots,\Delta_{t-i+1})c_{i,l}(\theta,\Delta_{t},\dots,\Delta_{t-i+1}), c_{i'}(\theta,\Delta_{t-k},\dots,\Delta_{t-k-i'+1})\right.\\&\times&\left.c_{i',r}(\theta,\Delta_{t-k},\dots,\Delta_{t-k-i'+1}))\right|\leq C\rho^{i+i'},
\end{eqnarray*}
we have
\begin{eqnarray*}
h_1&=&\sigma^4\sum_{i>[k/2]}\sum_{i'=0}^{\infty}\left|\mbox{Cov}(c_i(\theta,\Delta_{t},\dots,\Delta_{t-i+1})c_{i,l}(\theta,\Delta_{t},\dots,\Delta_{t-i+1}),\right.\\&&\left. c_{i'}(\theta,\Delta_{t-k},\dots,\Delta_{t-k-i'+1})c_{i',r}(\theta,\Delta_{t-k},\dots,\Delta_{t-k-i'+1}))\right|\leq \kappa'_1\rho^{k/2},
\end{eqnarray*}
for some positive constant $\kappa'_1$. Using the same arguments we obtain that $h_2$ is bounded by $\kappa'_2\rho^{k/2}$. The $\alpha-$mixing property (see Theorem 14.1 in \cite{davidson1994}, p. 210) and Lemma \ref{Davydov}, along with \eqref{coeff_c_i_expo_decrease}, entail that
\begin{eqnarray*}
h_3&=&\sigma^4\sum_{i=0}^{[k/2]}\sum_{i'=0}^{[k/2]}\left|\mbox{Cov}(c_i(\theta,\Delta_{t},\dots,\Delta_{t-i+1})c_{i,l}(\theta,\Delta_{t},\dots,\Delta_{t-i+1}),\right.\\&&\left. c_{i'}(\theta,\Delta_{t-k},\dots,\Delta_{t-k-i'+1})c_{i',r}(\theta,\Delta_{t-k},\dots,\Delta_{t-k-i'+1}))\right|\\
&\leq& \sum_{i=0}^{[k/2]}\sum_{i'=0}^{[k/2]}\kappa_6\left\|c_i(\theta,\Delta_{t},\dots,\Delta_{t-i+1})c_{i,l}(\theta,\Delta_{t},\dots,\Delta_{t-i+1})\right\|_{2+\nu}\\&&\times\left\|c_{i'}(\theta,\Delta_{t-k},\dots,\Delta_{t-k-i'+1})c_{i',r}(\theta,\Delta_{t-k},\dots,\Delta_{t-k-i'+1}))\right\|_{2+\nu}\\&&\times\left\{\alpha_{\Delta}\left(k+1-i\right)\right\}^{\nu/(2+\nu)}\leq\kappa'_3
\alpha_{\Delta}^{\nu/(2+\nu)}\left(\left[k/2\right]\right).
\end{eqnarray*}
It follows that
$$\sum_{k=0}^{\infty}\left|c_k(l,r)(\theta)\right|
\leq
\kappa\sum_{k=0}^{\infty}\rho^{|k|/2}+\kappa'\sum_{k=0}^{\infty}\alpha_{\epsilon}^{\nu/(2+\nu)}
\left(\left[k/2\right]\right)+\kappa''\sum_{k=0}^{\infty}\alpha_{\Delta}^{\nu/(2+\nu)}
\left(\left[k/2\right]\right)<\infty,$$ by Assumption ${(\bf A2)}$. The
same bounds clearly holds for
$$\sum_{k=-\infty}^{0}\left|c_k(l,r)(\theta)\right|,$$
which shows that $$\sum_{k=-\infty}^{\infty}\left|c_k(l,r)(\theta)\right|<\infty.$$
Then, the dominated convergence theorem gives
$$
I_n(l,r)(\theta)=\frac{1}{n}\sum_{k=-n+1}^{n-1}(n-|k|)c_k(l,r)(\theta)\longrightarrow I(l,r)(\theta):=\sum_{k=-\infty}^\infty  c_k (l,r)(\theta),\quad n\to +\infty ,
$$
and completes the proof. \zak

\begin{lemme}\label{NAscore}
Under the assumptions of Theorem \ref{CLT_theorem}, we have convergence in distribution of the random vector
$$\sqrt{n}\nabla
Q_n(\theta_0) \stackrel{{\cal
D}}{\to}{\cal N}(0,I),\text{ as } n\to\infty$$
where we recall that matrix $I$ is given by \eqref{spectraldensityatzero}.
\end{lemme}
{\bf Proof of Lemma \ref{NAscore}:} In view of Proposition \ref{prop_prelim}, it is easy to see that $$\sqrt{n}\nabla\left(Q_n-O_n\right)(\theta_0)=o_\P(1).$$
Thus $\nabla Q_n(\theta_0) $ and $\nabla O_n(\theta_0) $ have the same asymptotic distribution. Therefore, it remains to show that
$$\sqrt{n}\nabla
O_n(\theta_0) \stackrel{{\cal
D}}{\to}{\cal N}(0,I),\text{ as } n\to\infty.$$
For $l$, in $1$,\dots,$(p+q)K$ and $\theta\in \Theta$,
we have
\begin{equation}
\label{linearcombinationintheo2}
\frac{\partial
\epsilon_{t}(\theta)}{\partial \theta_{l}}=\sum_{i=1}^{\infty}c_{i,l}(\theta,\Delta_{t},\dots,\Delta_{t-i+1})\epsilon_{t-i},
\end{equation}
where the sequence $c_{i,l}(\theta,\Delta_{t},\dots,\Delta_{t-i+1})$
is such that
$\mathbb{E}\sup_{\theta\in\Theta}|(c_{i,l}(\theta,\Delta_{t},\dots,\Delta_{t-i+1}))^2\to
0$ at a geometric rate as $i\to\infty$ (see Lemma
\ref{c_i^2_expo_decrease}). Moreover, note that
\begin{eqnarray*}
\sqrt{n}\frac{\partial O_n(\theta)}{\partial\theta_l}
=\frac{1}{\sqrt{n}}\sum_{t=1}^{n}Y_t(l)(\theta)
=\frac{1}{\sqrt{n}}\sum_{t=1}^{n}\sum_{i=0}^{\infty}c_{i}(\theta,\Delta_{t},\dots,\Delta_{t-i+1})\epsilon_{t-i}\sum_{j=1}^{\infty}c_{j,l}(\theta,\Delta_{t},\dots,\Delta_{t-j+1})\epsilon_{t-j}.
\end{eqnarray*}
Since $ \nabla \epsilon_{t}(\theta_0) $ belongs to the Hilbert space ${\cal H}_\epsilon(t-1)$, the random variables $\epsilon_{t}(\theta_0)$ and $ \nabla \epsilon_{t}(\theta_0) $ are orthogonal and it is easy to verify that $\e\left[\sqrt{n}\nabla O_n(\theta_0)\right]=0$.
Now, we have for all $m$
\begin{eqnarray*}
\sqrt{n}\frac{\partial O_n(\theta_0)}{\partial\theta_l}=\frac{1}{\sqrt{n}}\sum_{t=1}^{n}Y_{t,m}(l)+\frac{1}{\sqrt{n}}\sum_{t=1}^{n}Z_{t,m}(l)
\end{eqnarray*}
where
\begin{eqnarray*}
Y_{t,m}(l)&=&\sum_{j=1}^{m}c_{j,l}(\theta_0,\Delta_{t},\dots,\Delta_{t-j+1})\epsilon_t\epsilon_{t-j}\\
Z_{t,m}(l)&=&\sum_{j=m+1}^{\infty}c_{j,l}(\theta_0,\Delta_{t},\dots,\Delta_{t-j+1})\epsilon_t\epsilon_{t-j}.
\end{eqnarray*}
 Let
 \begin{align*}
 Y_{t,m}&:=Y_{t,m}(\theta_0)=\left(Y_{t,m}(1),\dots,Y_{t,m}((p+q)K)\right)'\text{ and
 }\\
 Z_{t,m}&:=Z_{t,m}(\theta_0)=\left(Z_{t,m,}(1),\dots,Z_{t,m}((p+q)K)\right)'.
 \end{align*}
The processes $(Y_{t,m})_{t}$ and $(Z_{t,m})_{t}$ are stationary and
centered. Moreover, under Assumption {\bf (A2)} and  $m$ fixed, the
process $Y=(Y_{t,m})_{t}$ is strong mixing (see
\cite{davidson1994}, Theorem 14.1 p. 210), with mixing coefficients
$\alpha_Y(h)\leq \alpha_{\Delta,\epsilon}\left(\max\{0,h-m\}\right)
\leq
\alpha_{\Delta}\left(\max\{0,h-m+1\}\right)+\alpha_{\epsilon}\left(\max\{0,h-m\}\right)$,
by independence of $(\Delta_t)_{t\in\nbZ}$ and $(\epsilon_t)_{t\in\nbZ}$. Applying the central
limit theorem (CLT) for mixing processes (see \cite{herr}) we
directly obtain
$$\frac{1}{\sqrt{n}}\sum_{t=1}^{n}Y_{t,m}\stackrel{{\cal D}}{\to}{\cal N}(0,I_m),\quad I_m=\sum_{h=-\infty}^{\infty}
\cov\left(Y_{t,m},Y_{t-h,m}\right).$$
In the strong noise case, the infinite sum in $I_m$ reduces to one term corresponding to $h=0$, and $I_m$ simply equals $\cov\left(Y_{t,m},Y_{t,m}\right)$.

As in \cite{fz98} (see Lemma 3), we can show that $I=\lim_{m\to\infty}I_m$ exists. Since
$\|Z_{t,m}\|_2\to 0$ at an exponential rate when $m\to\infty$, using
the arguments given in \cite{fz98} (see Lemma 4), we show that
\begin{equation}
\label{Zt}
\lim_{m\to\infty}\limsup_{n\to\infty}\P\left\{\left\|n^{-1/2}\sum_{t=1}^{n}Z_{t,m}\right\|>\varepsilon\right\}=0
\end{equation}
for every $\varepsilon>0$ (see the following lemma \ref{lemsup}).
From a standard result (see {\em e.g.} \cite{broc-d}, Proposition
6.3.9), we deduce that
$$\frac{1}{\sqrt{n}}\sum_{t=1}^{n} \nabla O_n(\theta_0) =
\frac{1}{\sqrt{n}}\sum_{t=1}^{n}Y_{t,m}+\frac{1}{\sqrt{n}}\sum_{t=1}^{n}Z_{t,m}\stackrel{{\cal D}}{\to}{\cal N}(0,I),$$
which completes the proof. \zak

\begin{lemme}
\label{lemsup}
Under the assumptions of Theorem \ref{CLT_theorem}, (\ref{Zt}) holds, that is
$$\lim_{m\to\infty}\limsup_{n\to\infty}\P\left\{\left\|n^{-1/2}\sum_{t=1}^{n}Z_{t,m}\right\|>\varepsilon\right\}=0.
$$
\end{lemme}
{\bf Proof of Lemma \ref{lemsup}:}
For $l=1,\dots,(p+q)K$, by stationarity we have
\begin{eqnarray*}
\var\left(\frac{1}{\sqrt{n}}\sum_{t=1}^{n}Z_{t,m}(l)\right)&=&\frac{1}{n}\sum_{t,s=1}^{n}\mbox{Cov}(Z_{t,m}(l),Z_{s,m}(l))\\
&=&\frac{1}{n}\sum_{|h|<n}(n-|h|)\mbox{Cov}(Z_{t,m}(l),Z_{t-h,m}(l))\\
&\leq& \sum_{h=-\infty}^{\infty}\left|\mbox{Cov}(Z_{t,m}(l),Z_{t-h,m}(l))\right|.
\end{eqnarray*}
Consider first the case $h\geq0$. Because
$\mathbb{E}\sup_{\theta\in\Theta}(c_{j,l}(\theta_0,\Delta_{t},\dots,\Delta_{t-j+1}))^2\leq
\kappa\rho^{j}$ (see \ref{coeff_c_i_expo_decrease}), using also
$\e|\epsilon_t|^4<\infty$, for $[h/2]\leq m$, it follows from the
H\"{o}lder inequality that
\begin{equation}
\label{firstmajo}
\sup_h\left|\mbox{Cov}(Z_{t,m}(l),Z_{t-h,m}(l))\right|=\sup_h\left|\e(Z_{t,m}(l)Z_{t-h,m}(l))\right|\leq \kappa\rho^m.
\end{equation}
Let $h>0$ such that $[h/2]>m$. Write
$$Z_{t,m}=Z_{t,m}^{h^-}(l)+Z_{t,m}^{h^+}(l),$$
where
$$
Z_{t,m}^{h^-}(l)=\sum_{j=m+1}^{[h/2]}c_{j,l}(\theta_0,\Delta_{t},\dots,\Delta_{t-j+1})\epsilon_t\epsilon_{t-j},
\quad Z_{t,m}^{h^+}(l)=\sum_{j=[h/2]+1}^{\infty}c_{j,l}(\theta_0,\Delta_{t},\dots,\Delta_{t-j+1})\epsilon_t\epsilon_{t-j}.
$$
Note that $Z_{t,m}^{h^-}(l)$ belongs to the $\sigma$-field generated by $\{\Delta_{t},\dots,\Delta_{t-[h/2]+1},\epsilon_t,\epsilon_{t-1},\dots,\epsilon_{t-[h/2]}\}$ and
that $Z_{t-h,m}(l)$ belongs to the $\sigma$-field generated by $\{\Delta_{t-h},\Delta_{t-h-1},\dots,\epsilon_{t-h},\epsilon_{t-h-1},\dots\}$.
Note also that, by {\bf (A3)}, $\e|Z_{t,m}^{h^-}(l)|^{2+\nu}<\infty$
and $\e|Z_{t-h,m}(l)|^{2+\nu}<\infty$. The $\alpha-$mixing property
and Lemma \ref{Davydov} then entail that
\begin{eqnarray}
\label{secondmajo}
\left|\mbox{Cov}(Z_{t,m}^{h^-}(l),Z_{t-h,m}(l))\right|&\leq &\kappa_1\sum_{j=m+1}^{[h/2]}\sum_{j'=m+1}^{\infty}\left\|c_{j',l}(\theta_0,\Delta_{t-h},\dots,\Delta_{t-h-j'+1})\epsilon_t\epsilon_{t-j'}\right\|_{2+\nu}\nonumber\\&&\times\left\|c_{j,l}(\theta_0,\Delta_{t},\dots,\Delta_{t-j+1})\epsilon_t\epsilon_{t-j}\right\|_{2+\nu}\left[\alpha_{\Delta,\epsilon}([h/2])\right]^{\nu/(2+\nu)}\nonumber\\&\leq& \nonumber\kappa_2\sum_{j=m+1}^{[h/2]}\sum_{j'=m+1}^{\infty}\rho^j\rho^{j'}\left[\alpha_{\epsilon}^{\nu/(2+\nu)}([h/2])+\alpha_{\Delta}^{\nu/(2+\nu)}([h/2])\right]\\&\leq& \kappa\rho^{m}\left[\alpha_{\epsilon}^{\nu/(2+\nu)}([h/2])+\alpha_{\Delta}^{\nu/(2+\nu)}([h/2])\right].
\end{eqnarray}
By the argument used to show (\ref{firstmajo}), we also have
\begin{equation}
\label{thirdmajo}
\left|\mbox{Cov}(Z_{t,m}^{h^+}(l),Z_{t-h,m}(l))\right|\leq \kappa\rho^h\rho^m.
\end{equation}
In view of (\ref{firstmajo}), (\ref{secondmajo}) and
(\ref{thirdmajo}), we obtain
$$\sum_{h=0}^{\infty}\left|\mbox{Cov}(Z_{t,m}(l),Z_{t-h,m}(l))\right|\leq
\kappa m\rho^m+\sum_{h=m}^{\infty}\left\{\kappa\rho^h\rho^m+\kappa\rho^{m}\left[\alpha_{\epsilon}^{\nu/(2+\nu)}([h/2])+\alpha_{\Delta}^{\nu/(2+\nu)}([h/2])\right]\right\}\to 0$$
as $m\to\infty$ by {\bf (A2)}. This implies that
\begin{equation}
\label{markov}
\sup_n\var\left(\frac{1}{\sqrt{n}}\sum_{t=1}^{n}Z_{t,m}(l)\right)\xrightarrow[m\to\infty]{}0.
\end{equation}
We have the same bound for $h<0$. The conclusion follows from (\ref{markov}).
\zak
\begin{lemme}
\label{lem1}
Under the assumptions of Theorem \ref{CLT_theorem}, almost surely
$$\nabla^2
Q_n(\theta_0) \longrightarrow J,\quad n\to \infty ,$$
where $J$ given by \eqref{expression_J} exists  and is invertible.
\end{lemme}
{\bf Proof of Lemma \ref{lem1}:} For all $l$, $r$ in
$1,\dots,(p+q)K$, in view of Proposition \ref{prop_prelim}, we have
almost surely
$$\left|\frac{\partial^2
}{\partial\theta_l\partial\theta_r}\left(Q_n(\theta_0)-O_n(\theta_0)\right)\right|\to0,
\text{ as } t\to\infty.$$ Thus ${\partial^2
Q_n(\theta_0)}/{\partial\theta_l\partial\theta_r}$ and ${\partial^2
O_n(\theta_0)}/{\partial\theta_l\partial\theta_r}$ have almost
surely the same asymptotic distribution. From \eqref{dse_eps_theta}
and \eqref{coeff_c_i_expo_decrease}, there exists a sequence $\left(
c_{i,l,r}(\theta,\Delta_{t-1},\dots,\Delta_{t-i})\right)_{i\in\mathbb{N}}$
such that
\begin{equation}
\label{linearcombinationsecondorder} \frac{\partial^2
\epsilon_{t}(\theta)}{\partial \theta_{l}\partial
\theta_{r}}=\sum_{i=1}^{\infty}c_{i,l,r}(\theta,\Delta_{t},\dots,\Delta_{t-i+1})\epsilon_{t-i}
\text{ with
}\mathbb{E}(c_{i,l,r}(\theta,\Delta_{t},\dots,\Delta_{t-i+1}))^2\leq
C\rho^i, \,\forall i.
\end{equation}
This implies that ${\partial^2
\epsilon_{t}(\theta)}/{\partial \theta_{l}\partial \theta_{r}}$ belongs to $L^2$. On the other hand, we have
\begin{eqnarray*}
\frac{\partial^2 O_n(\theta)}{\partial\theta_l\partial\theta_r}&=&
\frac{1}{n}\sum_{t=1}^n\epsilon_t(\theta)\frac{\partial^2 \epsilon_t(\theta)}{\partial \theta_l\partial \theta_r}+\frac{1}{n}\sum_{t=1}^n\frac{\partial\epsilon_t(\theta)}{\partial \theta_l}\frac{\partial \epsilon_t(\theta)}{\partial \theta_r}\\
&\longrightarrow &\e\left(\epsilon_t(\theta)\frac{\partial^2 \epsilon_t(\theta)}{\partial \theta_l\partial \theta_r}\right)+\e\left(\frac{\partial\epsilon_t(\theta)}{\partial \theta_l}\frac{\partial \epsilon_t(\theta)}{\partial \theta_r}\right),\text{ as } n\to\infty,
\end{eqnarray*}
by the ergodic theorem.
Using the uncorrelatedness between $\epsilon_t(\theta_0)$ and
the linear past ${\cal H}_\epsilon(t-1)$,
 $\partial \epsilon_t(\theta_0)/\partial \theta_l\in {\cal H}_\epsilon(t-1)$, and  $\partial^2 \epsilon_t(\theta_0)/\partial \theta_l\partial \theta_r\in {\cal H}_\epsilon(t-1)$, we have
\begin{eqnarray}
\e\left(\frac{\partial^2
O_n(\theta_0)}{\partial\theta_l\partial\theta_r}\right)=
\e\left(\frac{\partial \epsilon_t(\theta_0)}{\partial
\theta_l}\frac{\partial \epsilon_t(\theta_0)}{\partial
\theta_r}\right)=J(l,r). \label{expressionJ}
\end{eqnarray}
Therefore, $J$ is the covariance matrix of $\partial
\epsilon_t(\theta_0)/\partial \theta$. If $J$ is singular, then
there exists a vector $\bc=(c_1,\dots,c_{(p+q)K})'\ne 0$ such that
$\bc'J\bc=0$. Thus we have
\begin{align}\label{conditionInvJ}
\sum_{k=1}^{(p+q)K}c_k\frac{\partial \epsilon_t(\theta_0)}{\partial
\theta_k}=0,\,a.s.
\end{align}
Differentiating the two sides of \eqref{def_inov_theta} yields
\begin{align*}
-\sum_{i=1}^p(g_i^{a})^*(\Delta_t,\theta_0)
X_{t-i}=\sum_{k=1}^{(p+q)K}c_k\frac{\partial
\epsilon_t(\theta_0)}{\partial \theta_k}-\sum_{j=1}^q
g_j^b(\Delta_t,\theta_0)\sum_{k=1}^{(p+q)K}c_k\frac{\partial
\epsilon_{t-j}(\theta_0)}{\partial \theta_k}-\sum_{j=1}^q
(g_j^{b})^*(\Delta_t,\theta_0)\epsilon_{t-j}(\theta_0)
\end{align*}
where
\begin{align*}
(g_i^{a})^*(\Delta_t,\theta_0)=\sum_{k=1}^{(p+q)K}c_k\frac{\partial
g_i^a(\Delta_t,\theta_0)}{\partial \theta_k}\text{ and }
(g_j^{b})^*(\Delta_t,\theta_0)=\sum_{k=1}^{(p+q)K}c_k\frac{\partial
g_j^b(\Delta_t,\theta_0)}{\partial \theta_k}.
\end{align*}
Because \eqref{conditionInvJ} is satisfied for all $t$, we have
\begin{align*}
\sum_{i=1}^p(g_i^{a})^*(\Delta_t,\theta_0) X_{t-i}=\sum_{j=1}^q
(g_j^{b})^*(\Delta_t,\theta_0)\epsilon_{t-j}(\theta_0).
\end{align*}
The latter equation yields a ARMARC$(p-1,q-1)$ representation at best. The identifiability assumption (see Proposition \ref{prop_stationary})
excludes the existence of such representation.

Thus
\begin{align*}
(g_i^{a})^*(\Delta_t,\theta_0)=\sum_{k=1}^{(p+q)K}c_k\frac{\partial
g_i^a(\Delta_t,\theta_0)}{\partial \theta_k}=0\text{ and }
(g_j^{b})^*(\Delta_t,\theta_0)=\sum_{k=1}^{(p+q)K}c_k\frac{\partial
g_j^b(\Delta_t,\theta_0)}{\partial \theta_k}=0
\end{align*}
and the conclusion follows.
\zak

\noindent{\bf Proof of Theorem \ref{CLT_theorem}:} For all
$i,j,k=1,\dots,K(p+q)$ we have
\begin{align*}
\frac{\partial^3O_n(\theta)}{\partial\theta_i\partial\theta_j\partial\theta_k}&=
\frac{1}{n}\sum_{t=1}^n\left\{\epsilon_t(\theta)\frac{\partial^3\epsilon_t(\theta)}{\partial\theta_i\partial\theta_j\partial\theta_k}
\right\}+\frac{1}{n}\sum_{t=1}^n\left\{\frac{\partial
\epsilon_t(\theta)}{\partial\theta_i}
\frac{\partial^2\epsilon_t(\theta)}{\partial\theta_j\partial\theta_k}
\right\}\\
&+\frac{1}{n}\sum_{t=1}^n\left\{\frac{\partial^2\epsilon_t(\theta)}{\partial\theta_i\partial\theta_j}\frac{\partial
\epsilon_t(\theta)}{\partial\theta_k}\right\}+\frac{1}{n}\sum_{t=1}^n\left\{\frac{\partial
\epsilon_t(\theta)}{\partial\theta_j}
\frac{\partial^2\epsilon_t(\theta)}{\partial\theta_i\partial\theta_k}\right\}.
\end{align*}
Using the ergodic theorem, the Cauchy-Schwarz inequality and Lemma
\ref{lemme_prelim}, we obtain
\begin{align}\label{der3Ofini}
\sup_n\sup_{\theta\in\Theta}\left|\frac{\partial^3O_n(\theta)}{\partial\theta_i\partial\theta_j\partial\theta_k}\right|<+\infty.
\end{align}
In view of Proposition \ref{prop_prelim}, we have almost surely
$$\sup_{\theta\in\Theta}\left|\frac{\partial^3
}{\partial\theta_i\partial\theta_j\partial\theta_k}\left(Q_n(\theta)-O_n(\theta)\right)\right|\longrightarrow 0,
\text{ as } n\to\infty.$$ Thus ${\partial^3
Q_n(\theta)}/{\partial\theta_i\partial\theta_j\partial\theta_k}$ and
${\partial^2
O_n(\theta)}/{\partial\theta_i\partial\theta_j\partial\theta_k}$
have almost surely the same asymptotic distribution.
 In view of Theorem \ref{theo_consistency} and {\bf (A4)}, we have almost surely
$\hat{\theta}_n\longrightarrow \theta_0\in \stackrel{\circ}{\Theta}$. Thus
$\nabla Q_n(\hat{\theta}_n)=0_{\nbR^{(p+q)K}}$ for sufficiently
large $n$, and a Taylor expansion gives for all $r\in\{1,...,(p+q)K \}$,
\begin{equation}
\label{taylorintheo2} 0=\sqrt{n} \frac{\partial}{\partial \theta_r}
Q_n(\theta_0) +
\nabla\frac{\partial}{\partial \theta_r} Q_n(\theta_{n,r}^*) \sqrt{n}\left(\hat{\theta}_n-\theta_0\right),
\end{equation}
where $\theta_{n,r}^*$ lies on the segment in $\nbR^{(p+q)K}$ with endpoints $\hat{\theta}_n$ and $\theta_0$.
Using again a Taylor expansion, Theorem \ref{theo_consistency_vrai}
and \eqref{der3Ofini}, we obtain for all $l=1,\dots,(p+q)K$,
\begin{eqnarray*}
\left|\frac{\partial^2
Q_n(\theta_{n,r}^*)}{\partial \theta_l\partial
\theta_r}-\frac{\partial^2
Q_n(\theta_0)}{\partial \theta_l\partial
\theta_r}\right|&\leq&\sup_n\sup_{\theta\in\Theta}\left\|\nabla\left( \frac{\partial^2
}{\partial \theta_l\partial
\theta_r}Q_n(\theta)\right)\right\|\left\|\theta_{n,r}^*-\theta_0\right\|
\\&\longrightarrow & 0 \text{ a.s. as }n\to\infty.
\end{eqnarray*}
This, along with \eqref{taylorintheo2}, implies that, as $n\to\infty$
$$\sqrt{n}\left(\hat{\theta}_n-\theta_0\right)=-\left[ \nabla^2
Q_n(\theta_0) \right]^{-1}\sqrt{n}\frac{\partial
Q_n(\theta_0)}{\partial \theta}+o_\p(1).$$
From Lemma \ref{NAscore} and Lemma \ref{lemsup}, we obtain that
$\sqrt{n}(\hat{\theta}_n-\theta_0)$ has a limiting normal distribution with mean $0$ and covariance matrix $J^{-1}IJ^{-1}$.\zak
\subsection{Proofs of Theorem \ref{estimationI}}
The proof of Theorem~\ref{estimationI} is based on a series of lemmas.

Consider the regression of $\Upsilon_t$ on $\Upsilon_{t-1},\dots,\Upsilon_{t-r}$
defined by
\begin{equation}\label{regp}
  \Upsilon_t=\sum_{i=1}^{r}\Phi_{r,i}\Upsilon_{t-i}+u_{r,t},\qquad
\end{equation}
where $u_{r,t}$ is orthogonal to $\left\{\Upsilon_{t-1} \dots \Upsilon_{t-r}\right\}$ for
the $L^2$ inner product.
If $\Upsilon_{1},\dots,\Upsilon_{n}$ were observed, the least squares estimators of
$\underline{{\mathbf{\Phi}}}_{r}=\left(\Phi_{r,1}\cdots \Phi_{r,r}\right)$ and
$\Sigma_{u_r}=\mbox{Var}(u_{r,t})$ would be given by
$$\underline{\breve{\mathbf{\Phi}}}_{r}=\hat{\Sigma}_{{\Upsilon},\underline{{\Upsilon}}_{r}}
\hat{\Sigma}_{\underline{{\Upsilon}}_{r}}^{-1}\qquad\mbox{and}\qquad
\hat{\Sigma}_{\breve{u}_r}=\frac{1}n\sum_{t=1}^n
\left({\Upsilon}_t-\underline{\breve{\mathbf{\Phi}}}_{r}\underline{{\Upsilon}}_{r,t}\right)
\left({\Upsilon}_t-\underline{\breve{\mathbf{\Phi}}}_{r}\underline{{\Upsilon}}_{r,t}\right)'$$
where $\underline{{\Upsilon}}_{r,t}=({\Upsilon}_{t-1}' \cdots
{\Upsilon}_{t-r}')',$
$$\hat{\Sigma}_{{\Upsilon},\underline{{\Upsilon}}_{r}}=\frac{1}{n}\sum_{t=1}^n{\Upsilon}_t\underline{{\Upsilon}}_{r,t}',\qquad
\hat{\Sigma}_{\underline{{\Upsilon}}_{r}}=
\frac{1}{n}\sum_{t=1}^n\underline{{\Upsilon}}_{r,t}\underline{{\Upsilon}}_{r,t}',
$$
with by convention
${\Upsilon}_t=0$ when $t\leq 0$, and assuming
$\hat{\Sigma}_{\underline{{\Upsilon}}_{r}}$ is non singular (which holds true asymptotically).

Actually, we just observe $X_1,\dots,X_n$. The residuals  $\hat\epsilon_t:=e_t(\hat\theta_n)$ are then available
for $t=1,\dots,n$ and the vectors $\hat\Upsilon_t$ obtained by replacing $\theta_0$ by $\hat\theta_n$ in \eqref{upsilon}
are  available for $t=1,\dots,n$. We therefore define the least squares estimators of
$\underline{{\mathbf{\Phi}}}_{r}=\left(\Phi_{r,1}\cdots \Phi_{r,r}\right)$ and
$\Sigma_{u_r}=\mbox{Var}(u_{r,t})$ by
$$\underline{\hat{\mathbf{\Phi}}}_{r}=\hat{\Sigma}_{\hat{\Upsilon},\underline{\hat{\Upsilon}}_{r}}
\hat{\Sigma}_{\underline{\hat{\Upsilon}}_{r}}^{-1}\qquad\mbox{and}\qquad
\hat{\Sigma}_{\hat{u}_r}=\frac{1}n\sum_{t=1}^n
\left(\hat{\Upsilon}_t-\underline{\hat{\mathbf{\Phi}}}_{r}\underline{\hat{\Upsilon}}_{r,t}\right)
\left(\hat{\Upsilon}_t-\underline{\hat{\mathbf{\Phi}}}_{r}\underline{\hat{\Upsilon}}_{r,t}\right)'$$
where $\underline{\hat{\Upsilon}}_{r,t}=(\hat{\Upsilon}_{t-1}' \cdots
\hat{\Upsilon}_{t-r}')',$
$$\hat{\Sigma}_{\hat{\Upsilon},\underline{\hat{\Upsilon}}_{r}}=\frac{1}{n}\sum_{t=1}^n\hat{\Upsilon}_t\underline{\hat{\Upsilon}}_{r,t}',\qquad
\hat{\Sigma}_{\underline{\hat{\Upsilon}}_{r}}=
\frac{1}{n}\sum_{t=1}^n\underline{\hat{\Upsilon}}_{r,t}\underline{\hat{\Upsilon}}_{r,t}',
$$
with by convention
$\hat{\Upsilon}_t=0$ when $t\leq 0$, and assuming
$\hat{\Sigma}_{\underline{\hat{\Upsilon}}_{r}}$ is non singular (which holds true asymptotically).

We specify a bit more the matrix norm defined at the end of Section \ref{modelassumption} and we use in the sequel the multiplicative matrix norm defined by
\begin{equation}\label{def_norm_matrix}
\|A\|=\sup_{\|x\|\leq 1}\|Ax\|=\varrho^{1/2}(A'\bar{A}),
\end{equation}
where $A$ is a $\nbC^{d_1\times d_2}$ matrix,
$\|x\|^2=x' \bar{x}$
 is the Euclidean norm of the vector $x\in \nbC^{d_2\times 1}$, and $\varrho(\cdot)$ denotes the spectral radius. This norm
satisfies
\begin{equation}
\label{proprinorm}
\|A\|^2\leq \sum_{i,j}a_{i,j}^2, \text{ when }A \text{ is a }\nbR^{d_1\times d_2}\text{ matrix}
\end{equation}
 with obvious notations. This choice of the norm is crucial for the following lemma to hold (with e.g. the Euclidean norm, this result is not valid).
Let
\begin{eqnarray*}
{\Sigma}_{{\Upsilon},\underline{\Upsilon}_{r}}&=&\mathbb{E}{\Upsilon}_{t}\underline{\Upsilon}_{r,t}',
\quad{\Sigma}_{{\Upsilon}}=\mathbb{E}{\Upsilon}_{t}{\Upsilon}_{t}',
\quad
{\Sigma}_{\underline{\Upsilon}_{r}}=\mathbb{E}\underline{\Upsilon}_{r,t}\underline{\Upsilon}_{r,t}',\quad
\hat{\Sigma}_{\hat{\Upsilon}}=\frac{1}{n}\sum_{t=1}^n\hat{\Upsilon}_t\hat{\Upsilon}_{t}'.
\end{eqnarray*}
In the sequel, $C$ and $\rho$ denote generic constant such as $K>0$ and $\rho\in(0,1)$, whose exact values are unimportant.

\begin{lemme}\label{lem4}
Under the assumptions of Theorem~\ref{estimationI},
\begin{equation*}
\sup_{r\geq
1}\max\left\{\left\|{\Sigma}_{{\Upsilon},\underline{\Upsilon}_{r}}\right\|,\left\|{\Sigma}_{\underline{\Upsilon}_{r}}\right\|,
\left\|{\Sigma}_{\underline{\Upsilon}_{r}}^{-1}\right\|\right\}<
\infty.
\end{equation*}
\end{lemme}
{\bf Proof.} The proof is an extension of Section 5.2 of \cite{Grenander58}. 
 We readily have
$$\|{\Sigma}_{\underline{\Upsilon}_{r}}x\|
\leq
\|{\Sigma}_{\underline{\Upsilon}_{r+1}}(x',
0_{(p+q)K}')'\|\quad\mbox{ and
}\quad\|{\Sigma}_{\underline{\Upsilon}_{r}}x\| \leq
\|{\Sigma}_{\underline{\Upsilon}_{r+1}}(0_{(p+q)K}',x')'\|$$
for any $x\in\mathbb{R}^{K(p+q)r}$ and  $0_{(p+q)K}=(0,\dots,0)'\in\mathbb{R}^{(p+q)K}$.
Therefore
$$
0<\left\|\mbox{Var}\left({\Upsilon}_{t}\right)\right\|=
\left\|{\Sigma}_{\underline{\Upsilon}_{1}}\right\|\leq
\left\|{\Sigma}_{\underline{\Upsilon}_{2}}\right\|\leq\cdots$$
and
$$
\left\|{\Sigma}_{{\Upsilon},\underline{\Upsilon}_{r}}\right\|\leq
\left\|{\Sigma}_{\underline{\Upsilon}_{r+1}}\right\|,
$$
so that it suffices to prove that $ \sup_{r\ge 1}\left\|{\Sigma}_{\underline{\Upsilon}_{r}}\right\|$ and $\sup_{r\ge 1}\left\|{\Sigma}_{\underline{\Upsilon}_{r}}^{-1}\right\|$ are finite to prove the result. Let us write matrix ${\Sigma}_{\underline{\Upsilon}_{r}}$ in blockwise form
$$
{\Sigma}_{\underline{\Upsilon}_{r}}=\left[ C(i-j)\right]_{i,j=1,...,r},\quad C(k)=\e (\Upsilon_{0}\Upsilon_{k}')\in \nbR^{K(p+q)\times K(p+q)},\ k\in\nbZ .
$$
Let now $f:\nbR\longrightarrow \nbC^{K(p+q)\times K(p+q)}$  be the spectral density of $(\Upsilon_t)_{t\in \nbZ}$ defined by
$$
f(\omega)=\frac{1}{2\pi} \sum_{k=-\infty}^\infty C(k) e^{i\omega k},\quad \omega\in \nbR .
$$
A direct consequence of \eqref{upsilon} and Lemma \ref{existenceI} is that  $f(\omega)$ is absolutely summable, and that $\sup_{\omega\in\nbR}\|f(\omega)\|<+\infty$, for any norm $\|.\|$ on $\nbC^{K(p+q)\times K(p+q)}$ (in particular, one which is independent from $r\ge 1$). Another consequence is that we have the inversion formula
\begin{equation}\label{inversion_formula}
C(k)=\int_{-\pi}^\pi f(x) e^{-ikx}dx,\quad \forall k\in\nbZ .
\end{equation}
Last, it is easy to check that $f(\omega)$ is an hermitian matrix for all $\omega\in\nbR$, i.e. $\overline{f(\omega)}=f(\omega) '$, where $\bar{z}$ is the conjugate of any vector or matrix $z$ with entries in $\nbC$. Let then $\delta^{(r)}=\left({\delta^{(r)}_1} ',...,{\delta^{(r)}_r} '\right)\in \nbR^{rK(p+q)\times 1}$ be an eigenvector for ${\Sigma}_{{\underline{\Upsilon}_{r}}}$, with $\delta^{(r)}_j \in \nbR^{K(p+q)\times 1}$, $j=1,...,r$, such that $\| {\delta^{(r)}}\|=1$ and
\begin{equation}\label{Parseval0}
{\delta^{(r)}} ' {\Sigma}_{{\underline{\Upsilon}_{r}}} \delta^{(r)}= \|{\Sigma}_{{\underline{\Upsilon}_{r}}}\|=\varrho\left( {\Sigma}_{{\underline{\Upsilon}_{r}}}\right),
\end{equation}
where $\|{\Sigma}_{{\underline{\Upsilon}_{r}}}\|$ is the norm of matrix ${\Sigma}_{{\underline{\Upsilon}_{r}}}$ defined in \eqref{def_norm_matrix}. We then check that
\begin{equation}\label{Parseval}
{\delta^{(r)}} ' {\Sigma}_{{\underline{\Upsilon}_{r}}} \delta^{(r)}=\sum_{i,j=1}^r {\delta^{(r)}_i} ' C(i-j) {\delta^{(r)}_j}=\int_{-\pi}^\pi \left(\sum_{m=1}^r  \delta^{(r)}_m e^{i(m-1)x}\right)' f(x) \overline{\left(\sum_{m=1}^r  \delta^{(r)}_m e^{i(m-1)x}\right)}dx ,
\end{equation}
the last equality a direct consequence of \eqref{inversion_formula}. $f(x)$ being hermitian, $(X,Y)\in \nbC^{K(p+q)\times 1}\times \nbC^{K(p+q)\times 1}\mapsto X' f(x) \bar{Y}$ defines a semi definite non negative bilinear form, hence we have for all $x\in \nbR$ and $X\in \nbC^{K(p+q)\times 1}$:
$$
0\le X' f(x) \bar{X}\le \| f(x) \|. X'\bar{X} \le \sup_{\omega\in\nbR}\|f(\omega)\| . X'\bar{X} .
$$
Let us point out that $\sup_{\omega\in\nbR}\|f(\omega)\|$ is a quantity which is independent from $r\ge 1$. We deduce from \eqref{Parseval} and the previous inequality that
\begin{equation}\label{Parseval2}
{\delta^{(r)}} ' {\Sigma}_{{\underline{\Upsilon}_{r}}} \delta^{(r)}\le \sup_{\omega\in\nbR}\|f(\omega)\| \int_{-\pi}^\pi \left(\sum_{m=1}^r  \delta^{(r)}_m e^{i(m-1)x}\right)'  \overline{\left(\sum_{m=1}^r  \delta^{(r)}_m e^{i(m-1)x}\right)}dx .
\end{equation}
A short computation yields that
$$\frac{1}{2\pi}\int_{-\pi}^\pi \left(\sum_{m=1}^r  \delta^{(r)}_m e^{i(m-1)x}\right)'  \overline{\left(\sum_{m=1}^r  \delta^{(r)}_m e^{i(m-1)x}\right)}dx = \sum_{m=1}^r {\delta^{(r)}_m} ' \delta^{(r)}_m=\| {\delta^{(r)}}\|^2=1,$$
which, coupled with \eqref{Parseval0} and \eqref{Parseval2}, yields that $ \|{\Sigma}_{{\underline{\Upsilon}_{r}}}\|\le 2\pi \sup_{\omega\in\nbR}\|f(\omega)\|<+\infty$, an upper bound independent from $r\ge 1$. By similar arguments, the smallest eigenvalue of $
{\Sigma}_{\underline{\Upsilon}_{r}}$ is
greater than a positive constant independent of $r$. Using the
fact that
$\|{\Sigma}_{\underline{\Upsilon}_{r}}^{-1}\|$
is equal to the inverse of the smallest eigenvalue of $
{\Sigma}_{\underline{\Upsilon}_{r}}$, the
proof  is completed.
 \zak
The following lemma is necessary in the sequel. 
\begin{lemme}\label{lemme_prelimbis}
Let us suppose that {\bf (A1)} and that Stationarity condition {\bf (A5a)} for $\nu=6$
$$
{\bf (A6)}\limsup_{t\to\infty}\frac{1}{t}\ln  \nbE\left( \sup_{\theta\in \Theta} \left| \left| \prod_{i=1}^t \Phi(\Delta_i,\theta)\right|\right|^{32}\right)<0,\quad \limsup_{t\to\infty}\frac{1}{t}\ln  \nbE\left( \left|\left|\prod_{i=1}^t \Psi(\Delta_i)\right|\right|^{32}\right)<0
$$ hold. We assume that $\epsilon_t\in L^{4\nu+8}$. Sequences $(\epsilon_t(\theta))_{t\in\nbZ}$ and $(e_t(\theta))_{t\in\nbZ}$ satisfy
\begin{enumerate}
\item $\left|\left| \sup_{\theta\in \Theta} |\epsilon_0(\theta)|\right| \right|_{16}<+\infty$ and $\sup_{t\ge 0}\left|\left| \sup_{\theta\in \Theta} |e_t(\theta)|\right| \right|_{16}<+\infty$,
\item $\left|\left| \sup_{\theta\in \Theta} |\epsilon_t(\theta)-e_t(\theta)|\right| \right|_4$ tends to $0$ exponentially fast as $t\to \infty$,
\item For all $\alpha>0$, $t^\alpha \sup_{\theta\in \Theta} |\epsilon_t(\theta)-e_t(\theta)|\longrightarrow 0$ a.s. as $t\to\infty$,
\item For all $j=1, 2,3$, $\left|\left| \sup_{\theta\in \Theta} ||\nabla^j\epsilon_0(\theta)||\right| \right|_{16}<+\infty$, $\sup_{t\ge 0}\left|\left| \sup_{\theta\in \Theta} ||\nabla^j e_t(\theta)||\right| \right|_{16}<+\infty$ and we have $t^\alpha\left|\left| \sup_{\theta\in \Theta} ||\nabla (e_t- \epsilon_t)(\theta)||\right| \right|_{16/5}\longrightarrow 0$ , as $t\to\infty$ for all $\alpha>0$.
\end{enumerate}
\end{lemme}
{\bf Proof of Lemma \ref{lemme_prelimbis}} is similar to the proofs of Lemmas \ref{c_i^2_expo_decrease} and \ref{lemme_prelim}.\zak

Denote by $\Upsilon_t(i)$ the $i$-th element of $\Upsilon_t.$
\begin{lemme}\label{lemsurprise!}
Let $(\epsilon_t)$ be a sequence of centered and uncorrelated variables, with $\e\left|\epsilon_t\right|^{8+4\nu}<\infty$ and $\sum_{h=0}^\infty\left[\alpha_\epsilon(h)\right]^{\nu/(2+\nu)}<\infty$
for some $\nu>0$. Then there exits a finite constant $C_1$ such that for
$m_1, m_2=1,\dots,(p+q)K$ and all $s\in \mathbb{Z}$,
\begin{equation*}
\sum_{h=-\infty}^{\infty}\left|\mbox{Cov}\left\{\Upsilon_{1}(m_1)\Upsilon_{1+s}(m_2),
\Upsilon_{1+h}(m_1)\Upsilon_{1+s+h}(m_2)\right\}\right|<C_1.
\end{equation*}
\end{lemme}
{\bf Proof.} Recall that
\begin{eqnarray}\label{der-eps}\frac{\partial\epsilon_t(\theta_0)}{\partial\theta_l}&=&\sum_{i=0}^\infty c_{i,l}(\theta_0,\Delta_{t},\dots,\Delta_{t-i+1}) \epsilon_{t-i},\text{ for }l=1,\dots,(p+q)K,
\end{eqnarray}
where $c_i(\theta_0,\Delta_{t},\dots,\Delta_{t-i+1})$ is defined by \eqref{expression_eps_theta} and $c_{i,l}(\theta_0,\Delta_{t},\dots,\Delta_{t-i+1})={\partial c_i(\theta_0,\Delta_{t},\dots,\Delta_{t-i+1})}/{\partial \theta_l}$, and
with the following upper bound holding thanks to \eqref{coeff_c_i_e_expo_decrease}:
\begin{eqnarray*}
\mathbb{E}\sup_{\theta\in\Theta}(c_i(\theta,\Delta_{t},\dots,\Delta_{t-i+1}))^2\leq C\rho^i \text{ and }\mathbb{E}\sup_{\theta\in\Theta}(
c_{i,l}(\theta,\Delta_{t},\dots,\Delta_{t-i+1}))^2\leq C\rho^i,\quad \forall i.
\end{eqnarray*}
Let
\begin{eqnarray}\label{gammaCov}
\gamma_{i,j,i',j',s,h}(m_1,m_2)(\theta_0)&=& \e\left[c_{i,m_1}(\theta_0,\Delta_{t},\dots,\Delta_{t-i+1})c_{j,m_2}(\theta_0,\Delta_{t+s},\dots,\Delta_{t+s-j+1})
\right.\nonumber\\&&\left.\times c_{i',m_1}(\theta_0,\Delta_{t+h},\dots,\Delta_{t+h-i'+1})c_{j',m_2}(\theta_0,\Delta_{t+s+h},\dots,\Delta_{t+s+h-j'+1})\right]\nonumber\\&&\times\mbox{Cov}\left(\epsilon_{t}\epsilon_{t-i}
\epsilon_{t+s}\epsilon_{t+s-j},\epsilon_{t+h}\epsilon_{t+h-i'}\epsilon_{t+s+h}\epsilon_{t+s+h-j'}\right)\nonumber\\&&+
\mbox{Cov}\left(c_{i,m_1}(\theta_0,\Delta_{t},\dots,\Delta_{t-i+1})c_{j,m_2}(\theta_0,\Delta_{t+s},\dots,\Delta_{t+s-j+1}),
\right.\nonumber\\&&\left. c_{i',m_1}(\theta_0,\Delta_{t+h},\dots,\Delta_{t+h-i'+1})c_{j',m_2}(\theta_0,\Delta_{t+s+h},\dots,\Delta_{t+s+h-j'+1})\right)
\nonumber\\&&\times\e\left[\epsilon_{t}\epsilon_{t-i}
\epsilon_{t+s}\epsilon_{t+s-j}\right]\e\left[\epsilon_{t+h}\epsilon_{t+h-i'}\epsilon_{t+s+h}\epsilon_{t+s+h-j'}\right].
\end{eqnarray}
The Cauchy-Schwarz inequality implies that
\begin{eqnarray}
\nonumber&&\left|\e[c_{i,m_1}(\theta_0,\Delta_{t},\dots,\Delta_{t-i+1})c_{j,m_2}(\theta_0,\Delta_{t+s},\dots,\Delta_{t+s-j+1})
\right.\\&&\left.\times c_{i',m_1}(\theta_0,\Delta_{t+h},\dots,\Delta_{t+h-i'+1})c_{j',m_2}(\theta_0,\Delta_{t+s+h},\dots,\Delta_{t+s+h-j'+1})]\right|\leq
C\rho^{i+j+i'+j'}.\label{c_igama-Cauchy-Schw}
\end{eqnarray}
In view of \eqref{der-eps} and \eqref{gammaCov}, we have
\begin{eqnarray*}
&&\sum_{h=-\infty}^{\infty}\mbox{Cov}\left\{\Upsilon_{1}(m_1)\Upsilon_{1+s}(m_2),
\Upsilon_{1+h}(m_1)\Upsilon_{1+s+h}(m_2)\right\}\\
&=& \sum_{h=-\infty}^{\infty}\sum_{i=0}^\infty\sum_{j=0}^\infty\sum_{i'=0}^\infty\sum_{j'=0}^\infty\gamma_{i,j,i',j',s,h}(m_1,m_2)(\theta_0).
\end{eqnarray*}
Without loss of generality, we can take the supremum
over the integers $s>0$, and consider the sum for positive $h$. Let  $m_0=m_1\wedge m_2$
and $Y_{t,h_1}=\epsilon_{t}\epsilon_{t-h_1}-\mathbb{E}(\epsilon_{t}\epsilon_{t-h_1})$.
We first suppose that $h\geq0$. It follows that
\begin{eqnarray*}
&&\sum_{i=0}^\infty\sum_{j=0}^\infty\sum_{i'=0}^\infty\sum_{j'=0}^\infty\left|\mbox{Cov}\left(c_{i,m_1}(\theta_0,\Delta_{t},\dots,\Delta_{t-i+1})c_{j,m_2}(\theta_0,\Delta_{t+s},\dots,\Delta_{t+s-j+1}),
\right.\right.\\&&\left.\left. c_{i',m_1}(\theta_0,\Delta_{t+h},\dots,\Delta_{t+h-i'+1})c_{j',m_2}(\theta_0,\Delta_{t+s+h},\dots,\Delta_{t+s+h-j'+1})\right)\right|\\&\leq&v_1+v_2+v_3+v_4+v_5,
\end{eqnarray*}
where
\begin{eqnarray*}
v_1=v_1(h)&=&\sum_{i>[h/2]}\sum_{j=0}^\infty\sum_{i'=0}^\infty\sum_{j'=0}^\infty\left|\mbox{Cov}\left(\mathbf{c}^{t}_{i,m_1}\mathbf{c}^{t+s}_{j,m_2},\mathbf{c}^{t+h}_{i',m_1}\mathbf{c}^{t+h+s}_{j',m_2}
\right)\right|,\\
v_2=v_2(h)&=&\sum_{i=0}^\infty\sum_{j>[h/2]}\sum_{i'=0}^\infty\sum_{j'=0}^\infty\left|\mbox{Cov}\left(\mathbf{c}^{t}_{i,m_1}\mathbf{c}^{t+s}_{j,m_2},\mathbf{c}^{t+h}_{i',m_1}\mathbf{c}^{t+h+s}_{j',m_2}\right)\right|\\
v_3=v_3(h)&=&\sum_{i=0}^\infty\sum_{j=0}^\infty\sum_{i'>[h/2]}\sum_{j'=0}^\infty\left|\mbox{Cov}\left(\mathbf{c}^{t}_{i,m_1}\mathbf{c}^{t+s}_{j,m_2},\mathbf{c}^{t+h}_{i',m_1}\mathbf{c}^{t+h+s}_{j',m_2}\right)\right|,\\
v_4=v_4(h)&=&\sum_{i=0}^\infty\sum_{j=0}^\infty\sum_{i'=0}^\infty\sum_{j'>[h/2]}\left|\mbox{Cov}\left(\mathbf{c}^{t}_{i,m_1}\mathbf{c}^{t+s}_{j,m_2},\mathbf{c}^{t+h}_{i',m_1}\mathbf{c}^{t+h+s}_{j',m_2}\right)\right|,
\\v_5=v_5(h)&=&\sum_{i=0}^{[h/2]}\sum_{j=0}^{[h/2]}\sum_{i'=0}^{[h/2]}\sum_{j'=0}^{[h/2]}\left|\mbox{Cov}\left(\mathbf{c}^{t}_{i,m_1}\mathbf{c}^{t+s}_{j,m_2},\mathbf{c}^{t+h}_{i',m_1}\mathbf{c}^{t+h+s}_{j',m_2}\right)\right|,
\end{eqnarray*}
where $$\mathbf{c}^{t}_{i_1,m}=c_{i_1,m}(\theta_0,\Delta_{t},\dots,\Delta_{t-i_1+1}).$$
One immediate remark is that $\mathbf{c}^{t}_{i_1,m}$ is measurable with respect to $\Delta_r$, $r\in \{ t,...,t-i_1+1\}$. Since
\begin{eqnarray*}
&&\left|\mbox{Cov}\left(\mathbf{c}^{t}_{i,m_1}\mathbf{c}^{t+s}_{j,m_2},\mathbf{c}^{t+h}_{i',m_1}\mathbf{c}^{t+h+s}_{j',m_2}\right)\right|\leq C\rho^{i+i'+j+j'},
\end{eqnarray*}
we have
\begin{eqnarray*}
v_1&=&\sum_{i>[h/2]}\sum_{j=0}^\infty\sum_{i'=0}^\infty\sum_{j'=0}^\infty\left|\mbox{Cov}\left(\mathbf{c}^{t}_{i,m_1}\mathbf{c}^{t+s}_{j,m_2},\mathbf{c}^{t+h}_{i',m_1}\mathbf{c}^{t+h+s}_{j',m_2}\right)\right|
\leq \kappa_1\rho^{h/2},
\end{eqnarray*}
for some positive constant $\kappa_1$. Using the same arguments we obtain that $v_i$, $i=2,3,4$ are bounded by $\kappa_i\rho^{h/2}$. The $\alpha-$mixing property (see Theorem 14.1 in \cite{davidson1994}, p. 210) and Lemmas \ref{Davydov} and \ref{lemme_prelimbis}, entail that
\begin{eqnarray*}
v_5&=&\sum_{i=0}^{[h/2]}\sum_{j=0}^{[h/2]}\sum_{i'=0}^{[h/2]}\sum_{j'=0}^{[h/2]}\left|\mbox{Cov}\left(
\mathbf{c}^{t}_{i,m_1}\mathbf{c}^{t+s}_{j,m_2},\mathbf{c}^{t+h}_{i',m_1}\mathbf{c}^{t+h+s}_{j',m_2}\right)\right|
\\
&\leq& \sum_{k=1}^{4}\sum_{(i,j,i',j')\in\mathcal{C}_k}\kappa_6
\left\|\mathbf{c}^{t}_{i,m_1}\mathbf{c}^{t+s}_{j,m_2}\right\|_{2+\nu}
\left\|\mathbf{c}^{t+h}_{i',m_1}\mathbf{c}^{t+h+s}_{j',m_2}\right\|_{2+\nu}
\left\{\alpha\left(\mathbf{c}^{t}_{i,m_1}\mathbf{c}^{t+s}_{j,m_2},\mathbf{c}^{t+h}_{i',m_1}\mathbf{c}^{t+h+s}_{j',m_2}\right)\right\}^{\nu/(2+\nu)},
\end{eqnarray*}
where $\alpha(U,V)$ denotes the strong mixing coefficient between
the $\sigma-$field generated by the random variable $U$ and that generated by $V$ and where
\begin{eqnarray*}
\mathcal{C}_1= \mathcal{C}_1(h)&=&\left\{(i,j,i',j')\in\{0,1,\dots,[h/2]\}^4: i\geq j-s,\;j'\leq i'+s\right\},\\
\mathcal{C}_2= \mathcal{C}_2(h)&=&\left\{(i,j,i',j')\in\{0,1,\dots,[h/2]\}^4: i\geq j-s,\;j'\geq i'+s\right\},\\
\mathcal{C}_3= \mathcal{C}_3(h)&=&\left\{(i,j,i',j')\in\{0,1,\dots,[h/2]\}^4: i\leq j-s,\;j'\leq i'+s\right\},\\
\mathcal{C}_4= \mathcal{C}_4(h)&=&\left\{(i,j,i',j')\in\{0,1,\dots,[h/2]\}^4: i\leq j-s,\;j'\geq i'+s\right\}.
\end{eqnarray*}
We check easily that $\mathbf{c}^{t}_{i,m_1}\mathbf{c}^{t+s}_{j,m_2}$ and $\mathbf{c}^{t+h}_{i',m_1}\mathbf{c}^{t+h+s}_{j',m_2}$ are respectively measurable with respect to $\Delta_r$, $r\in \{t-i+1,...,t+s\}$ and $\Delta_r$, $r\in \{t-i'+h+1,...,t+h+s\}$ when  $(i,j,i',j')\in\mathcal{C}_1$. 
We have $t-i+1\leq t+s-j+1$, $t+h-i'+1\leq t+h+s-j'+1$ and we thus deduce that
\begin{eqnarray*}
\left|\alpha\left(\mathbf{c}^{t}_{i,m_1}\mathbf{c}^{t+s}_{j,m_2},\mathbf{c}^{t+h}_{i',m_1}\mathbf{c}^{t+h+s}_{j',m_2}\right)\right|&\leq&\alpha_{\Delta}\left(h-i'-s+1\right),\quad\forall h\geq i'+s-1,
\\
\left|\alpha\left(\mathbf{c}^{t}_{i,m_1}\mathbf{c}^{t+s}_{j,m_2},\mathbf{c}^{t+h}_{i',m_1}\mathbf{c}^{t+h+s}_{j',m_2}\right)\right|&\leq&\alpha_{\Delta}\left(-i-h-s+1\right),\quad\forall h\leq -i-s+1,
\\
\left|\alpha\left(\mathbf{c}^{t}_{i,m_1}\mathbf{c}^{t+s}_{j,m_2},\mathbf{c}^{t+h}_{i',m_1}\mathbf{c}^{t+h+s}_{j',m_2}\right)\right|&\leq&\alpha_{\Delta}\left(0\right)\leq 1/4,\quad\forall h= -i-s+1,\dots ,i'+s-1.
\end{eqnarray*}
Note also that, by the H\"older inequality,
$$\left\|\mathbf{c}^{t}_{i,m_1}\mathbf{c}^{t+s}_{j,m_2}\right\|_{2+\nu}\leq\left\|\mathbf{c}^{t}_{i,m_1}\right\|_{4+2\nu}\left\|\mathbf{c}^{t+s}_{j,m_2}\right\|_{4+2\nu}\leq C\rho^{i+j}.$$
Therefore
\begin{eqnarray*}
&&\sum_{h=0}^\infty\sum_{(i,j,i',j')\in\mathcal{C}_1}\left\|\mathbf{c}^{t}_{i,m_1}\mathbf{c}^{t+s}_{j,m_2}\right\|_{2+\nu}
\left\|\mathbf{c}^{t+h}_{i',m_1}\mathbf{c}^{t+h+s}_{j',m_2}\right\|_{2+\nu}
\left\{\alpha\left(\mathbf{c}^{t}_{i,m_1}\mathbf{c}^{t+s}_{j,m_2},\mathbf{c}^{t+h}_{i',m_1}\mathbf{c}^{t+h+s}_{j',m_2}\right)\right\}^{\nu/(2+\nu)},
\\&\leq&C^2\sum_{i,j,i',j'=0}^\infty \rho^{i+j+i'+j'}\left(i'+2s-1+i+\sum_{r=0}^{\infty}\alpha_{\Delta}^{\nu/(2+\nu)}\left(r \right)\right)<\infty.
\end{eqnarray*}
Continuing in this way, we obtain that $\sum_{h=0}^{\infty}v_5(h)<\infty$. It follows that
\begin{eqnarray}\nonumber
&&\sum_{h=0}^{\infty}\sum_{i=0}^\infty\sum_{j=0}^\infty\sum_{i'=0}^\infty\sum_{j'=0}^\infty\left|\mbox{Cov}\left(c_{i,m_1}(\theta_0,\Delta_{t},\dots,\Delta_{t-i+1})c_{j,m_2}(\theta_0,\Delta_{t+s},\dots,\Delta_{t+s-j+1}),
\right.\right.\\
\nonumber&&\left.\left. c_{i',m_1}(\theta_0,\Delta_{t+h},\dots,\Delta_{t+h-i'+1})c_{j',m_2}(\theta_0,\Delta_{t+s+h},\dots,\Delta_{t+s+h-j'+1})\right)\right|\\
&\leq & \sum_{h=0}^{\infty} \sum_{i=1}^5 v_i(h)<\infty .
\label{somme4c_i}
\end{eqnarray}
The
same bounds clearly holds for
\begin{eqnarray*}\nonumber
&&\sum_{h=-\infty}^{0}\sum_{i=0}^\infty\sum_{j=0}^\infty\sum_{i'=0}^\infty\sum_{j'=0}^\infty\left|\mbox{Cov}\left(c_{i,m_1}(\theta_0,\Delta_{t-1},\dots,\Delta_{t-i})c_{j,m_2}(\theta_0,\Delta_{t+s-1},\dots,\Delta_{t+s-j}),
\right.\right.\\ \nonumber&&\left.\left. c_{i',m_1}(\theta_0,\Delta_{t+h},\dots,\Delta_{t+h-i'+1})c_{j',m_2}(\theta_0,\Delta_{t+s+h},\dots,\Delta_{t+s+h-j'+1})\right)\right|<\infty,
\end{eqnarray*}
which shows that
\begin{eqnarray*}
&&\sum_{h=-\infty}^{\infty}\sum_{i=0}^\infty\sum_{j=0}^\infty\sum_{i'=0}^\infty\sum_{j'=0}^\infty\left|\mbox{Cov}\left(c_{i,m_1}(\theta_0,\Delta_{t},\dots,\Delta_{t-i+1})c_{j,m_2}(\theta_0,\Delta_{t+s},\dots,\Delta_{t+s-j+1}),
\right.\right.\\&&\left.\left. c_{i',m_1}(\theta_0,\Delta_{t+h},\dots,\Delta_{t+h-i'+1})c_{j',m_2}(\theta_0,\Delta_{t+s+h},\dots,\Delta_{t+s+h-j'+1})\right)\right|<\infty.
\end{eqnarray*}
A slight extension of Corollary A.3 in \cite{FZ2010_book} shows that
\begin{eqnarray}
\sum_{h=-\infty}^{\infty}\left|\mbox{Cov}\left(Y_{1,i}Y_{1+s,j},
Y_{1+h,i'}Y_{1+s+h,j'}\right)\right|<\infty.\label{corolA3FZ}
\end{eqnarray}
Because, by Cauchy-Schwarz inequality
\begin{eqnarray*}
\left|\e\left[\epsilon_{t}\epsilon_{t-i}
\epsilon_{t+s}\epsilon_{t+s-j}\right]\right|\leq\e\left|\epsilon_t\right|^4<\infty
\end{eqnarray*}
by the assumption that $\e\left|\epsilon_t\right|^{8+4\nu}<\infty$ and in view of \eqref{c_igama-Cauchy-Schw} it follows that
\begin{eqnarray*}
&&\sum_{h=-\infty}^{\infty}\left|\mbox{Cov}\left\{\Upsilon_{1}(m_1)\Upsilon_{1+s}(m_2),
\Upsilon_{1+h}(m_1)\Upsilon_{1+s+h}(m_2)\right\}\right|\\
&\leq& \kappa\sum_{i=0}^\infty\sum_{j=0}^\infty\sum_{i'=0}^\infty\sum_{j'=0}^\infty\rho^{i+j+i'+j'}\sum_{h=-\infty}^{\infty}\left|\mbox{Cov}\left(Y_{1,i}Y_{1+s,j},
Y_{1+h,i'}Y_{1+s+h,j'}\right)\right|\\&&+ \kappa'\sum_{i=0}^\infty\sum_{j=0}^\infty\sum_{i'=0}^\infty\sum_{j'=0}^\infty\sum_{h=-\infty}^{\infty}\left|\mbox{Cov}\left(c_{i,m_1}(\theta_0,\Delta_{t},\dots,\Delta_{t-i+1})c_{j,m_2}(\theta_0,\Delta_{t+s},\dots,\Delta_{t+s-j+1}),
\right.\right.\\&&\left.\left. c_{i',m_1}(\theta_0,\Delta_{t+h},\dots,\Delta_{t+h-i'+1})c_{j',m_2}(\theta_0,\Delta_{t+s+h},\dots,\Delta_{t+s+h-j'+1})\right)\right|
\end{eqnarray*}
The conclusion follows from \eqref{somme4c_i} and \eqref{corolA3FZ}.\zak

Let $\hat{\Sigma}_{\Upsilon}$ be the matrix obtained by replacing $\hat{\Upsilon}_t$ by $\Upsilon_t$ in $\hat{\Sigma}_{\hat{\Upsilon}}$.
\begin{lemme}\label{lem3}
Under the assumptions of Theorem~\ref{estimationI}, $
\sqrt{r}\|\hat{\Sigma}_{\underline{\Upsilon}_{r}}-
{\Sigma}_{\underline{\Upsilon}_{r}}\|$,
$\sqrt{r}\|\hat{\Sigma}_{{\Upsilon}}-
{\Sigma}_{{\Upsilon}}\|,$
 and $\sqrt{r}\|\hat{\Sigma}_{{\Upsilon},\underline{\Upsilon}_{r}}-
{\Sigma}_{{\Upsilon},\underline{\Upsilon}_{r}}\| $ tend to zero
in probability as $n\to\infty$ when $r=\mathrm{o}(n^{1/3})$.
\end{lemme}
{\bf Proof.}  For $1\leq m_1,m_2\leq K(p+q)$ and $1\leq r_1,r_2\leq
r$, the element of the $\left\{(r_1-1)(p+q)K+m_1\right\}$-th row and
$\left\{(r_2-1)(p+q)K+m_2\right\}$-th column of
$\hat{\Sigma}_{\underline{\Upsilon}_{r}}
$ is of the form $n^{-1}\sum_{t=1}^nZ_t$ where $Z_t:=Z_{t,r_1,r_2}(m_1,m_2)=\Upsilon_{t-r_1}(m_1)\Upsilon_{t-r_2}(m_2).$
 By stationarity of
$\left(Z_t\right)$, we have
\begin{eqnarray}
\mbox{Var}\left(\frac{1}{n}\sum_{t=1}^nZ_t\right)=\frac{1}{n^{2}}\sum_{h=-n+1}^{n-1}\left(n-|h|\right)\mbox{Cov}\left(Z_t,Z_{t-h}\right)
\leq
\frac{C_1}{n},\label{eqm}
\end{eqnarray}
where,  by Lemma \ref{lemsurprise!},  $C_1$ is a constant independent of $r_1,r_2,m_1,m_2$ and $r,n$.
Now using the Tchebychev inequality,  we have
\begin{eqnarray*}
\forall \beta>0, \quad \mathbb{P}\left\{\sqrt{r}\|\hat{\Sigma}_{\underline{\Upsilon}_{r}}-
{\Sigma}_{\underline{\Upsilon}_{r}}\|> \beta\right\}\leq
\frac{1}{\beta^2}\mathbb{E}\left\{r\|\hat{\Sigma}_{\underline{\Upsilon}_{r}}-
{\Sigma}_{\underline{\Upsilon}_{r}}\|^2\right\}.
\end{eqnarray*}
In view of (\ref{proprinorm}) and (\ref{eqm}) we have
\begin{eqnarray*}
&&\mathbb{E}\left\{r\|\hat{\Sigma}_{\Upsilon}-
{\Sigma}_{{\Upsilon}}\|^2 \right\}\leq
\mathbb{E}\left\{r\|\hat{\Sigma}_{{\Upsilon},\underline{\Upsilon}_{r}}-
{\Sigma}_{{\Upsilon},\underline{\Upsilon}_{r}}\|^2 \right\}
\\&&\leq \mathbb{E}\left\{r\|\hat{\Sigma}_{\underline{\Upsilon}_{r}}-
{\Sigma}_{\underline{\Upsilon}_{r}}\|^2\right\}\leq r \sum_{m_1,m_2=1}^{K(p+q)r}\mbox{Var}\left(\frac{1}{n}\sum_{t=1}^nZ_t\right)\leq
\frac{C_1K^2(p+q)^2r^3}{n}=\mathrm{o}(1)
\end{eqnarray*}
 as $n\to\infty$ when
$r=\mathrm{o}(n^{1/3})$. Hence, when $r=\mathrm{o}(n^{1/3})$
\begin{eqnarray*}
\sqrt{r}\|\hat{\Sigma}_{\underline{\Upsilon}_{r}}-
{\Sigma}_{\underline{\Upsilon}_{r}}\|&=&\mathrm{o}_{\mathbb{P}}(1),
\\\sqrt{r}\|\hat{\Sigma}_{{\Upsilon}}-
{\Sigma}_{{\Upsilon}}\|&=&\mathrm{o}_{\mathbb{P}}(1)\text{ and }\sqrt{r}\|\hat{\Sigma}_{{\Upsilon},\underline{\Upsilon}_{r}}-
{\Sigma}_{{\Upsilon},\underline{\Upsilon}_{r}}\|=\mathrm{o}_{\mathbb{P}}(1).
\end{eqnarray*}
The proof is complete.\zak
We now show that the previous lemma applies when $\Upsilon_t$ is replaced by $\hat{\Upsilon}_t$.
\begin{lemme}\label{lem3bis}
Under the assumptions of Theorem~\ref{estimationI}, $
\sqrt{r}\|\hat{\Sigma}_{\underline{\hat{\Upsilon}}_{r}}-
{\Sigma}_{\underline{\Upsilon}_{r}}\|$,
$\sqrt{r}\|\hat{\Sigma}_{{\hat{\Upsilon}}}-
{\Sigma}_{\Upsilon}\|,$
 and $\sqrt{r}\|\hat{\Sigma}_{{\hat{\Upsilon}},\underline{\hat{\Upsilon}}_{r}}-
{\Sigma}_{{\Upsilon},\underline{\Upsilon}_{r}}\| $ tend to zero
in probability as $n\to\infty$ when $r=\mathrm{o}(n^{1/3})$.
\end{lemme}
{\bf Proof.} We first show that the replacement of the unknown initial values $\{X_u,\;u\leq 0\}$ by zero is asymptotically unimportant. Let $\hat{\Sigma}_{\underline{\Upsilon}_{r,n}}$ be the matrix obtained by replacing $e_t(\hat\theta_n)$ by $\epsilon_t(\hat\theta_n)$ in
$\hat{\Sigma}_{{\hat{\underline{\Upsilon}}_r}}$. We start by evaluating $\mathbb{E}\|\hat{\Sigma}_{{\hat{\underline{\Upsilon}}_r}}-\hat{\Sigma}_{\underline{\Upsilon}_{r,n}}\|^2$.
We first note that
$$\hat{\Sigma}_{{\hat{\underline{\Upsilon}}_r}}-\hat{\Sigma}_{\underline{\Upsilon}_{r,n}}=\left[\frac{1}{n}\sum_{t=1}^na_{t-i,t-i',m_1,m_2}(\hat\theta_n)\right]$$
for $i,i'=1,\dots,r$ and $m_1,m_2=1,\dots,K(p+q)$ and where $$a_{t-i,t-i',m_1,m_2}(\hat\theta_n)=e_{t-i}(\hat\theta_n)
e_{t-i'}(\hat\theta_n)\frac{\partial e_{t-i}(\hat\theta_n)}{\partial \theta_{m_1}}\frac{\partial e_{t-i'}(\hat\theta_n)}{\partial \theta_{m_2}}
-\epsilon_{t-i}(\hat\theta_n)
\epsilon_{t-i'}(\hat\theta_n)\frac{\partial \epsilon_{t-i}(\hat\theta_n)}{\partial \theta_{m_1}}\frac{\partial \epsilon_{t-i'}(\hat\theta_n)}{\partial \theta_{m_2}}.$$
 Using \eqref{proprinorm}, we have
\begin{eqnarray*}
\| \hat{\Sigma}_{{\hat{\underline{\Upsilon}}_r}}-\hat{\Sigma}_{\underline{\Upsilon}_{r,n}}\|^2&\leq&\sum_{i,i'=1}^r\sum_{m_1,m_2=1}^{K(p+q)}
\left[\frac{1}{n}\sum_{t=1}^na_{t-i,t-i',m_1,m_2}(\hat\theta_n)\right]^2.
\end{eqnarray*}
We thus deduce the following $L^2$ estimate:
\begin{eqnarray*}
\mathbb{E}\| \hat{\Sigma}_{{\hat{\underline{\Upsilon}}_r}}-\hat{\Sigma}_{\underline{\Upsilon}_{r,n}}\|^2&\leq&\sum_{i,i'=1}^r\sum_{m_1,m_2=1}^{K(p+q)}\left\|
\frac{1}{n}\sum_{t=1}^na_{t-i,t-i',m_1,m_2}(\hat\theta_n)\right\|_2^2\\&\leq&\sum_{i,i'=1}^r\sum_{m_1,m_2=1}^{K(p+q)}\frac{1}{n}\sum_{t=1}^n\left\|
a_{t-i,t-i',m_1,m_2}(\hat\theta_n)\right\|_2^2,
\end{eqnarray*}
by Minkowski's inequality. Thanks to H\"{o}lder's inequality:
$$\left\|
a_{t-i,t-i',m_1,m_2}(\hat\theta_n)\right\|_2\leq\sum_{j=1}^4 {\cal A}^j_{t-i,t-i',m_1,m_2},\text{ with}$$
\begin{eqnarray*}
{\cal A}^1_{t-i,t-i',m_1,m_2}&=&\left\|\sup_{\theta\in\Theta}\left|e_{t-i}(\theta)-\epsilon_{t-i}(\theta)\right|\right\|_4
\sup_{t\geq0}\left\|\sup_{\theta\in\Theta}\left|e_{t}(\theta)\right|\right\|_{12}\left(\sup_{t\geq0}\left\|\sup_{\theta\in\Theta}\left\|\frac{\partial e_{t}(\theta)}{\partial \theta}\right\|\right\|_{12}\right)^2\\
{\cal A}^2_{t-i,t-i',m_1,m_2}&=&\left\|\sup_{\theta\in\Theta}\left|\epsilon_{t}(\theta)\right|\right\|_{12}\left\|\sup_{\theta\in\Theta}\left|e_{t-i'}(\theta)-\epsilon_{t-i'}(\theta)\right|\right\|_4
\left(\sup_{t\geq0}\left\|\sup_{\theta\in\Theta}\left\|\frac{\partial e_{t}(\theta)}{\partial \theta}\right\|\right\|_{12}\right)^2
\\
{\cal A}^3_{t-i,t-i',m_1,m_2}&=&\left(\left\|\sup_{\theta\in\Theta}\left|\epsilon_{t}(\theta)\right|\right\|_{16}\right)^2\left\|\sup_{\theta\in\Theta}\left\|\frac{\partial }{\partial \theta}\left(e_{t-i}(\theta)-\epsilon_{t-i}(\theta)\right)\right\|\right\|_{16/5}
\sup_{t\geq0}\left\|\sup_{\theta\in\Theta}\left\|\frac{\partial e_{t}(\theta)}{\partial \theta}\right\|\right\|_{16}
\\
{\cal A}^4_{t-i,t-i',m_1,m_2}&=&\left(\left\|\sup_{\theta\in\Theta}\left|\epsilon_{t}(\theta)\right|\right\|_{16}\right)^2\left\|\sup_{\theta\in\Theta}\left\|\frac{\partial \epsilon_{t}(\theta)}{\partial \theta}\right\|\right\|_{16}
\left\|\sup_{\theta\in\Theta}\left\|\frac{\partial }{\partial \theta}\left(e_{t-i'}(\theta)-\epsilon_{t-i'}(\theta)\right)\right\|\right\|_{16/5}.
\end{eqnarray*}
We deal with ${\cal A}^1_{t-i,t-i',m_1,m_2}$ and ${\cal A}^2_{t-i,t-i',m_1,m_2}$, as ${\cal A}^3_{t-i,t-i',m_1,m_2}$ and ${\cal A}^4_{t-i,t-i',m_1,m_2}$ are dealt with similarly. In view of Lemma~\ref{lemme_prelimbis},  we have
\begin{eqnarray*}
\frac{1}{n}\sum_{t=1}^n{\cal A}^1_{t-i,t-i',m_1,m_2}&\leq&\kappa_1\frac{1}{n}\sum_{t=1}^n\left\|\sup_{\theta\in\Theta}\left|e_{t-i}(\theta)-\epsilon_{t-i}(\theta)\right|\right\|_4
\\&\leq&\frac{\kappa_1}{n}\left(\sum_{t=1}^{n-r}\left\|\sup_{\theta\in\Theta}\left|e_{t}(\theta)-\epsilon_{t}(\theta)\right|\right\|_4+r\left\|\sup_{\theta\in\Theta}
\left|\epsilon_{0}(\theta)\right|\right\|_4\right)=\mathrm{O}\left(\frac{1}{n}+\frac{r}{n}\right)=\mathrm{O}\left(\frac{r}{n}\right),
\end{eqnarray*}
independent from $i$, $i'$, $m_1$ and $m_2$. Similarly, we have
\begin{eqnarray*}
\frac{1}{n}\sum_{t=1}^n{\cal A}^3_{t-i,t-i',m_1,m_2}&\leq&\kappa_3\frac{1}{n}\sum_{t=1}^n\left\|\sup_{\theta\in\Theta}\left\|\frac{\partial }{\partial \theta}\left(e_{t-i}(\theta)-\epsilon_{t-i}(\theta)\right)\right\|\right\|_{16/5}
\\&\leq&\kappa_3\frac{1}{n}\left(\sum_{t=1}^{n-r}\left\|\sup_{\theta\in\Theta}\left\|\frac{\partial }{\partial \theta}\left(e_{t}(\theta)-\epsilon_{t}(\theta)\right)\right\|\right\|_{16/5}\right.\\&&\left.+r\left\|\sup_{\theta\in\Theta}\left\|\frac{\partial \epsilon_{0}(\theta)}{\partial \theta}\right\|\right\|_{16/5}\right)=\mathrm{O}\left(\frac{1}{n}+\frac{r}{n}\right)=\mathrm{O}\left(\frac{r}{n}\right),
\end{eqnarray*}
because $\sum_{t=1}^{\infty}\left\|\sup_{\theta\in\Theta}\left\|{\partial\left(e_{t}(\theta)-\epsilon_{t}(\theta)\right) }/{\partial \theta}\right\|\right\|_{16/5}<\infty$ and $\left\|\sup_{\theta\in\Theta}\left\|{\partial \epsilon_{0}(\theta)}/{\partial \theta}\right\|\right\|_{16/5}<\infty$ (see Lemma~\ref{lemme_prelimbis}, Point 4).
Gathering ${\cal A}^1_{t-i,t-i',m_1,m_2}$, ${\cal A}^2_{t-i,t-i',m_1,m_2}$,  ${\cal A}^3_{t-i,t-i',m_1,m_2}$ and ${\cal A}^4_{t-i,t-i',m_1,m_2}$, we arrive at
\begin{eqnarray*}
\mathbb{E}\| \hat{\Sigma}_{{\hat{\underline{\Upsilon}}_r}}-\hat{\Sigma}_{\underline{\Upsilon}_{r,n}}\|^2&\leq&\sum_{i,i'=1}^r\sum_{m_1,m_2=1}^{K(p+q)}\left(
\frac{1}{n}\sum_{t=1}^n\sum_{j=1}^4 {\cal A}^j_{t-i,t-i',m_1,m_2}\right)^2=\mathrm{O}\left(r^2\left\{\frac{r}{n}\right\}^2\right)=\mathrm{O}\left(\frac{r^4}{n^2}\right).
\end{eqnarray*}
We thus deduce that
\begin{equation}
\label{resinterm1}\sqrt{r}\|\hat{\Sigma}_{{\hat{\underline{\Upsilon}}_r}}-
\hat{\Sigma}_{\underline{\Upsilon}_{r,n}}\|=\mathrm{o}_{\mathbb{P}}(1),\text{ when }r=r(n)=\mathrm{o}\left(n^{2/5}\right).
\end{equation}
We now prove that $$\sqrt{r}\|\hat{\Sigma}_{\underline{\Upsilon}_{r,n}}-\hat{\Sigma}_{{{\underline{\Upsilon}}_r}}\|=\mathrm{o}_{\mathbb{P}}(1),\text{ when }r=r(n)=\mathrm{o}\left(n^{1/3}\right).$$

Taylor expansions around $\theta_0$ yield
\begin{equation}
\label{taylor-res}
\left|\epsilon_t(\hat\theta_n)-{\epsilon}_t(\theta_0)\right|\leq
r_t\left\|\hat\theta_n-\theta_0\right\|,\quad \left|\frac{\partial \epsilon_t(\hat\theta_n)}{\partial \theta_m}- \frac{\partial \epsilon_t(\theta_0)}{\partial \theta_m}\right|\leq
s_t(m)\left\|\hat\theta_n-\theta_0\right\|
\end{equation}
with
$r_t=\sup_{\theta\in\Theta}\left\|{\partial{\epsilon}_t({\theta})}/{\partial
\theta}\right\|$, $s_{t}(m)=\sup_{\theta\in\Theta}\left\|{\partial^2{\epsilon}_t({\theta})}/{\partial
\theta\partial \theta_m}\right\|$
where $m=m_1=m_2$. Define $Z_t$ as in the proof of Lemma~\ref{lem3}, and let $Z_{t,n}$ be obtained by replacing $\Upsilon_t(m)$ by
$\Upsilon_{t,n}(m)=\epsilon_t(\hat\theta_n)\partial \epsilon_t(\hat\theta_n)/\partial \theta_m$ in $Z_t$.
Using \eqref{taylor-res}, for $i,i'=1,\dots,r$ and $m_1,m_2=1,\dots,K(p+q)$, we have
 \begin{eqnarray}
\label{taylor-res-ecart}
\left|\epsilon_{t-i}(\hat\theta_n)
\epsilon_{t-i'}(\hat\theta_n)\frac{\partial \epsilon_{t-i}(\hat\theta_n)}{\partial \theta_{m_1}}\frac{\partial \epsilon_{t-i'}(\hat\theta_n)}{\partial \theta_{m_2}}
-
\epsilon_{t-i}(\theta_0)
\epsilon_{t-i'}(\theta_0)\frac{\partial \epsilon_{t-i}(\theta_0)}{\partial \theta_{m_1}}\frac{\partial \epsilon_{t-i'}(\theta_0)}{\partial \theta_{m_2}}\right|
\leq\sum_{j=1}^4 {\cal B}^j_{t-i,t-i',m_1,m_2},
\end{eqnarray}
with
\begin{eqnarray*}
{\cal B}^1_{t-i,t-i',m_1,m_2}&=&r_{t-i}\left\|\hat\theta_n-\theta_0\right\|\sup_{\theta\in\Theta}\left|\epsilon_{t-i'}(\theta)\right|
\sup_{\theta\in\Theta}\left|\frac{\partial \epsilon_{t-i}(\theta)}{\partial \theta_{m_1}}\right|
\sup_{\theta\in\Theta}\left|\frac{\partial \epsilon_{t-i'}(\theta)}{\partial \theta_{m_2}}\right|\\
{\cal B}^2_{t-i,t-i',m_1,m_2}&=&r_{t-i'}\left\|\hat\theta_n-\theta_0\right\|\sup_{\theta\in\Theta}\left|\epsilon_{t-i}(\theta)\right|
\sup_{\theta\in\Theta}\left|\frac{\partial \epsilon_{t-i}(\theta)}{\partial \theta_{m_1}}\right|
\sup_{\theta\in\Theta}\left|\frac{\partial \epsilon_{t-i'}(\theta)}{\partial \theta_{m_2}}\right|
\\
{\cal B}^3_{t-i,t-i',m_1,m_2}&=&s_{t-i}(m_1)\left\|\hat\theta_n-\theta_0\right\|\sup_{\theta\in\Theta}\left|\epsilon_{t-i}(\theta)\right|
\sup_{\theta\in\Theta}\left|\epsilon_{t-i'}(\theta)\right|
\sup_{\theta\in\Theta}\left|\frac{\partial \epsilon_{t-i'}(\theta)}{\partial \theta_{m_2}}\right|
\\
{\cal B}^4_{t-i,t-i',m_1,m_2}&=&s_{t-i'}(m_2)\left\|\hat\theta_n-\theta_0\right\|\sup_{\theta\in\Theta}\left|\epsilon_{t-i}(\theta)\right|
\sup_{\theta\in\Theta}\left|\epsilon_{t-i'}(\theta)\right|
\sup_{\theta\in\Theta}\left|\frac{\partial \epsilon_{t-i}(\theta)}{\partial \theta_{m_1}}\right|.
\end{eqnarray*}
We deal with ${\cal B}^1_{t-i,t-i',m_1,m_2}$ and ${\cal B}^2_{t-i,t-i',m_1,m_2}$, as ${\cal B}^3_{t-i,t-i',m_1,m_2}$ and ${\cal B}^4_{t-i,t-i',m_1,m_2}$ are dealt with similarly.
We note first that, for all $i=1,\dots,r$,
\begin{eqnarray}\nonumber
\frac{1}{n}\sum_{t=1}^n\sup_{\theta\in\Theta}\left|\epsilon_{t-i}(\theta)\right|^4&=&\frac{1}{n}\sum_{t=1-i}^{n-i}\sup_{\theta\in\Theta}\left|\epsilon_{t}(\theta)\right|^4
=\frac{1}{n}\sum_{t=1-i}^{0}\sup_{\theta\in\Theta}\left|\epsilon_{t}(\theta)\right|^4+\frac{1}{n}\sum_{t=1}^{n-i}\sup_{\theta\in\Theta}\left|\epsilon_{t}(\theta)\right|^4
\\ \nonumber&\leq&\frac{r}{n}\frac{1}{r}\sum_{t=1-r}^{0}\sup_{\theta\in\Theta}\left|\epsilon_{t}(\theta)\right|^4+\frac{1}{n}\sum_{t=1}^{n}\sup_{\theta\in\Theta}\left|\epsilon_{t}(\theta)\right|^4
\\&=&\left(\frac{r}{n}+1\right)\left(\left\|\sup_{\theta\in\Theta}\left|\epsilon_{0}(\theta)\right|\right\|^4_4+\mathrm{o}_{a.s.}(1)\right),\label{O-de-un-eps4}
\end{eqnarray}
by the ergodic theorem. Similarly to \eqref{O-de-un-eps4}, we have
\begin{eqnarray}
\frac{1}{n}\sum_{t=1}^n\sup_{\theta\in\Theta}\left|\frac{\partial \epsilon_{t-i}(\theta)}{\partial \theta_{m}}\right|^4\leq\left(\frac{r}{n}+1\right)\left(\left\|\sup_{\theta\in\Theta}\left|\frac{\partial \epsilon_{0}(\theta)}{\partial \theta_{m}}\right|\right\|^4_4+\mathrm{o}_{a.s.}(1)\right).\label{O-de-un-grad.eps4}
\end{eqnarray}
By the Cauchy-Schwarz inequality and using \eqref{O-de-un-eps4} and \eqref{O-de-un-grad.eps4}, we have
\begin{eqnarray*}
\sum_{i,i'=1}^r\sum_{m_1,m_2=1}^{K(p+q)}\frac{1}{n}\sum_{t=1}^n{\cal B}^1_{t-i,t-i',m_1,m_2}&\leq & r^2\left\|\hat\theta_n-\theta_0\right\|\left(\frac{r}{n}+1\right)^3\left(\kappa_1+\mathrm{o}_{a.s.}(1)\right)
\\&=&r^2\left\|\hat\theta_n-\theta_0\right\|\mathrm{O}(1)\left(\kappa_1+\mathrm{o}_{a.s.}(1)\right),
\end{eqnarray*}
when $r=\mathrm{o}\left(n^{1/3}\right)$ and for some constant $\kappa_1>0$. Similar inequalities
hold for ${\cal B}^j_{t-i,t-i',m_1,m_2}$, for $j=2, 3, 4$. We thus deduce from \eqref{proprinorm} and \eqref{taylor-res-ecart} that
\begin{eqnarray}\label{presquefini}
r\|\hat{\Sigma}_{\underline{\Upsilon}_{r,n}}-\hat{\Sigma}_{{{\underline{\Upsilon}}_r}}\|^2&\leq &r^3\left\|\hat\theta_n-\theta_0\right\|^2\mathrm{O}_{\mathbb{P}}(1).
\end{eqnarray}

Since $\sqrt{n}\left(\hat\theta_n-\theta_0\right)$ converges in distribution, a tightness argument yields
$\left\|\hat\theta_n-\theta_0\right\|=\mathrm{O}_{\mathbb{P}}\left(n^{-1/2}\right)$ and hence from \eqref{presquefini}, we obtain for $r=\mathrm{o}(n^{1/3})$
\begin{equation}
\label{resinterm2}\sqrt{r}\|\hat{\Sigma}_{\underline{\Upsilon}_{r,n}}-
\hat{\Sigma}_{\underline{\Upsilon}_r}\|=\mathrm{o}_{\mathbb{P}}(1).\end{equation}
By Lemma~\ref{lem3} , (\ref{resinterm1}) and (\ref{resinterm2}) show that $\sqrt{r}\|\hat{\Sigma}_{{\hat{\underline{\Upsilon}}_r}}-
{\Sigma}_{\underline{\Upsilon}_r}\|=\mathrm{o}_{\mathbb{P}}(1)$. The other results are obtained similarly.
\zak

Write $\underline{\mathbf{\Phi}}_{r}^*=\left(\Phi_{1}\cdots \Phi_{r}\right)$ where
the $\Phi_{i}$'s are defined by (\ref{arinfty}).
\begin{lemme}\label{lem5}
Under the assumptions of Theorem~\ref{estimationI},
\begin{equation*}
\sqrt{r}\left\|\underline{\mathbf{\Phi}}_{r}^*-\underline{\mathbf{\Phi}}_{r}\right\|\to 0,
\end{equation*}
as $r\to\infty$.
\end{lemme}
{\bf Proof.} Recall that by (\ref{arinfty}) and (\ref{regp})
\begin{eqnarray*}  \Upsilon_t&=&\underline{\mathbf{\Phi}}_{r}\underline{\Upsilon}_{r,t}+u_{r,t}
=\underline{\mathbf{\Phi}}_{r}^*\underline{\Upsilon}_{r,t}+\sum_{i=r+1}^\infty
\Phi_i{\Upsilon}_{t-i}+u_t
:=\underline{\mathbf{\Phi}}_{r}^*\underline{\Upsilon}_{r,t}+u_{r,t}^*.
\end{eqnarray*}
Hence, using the orthogonality conditions in (\ref{arinfty}) and
(\ref{regp})
\begin{eqnarray}\label{r1}
\underline{\mathbf{\Phi}}_{r}^*-\underline{\mathbf{\Phi}}_{r}&=&
-{\Sigma}_{u_r^*,\underline{\Upsilon}_{r}}
{\Sigma}_{\underline{\Upsilon}_{r}}^{-1}
\end{eqnarray}
where
${\Sigma}_{u_r^*,\underline{\Upsilon}_{r}}=\mathbb{E}u_{r,t}^*\underline{\Upsilon}_{r,t}'$.
Using arguments and notations of the proof of Lemma~\ref{lemsurprise!}, there exists a constant $C_2$ independent of $s$ and $m_1,m_2$ such that
$$\mathbb{E}\left|{\Upsilon}_{1}(m_1){\Upsilon}_{1+s}(m_2)\right|\leq C_1\sum_{h_1,h_2=0}^{\infty}\rho^{h_1+h_2}\|\epsilon_1\|_{4}^4\leq C_2.$$
By the Cauchy-Schwarz inequality and (\ref{proprinorm}), we then have
$$\left\|\mbox{Cov}\left({\Upsilon}_{t-r-h},\underline{\Upsilon}_{r,t}\right)\right\|
\leq C_2r^{1/2}K(p+q).$$
Thus,
\begin{eqnarray}
\nonumber \|{\Sigma}_{u_r^*,\underline{\Upsilon}_{r}}\|&=&
\|\sum_{i=r+1}^\infty
\Phi_i\mathbb{E}{\Upsilon}_{t-i}\underline{\Upsilon}_{r,t}'\|\leq
\sum_{h=1}^\infty \|\Phi_{r+h}\|\left\|
\mbox{Cov}\left({\Upsilon}_{t-r-h},\underline{\Upsilon}_{r,t}\right)\right\|\\
&=&\mathrm O(1)r^{1/2}\sum_{h=1}^\infty \|\Phi_{r+h}\|.
\label{r2}\end{eqnarray} Note that the assumption
$\|\Phi_i\|=\mathrm o\left(i^{-2}\right)$ entails $ r\sum_{h=1}^\infty
\|\Phi_{r+h}\|=\mathrm o(1)$ as $r\to\infty$. The lemma therefore follows
from (\ref{r1}), (\ref{r2}) and Lemma \ref{lem4}. \zak

The following lemma is similar to Lemma 3 in \cite{berk}.
\begin{lemme}\label{lem6}
Under the assumptions of Theorem~\ref{estimationI},
\begin{eqnarray*}
\sqrt{r}\|\hat{\Sigma}_{\underline{\hat{\Upsilon}}_{r}}^{-1}-
{\Sigma}_{\underline{\Upsilon}_{r}}^{-1}\|
&=&\mathrm o_\mathbb P(1)
\end{eqnarray*}
as $n\to\infty$ when $r=\mathrm o(n^{1/3})$ and $r\to\infty$.
\end{lemme}
{\bf Proof.}
We have
\begin{eqnarray}
\left\|\hat{\Sigma}_{\underline{\hat{\Upsilon}}_{r}}^{-1}-
{\Sigma}_{\underline{\Upsilon}_{r}}^{-1}\right\|
\nonumber
&=&\left\|\left\{\hat{\Sigma}_{\underline{\hat{\Upsilon}}_{r}}^{-1}-
{\Sigma}_{\underline{\Upsilon}_{r}}^{-1}+
{\Sigma}_{\underline{\Upsilon}_{r}}^{-1}\right\}\left\{
{\Sigma}_{\underline{\Upsilon}_{r}}-
\hat{\Sigma}_{\underline{\hat{\Upsilon}}_{r}}\right\}
{\Sigma}_{\underline{\Upsilon}_{r}}^{-1}\right\|\\
&\leq&\left(\left\|\hat{\Sigma}_{\underline{\hat{\Upsilon}}_{r}}^{-1}-
{\Sigma}_{\underline{\Upsilon}_{r}}^{-1}\right\|+\left\|{\Sigma}_{\underline{\Upsilon}_{r}}^{-1}\right\|\right)
\left\|\hat{\Sigma}_{\underline{\hat{\Upsilon}}_{r}}-
{\Sigma}_{\underline{\Upsilon}_{r}}\right\|
\left\|{\Sigma}_{\underline{\Upsilon}_{r}}^{-1}\right\|.
\nonumber
\end{eqnarray}
Iterating this inequality, we obtain
\begin{eqnarray*}
\left\|\hat{\Sigma}_{\underline{\hat{\Upsilon}}_{r}}^{-1}-
{\Sigma}_{\underline{\Upsilon}_{r}}^{-1}\right\|
&\leq&\left\|{\Sigma}_{\underline{\Upsilon}_{r}}^{-1}\right\|
\sum_{i=1}^{\infty}\left\|\hat{\Sigma}_{\underline{\hat{\Upsilon}}_{r}}-
{\Sigma}_{\underline{\Upsilon}_{r}}\right\|^i
\left\|{\Sigma}_{\underline{\Upsilon}_{r}}^{-1}\right\|^i.
\nonumber
\end{eqnarray*}
Thus, for every $\varepsilon>0$,
\begin{eqnarray*}
&&\mathbb P\left(\sqrt{r}\left\|\hat{\Sigma}_{\underline{\hat{\Upsilon}}_{r}}^{-1}-
{\Sigma}_{\underline{\Upsilon}_{r}}^{-1}\right\|>\varepsilon\right)
\\&\leq&
\mathbb P\left(\sqrt{r}\frac{\left\|{\Sigma}_{\underline{\Upsilon}_{r}}^{-1}
\right\|^2
\left\|\hat{\Sigma}_{\underline{\hat{\Upsilon}}_{r}}-
{\Sigma}_{\underline{\Upsilon}_{r}}\right\|
}{1-
\left\|\hat{\Sigma}_{\underline{\hat{\Upsilon}}_{r}}-
{\Sigma}_{\underline{\Upsilon}_{r}}\right\|
\left\|{\Sigma}_{\underline{\Upsilon}_{r}}^{-1}\right\|}
>\varepsilon
\mbox{ and }\left\|\hat{\Sigma}_{\underline{\hat{\Upsilon}}_{r}}-
{\Sigma}_{\underline{\Upsilon}_{r}}\right\|
\left\|{\Sigma}_{\underline{\Upsilon}_{r}}^{-1}\right\|<1\right)\\
&&+\mathbb P\left(\sqrt{r}\left\|\hat{\Sigma}_{\underline{\hat{\Upsilon}}_{r}}-
{\Sigma}_{\underline{\Upsilon}_{r}}\right\|
\left\|{\Sigma}_{\underline{\Upsilon}_{r}}^{-1}\right\|\geq
1\right)\\
&\leq&
\mathbb P\left(
\sqrt{r}\left\|\hat{\Sigma}_{\underline{\hat{\Upsilon}}_{r}}-
{\Sigma}_{\underline{\Upsilon}_{r}}\right\|
>\frac{\varepsilon}{
\left\|{\Sigma}_{\underline{\Upsilon}_{r}}^{-1}\right\|^2+\varepsilon
r^{-1/2}\left\|{\Sigma}_{\underline{\Upsilon}_{r}}^{-1}\right\|}
\right)\\
&&+\mathbb P\left(\sqrt{r}\left\|\hat{\Sigma}_{\underline{\hat{\Upsilon}}_{r}}-
{\Sigma}_{\underline{\Upsilon}_{r}}\right\|
\geq
\left\|{\Sigma}_{\underline{\Upsilon}_{r}}^{-1}\right\|^{-1}\right)=\mathrm o(1)
\end{eqnarray*}
by Lemmas \ref{lem3} and \ref{lem4}. This establishes Lemma \ref{lem6}.
\zak

\begin{lemme}\label{lem7}
Under the assumptions of Theorem~\ref{estimationI},
$$
\sqrt{r}\left\|\underline{\hat{\mathbf{\Phi}}}_{r}-\underline{\mathbf{\Phi}}_{r}\right\|=\mathrm o_\mathbb P(1)$$
as $r\to\infty$ and $r=\mathrm o(n^{1/3})$.
\end{lemme}
{\bf Proof.}
By the triangle
inequality and Lemmas \ref{lem4} and \ref{lem6}, we have
\begin{equation}\label{r5}
\left\|\hat{\Sigma}_{\underline{\hat{\Upsilon}}_{r}}^{-1}\right\|
\leq
\left\|\hat{\Sigma}_{\underline{\hat{\Upsilon}}_{r}}^{-1}-
{\Sigma}_{\underline{\Upsilon}_{r}}^{-1}\right\|+
\left\|
{\Sigma}_{\underline{\Upsilon}_{r}}^{-1}\right\|=\mathrm O_\mathbb P(1).
\end{equation}
Note that the orthogonality conditions in (\ref{regp}) entail that
 $\underline{\mathbf{\Phi}}_{r}={\Sigma}_{{\Upsilon},\underline{\Upsilon}_{r}}
{\Sigma}_{\underline{\Upsilon}_{r}}^{-1}$.
By Lemmas \ref{lem4}, \ref{lem3}, \ref{lem6}, and (\ref{r5}), we then have
\begin{eqnarray*}\nonumber
&&\sqrt{r}\left\|\underline{\hat{\mathbf{\Phi}}}_{r}-\underline{\mathbf{\Phi}}_{r}\right\|
=\sqrt{r}\left\|\hat{\Sigma}_{\hat{\Upsilon},\underline{\hat{\Upsilon}}_{r}}
\hat{\Sigma}_{\underline{\hat{\Upsilon}}_{r}}^{-1}
-{\Sigma}_{{\Upsilon},\underline{\Upsilon}_{r}}
{\Sigma}_{\underline{\Upsilon}_{r}}^{-1}\right\|\\
&=&
\sqrt{r}\left\|\left(\hat{\Sigma}_{\hat{\Upsilon},\underline{\hat{\Upsilon}}_{r}}
-{\Sigma}_{{\Upsilon},\underline{\Upsilon}_{r}}\right)
\hat{\Sigma}_{\underline{\hat{\Upsilon}}_{r}}^{-1}
+{\Sigma}_{{\Upsilon},\underline{\Upsilon}_{r}}
\left(\hat{\Sigma}_{\underline{\hat{\Upsilon}}_{r}}^{-1}
-{\Sigma}_{\underline{\Upsilon}_{r}}^{-1}\right)\right\|=\mathrm o_\mathbb P(1).
\end{eqnarray*}\zak
\noindent{\bf Proof of Theorem~\ref{estimationI}.} In view of (\ref{Idensitespectrale}),
it suffices to show that
$\underline{\hat{\mathbf{\Phi}}}_r(1)\to \underline{{\mathbf{\Phi}}}(1)$ and
$\hat{\Sigma}_{u_r}\to {\Sigma}_{u}$ in
probability. Let the $r\times 1$ vector ${\bf 1}_r=(1,\dots,1)'$ and the
$r(p+q)K\times (p+q)K$ matrix ${\bf E}_r=\mathbb{I}_{(p+q)K}\otimes {\bf 1}_r$,
where $\otimes$ denotes the matrix Kronecker product and $\mathbb{I}_d$ the $d\times d$ identity matrix.
Using
(\ref{proprinorm}), and Lemmas \ref{lem5}, \ref{lem7}, we obtain
\begin{eqnarray*}\left\|\underline{\hat{\mathbf{\Phi}}}_r(1)-\underline{{\mathbf{\Phi}}}(1)\right\|&\leq &\left\|\sum_{i=1}^r\left(\hat{
\Phi}_{r,i}-\Phi_{r,i}\right)\right\|+\left\|\sum_{i=1}^r\left(
{\Phi}_{r,i}-\Phi_{i}\right)\right\|+\left\|\sum_{i=r+1}^{\infty}\Phi_{i}\right\|\\
&=&\left\|\left(\underline{\hat{\mathbf{\Phi}}}_{r}-\underline{{\mathbf{\Phi}}}_{r}\right){\bf E}_r\right\|
+\left\|\left(\underline{{\mathbf{\Phi}}}_{r}^*-\underline{{\mathbf{\Phi}}}_{r}\right){\bf E}_r\right\|
+\left\|\sum_{i=r+1}^{\infty}\Phi_{i}\right\|\\
&\leq&\sqrt{(p+q)K}\sqrt{r}\left\{\left\|\underline{\hat{\mathbf{\Phi}}}_{r}-\underline{{\mathbf{\Phi}}}_{r}\right\|
+\left\|\underline{{\mathbf{\Phi}}}_{r}^*-\underline{{\mathbf{\Phi}}}_{r}\right\|\right\}
+\left\|\sum_{i=r+1}^{\infty}\Phi_{i}\right\|\\
&=&\mathrm o_\mathbb P(1).\end{eqnarray*}
Now note that
$$\hat{\Sigma}_{u_r}=\hat{\Sigma}_{{\hat{\Upsilon}}}
-\underline{\hat{\mathbf{\Phi}}}_{r}\hat{\Sigma}_{{\hat{\Upsilon}},\underline{\hat{\Upsilon}}_r}'$$
and, by (\ref{arinfty})
\begin{eqnarray*}
{\Sigma}_{u}&=&\mathbb Eu_tu_t'=\mathbb Eu_t\Upsilon_t'=\mathbb E\left\{\left(\Upsilon_t-\sum_{i=1}^{\infty}\Phi_i\Upsilon_{t-i}\right)
\Upsilon_t'\right\}\\&=&
{\Sigma}_{{{\Upsilon}}}-\sum_{i=1}^{\infty}\Phi_i\mathbb E{\Upsilon}_{t-i}{\Upsilon}_{t}'
={\Sigma}_{\Upsilon}-\underline{{\mathbf{\Phi}}}_{r}^*{\Sigma}_{{{\Upsilon}},{\underline{\Upsilon}_{r}}}'
-\sum_{i=r+1}^{\infty}\Phi_i\mathbb E{\Upsilon}_{t-i}{\Upsilon}_{t}'.
\end{eqnarray*}
Thus,
\begin{eqnarray}
\left\|\hat{\Sigma}_{u_r}-{\Sigma}_{u}\right\|
\nonumber
&= &\left\|\hat{\Sigma}_{\hat{\Upsilon}}
-{\Sigma}_{\Upsilon}-
\left(\underline{\hat{\mathbf{\Phi}}}_{r}-\underline{{\mathbf{\Phi}}}_{r}^*\right)
\hat{\Sigma}_{{\hat{\Upsilon}},\underline{\hat{\Upsilon}}_r}'\right.\\
\nonumber
&&\left.
-
\underline{{\mathbf{\Phi}}}_{r}^*\left(\hat{\Sigma}_{{\hat{\Upsilon}},\underline{\hat{\Upsilon}}_r}'-{\Sigma}_{{{\Upsilon}},{\underline{\Upsilon}_{r}}}'\right)
+\sum_{i=r+1}^{\infty}\Phi_i\mathbb E{\Upsilon}_{t-i}{\Upsilon}_{t}'\right\|\\
\nonumber
&\leq&\left\|\hat{\Sigma}_{\hat{\Upsilon}}
-{\Sigma}_{\Upsilon}\right\|+\left\|\left(\underline{\hat{\mathbf{\Phi}}}_{r}-\underline{{\mathbf{\Phi}}}_{r}^*\right)
\left(\hat{\Sigma}_{{\hat{\Upsilon}},\underline{\hat{\Upsilon}}_r}'-{\Sigma}_{{{\Upsilon}},\underline{{\Upsilon}}_r}'\right)\right\|\\
&&+\left\|\left(\underline{\hat{\mathbf{\Phi}}}_{r}-\underline{{\mathbf{\Phi}}}_{r}^*\right){\Sigma}_{{{\Upsilon}},\underline{{\Upsilon}}_r}'\right\|
+\left\|\underline{{\mathbf{\Phi}}}_{r}^*\left(\hat{\Sigma}_{{\hat{\Upsilon}},\underline{\hat{\Upsilon}}_r}'-
{\Sigma}_{{{\Upsilon}},{\underline{\Upsilon}_{r}}}'\right)\right\|
\nonumber\\
&&
+\left\|\sum_{i=r+1}^{\infty}\Phi_i\mathbb E{\Upsilon}_{t-i}{\Upsilon}_{t}'\right\|.
\label{inegaga}
\end{eqnarray}
In the right-hand side of this inequality, the first norm
is $\mathrm o_\mathbb P(1)$ by Lemma \ref{lem3}. By Lemmas \ref{lem5} and
\ref{lem7}, we have
$\|\underline{\hat{\mathbf{\Phi}}}_{r}-\underline{{\mathbf{\Phi}}}_{r}^*\|=\mathrm o_\mathbb P(r^{-1/2})=\mathrm o_\mathbb P(1)$, and by Lemma \ref{lem3},
$\|\hat{\Sigma}_{{\hat{\Upsilon}},\underline{\hat{\Upsilon}}_r}'-{\Sigma}_{{{\Upsilon}},\underline{{\Upsilon}}_r}'\|=\mathrm o_\mathbb P(r^{-1/2})=\mathrm o_\mathbb P(1)$.
Therefore the second norm in the right-hand side of
(\ref{inegaga}) tends to zero in probability. The third norm tends
to zero in probability because
$\|\underline{\hat{\mathbf{\Phi}}}_{r}-\underline{{\mathbf{\Phi}}}_{r}^*\|=\mathrm o_\mathbb P(1)$ and, by
Lemma \ref{lem4},
$\|{\Sigma}_{{{\Upsilon}},\underline{{\Upsilon}}_r}'\|=\mathrm O(1)$. The
fourth norm tends to zero in probability because, in view of Lemma \ref{lem3},
$\|\hat{\Sigma}_{{\hat{\Upsilon}},\underline{\hat{\Upsilon}}_r}'-
{\Sigma}_{{{\Upsilon}},{\underline{\Upsilon}_{r}}}'\|=\mathrm o_\mathbb P(1)$,
and, in view of (\ref{proprinorm}),
$\|\underline{{\mathbf{\Phi}}}_{r}^*\|^2\leq \sum_{i=1}^\infty
\mbox{Tr}(\Phi_i\Phi_i')<\infty$.
Clearly, the last norm  tends to zero, which completes the
proof.\zak

{\bf Acknowledgements.} We sincerely thank the
anonymous reviewers and Editor in Chief for helpful remarks. The authors wish to acknowledge the support from the "S\'{e}ries temporelles et valeurs extr\^{e}mes : th\'{e}orie et applications en mod\'{e}lisation et estimation des risques" Projet R\'{e}gion grant No OPE-2017-0068.


\end{document}